\documentclass[10pt,leqno]{amsart}
\usepackage{a4wide}
\usepackage{amssymb,upgreek}
\usepackage{amsmath}
\usepackage{amsthm}
\usepackage{mathrsfs,fancybox,mdframed}
\usepackage{bm,euscript,csquotes,comment,appendix,multirow}
\usepackage{upgreek,csquotes}
\usepackage[curve]{xypic}
\usepackage[usenames,dvipsnames]{xcolor}
\usepackage{enotez}
\usepackage{calligra,yfonts}
\DeclareMathOperator{\cHom}{\mathscr{H}\text{\kern -3pt {\calligra\large om}}\,}
\usepackage{multirow}
\usepackage{tikz-cd}
\usepackage{enumerate}

\newenvironment{psmallmat}
  {\left(\begin{smallmatrix}}
  {\end{smallmatrix}\right)}

\usepackage{hyperref}

\hypersetup{colorlinks,  citecolor=ForestGreen,  linkcolor=blue, urlcolor=black}

\theoremstyle{plain}

\newcommand{\id}{\operatorname{id}}

\newcommand{\pr}{\operatorname{pr}}
\newcommand{\ev}{\operatorname{ev}}
\newcommand{\Sh}{\operatorname{Sh}}
\newcommand{\res}{\operatorname{res}}

\newcommand{\Hom}{\operatorname{Hom}}
\newcommand{\End}{\operatorname{End}}
\newcommand{\Ext}{\operatorname{Ext}}

\newcommand{\Mod}{\operatorname{Mod}}

\newcommand{\ind}{\operatorname{ind}}

\newcommand{\triv}{\operatorname{triv}}
\newcommand{\bbF}{\mathbb{F}}
\newcommand{\bbQ}{\mathbb{Q}}
\newcommand{\bbZ}{\mathbb{Z}}
\newcommand{\aff}{\operatorname{aff}}
\newcommand{\Ad}{\textnormal{Ad}}

\newcommand{\sm}[4]{\left(\begin{smallmatrix} #1 & #2 \\ #3 & #4 \end{smallmatrix}\right)}

\newcommand{\chara}{\operatorname{char}}
\newcommand{\cores}{\operatorname{cores}}
\newcommand{\val}{\operatorname{val}}

\def\c{\mathbf{ c}}
\def\x{{\mathbf{x}}}

\def\D{\EuScript{D}}
\def\P{\EuScript{P}}
\def\DPP{{{}_{\P}\D_\P}}

\def\U{\EuScript{U}}
\def\V{\EuScript{V}}
\def\W{\EuScript{W}}

\def\X{{\mathbf X}}

\def\anti{\EuScript J}
\def\trace{\EuScript{S}}
\def\chara{{\rm char}}
\def\Oo{\mathfrak{O}}
\def\Aa{\mathscr{A}}
\def\Ker{{\rm Ker}}

\def\WP{{{\mathfrak W}_{\EuScript P}}}

\newcommand*{\longhookrightarrow}{\ensuremath{\lhook\joinrel\relbar\joinrel\rightarrow}}
\newcommand*{\longtwoheadrightarrow}{\ensuremath{\relbar\joinrel\twoheadrightarrow}}

\def\lp{\langle}
\def\rp{\rangle}

\newtheorem{theorem}{Theorem}[section]
\newtheorem{corollary}[theorem]{Corollary}
\newtheorem{lemma}[theorem]{Lemma}
\newtheorem{sublemma}[theorem]{Sublemma}
\newtheorem{proposition}[theorem]{Proposition}

\newtheorem*{theorem*}{Theorem}
\newtheorem*{proposition*}{Proposition}

\theoremstyle{remark}

\newtheorem{example}[theorem]{Example}
\newtheorem{definition}[theorem]{Definition}
\newtheorem{notn}[theorem]{Notation}

\theoremstyle{definition}

\newtheorem{remark}[theorem]{Remark}

\title{Parahoric Hecke Ext-algebras in characteristic $p$}
\author{}
\date{\today}

\author{Karol Kozio{\l}}
\address{Department of Mathematics, CUNY Baruch College, 55 Lexington Ave, New York, NY 10010 USA} 
\email{karol.koziol@baruch.cuny.edu}

\author{Rachel Ollivier}
\address{University of British Columbia,  1984 Mathematics Road, Vancouver, BC V6T 1Z2, Canada}
\email{ollivier@math.ubc.ca}

\author{Jacob Stockton}
\address{University of British Columbia,  1984 Mathematics Road, Vancouver, BC V6T 1Z2, Canada}
\email{jstockton@math.ubc.ca}


\makeatletter
\def\@tocline#1#2#3#4#5#6#7{\relax
  \ifnum #1>\c@tocdepth 
  \else
    \par \addpenalty\@secpenalty\addvspace{#2}%
    \begingroup \hyphenpenalty\@M
    \@ifempty{#4}{%
      \@tempdima\csname r@tocindent\number#1\endcsname\relax
    }{%
      \@tempdima#4\relax
    }%
    \parindent\z@ \leftskip#3\relax \advance\leftskip\@tempdima\relax
    \rightskip\@pnumwidth plus4em \parfillskip-\@pnumwidth
    #5\leavevmode\hskip-\@tempdima
      \ifcase #1
       \or\or \hskip 1em \or \hskip 2em \else \hskip 3em \fi%
      #6\nobreak\relax
    \dotfill\hbox to\@pnumwidth{\@tocpagenum{#7}}\par
    \nobreak
    \endgroup
  \fi}
\makeatother

\begin{document}

\maketitle

\begin{abstract}
Let $\mathfrak{F}$ be a nonarchimedean local field of residual characteristic $p$, and let $G$ denote the group of $\mathfrak{F}$-points of a connected reductive group over $\mathfrak{F}$.  For an open compact subgroup $\EuScript{U}$ of $G$ and a unital commutative ring $k$, we let $\mathbf{X}_{\EuScript{U}}$ denote the space of compactly supported $k$-valued functions on $G/\EuScript{U}$.  Building on work of Ollivier--Schneider, we investigate the graded $\textnormal{Ext}$-algebra $E_{\EuScript{U}}^* := \textnormal{Ext}_G^*(\mathbf{X}_{\EuScript{U}},\mathbf{X}_{\EuScript{U}})^{\textrm{op}}$.  In particular, we describe the Yoneda product, an involutive anti-automorphism, and (when $k$ is a field of characteristic $p$ and $\EuScript{U}$ has no $p$-torsion) a duality operation.  We allow for the reductive group to be non-split, and for the open compact subgroup $\EuScript{U}$ to be non-pro-$p$.

Specializing further to the case $G = \textrm{SL}_2(\mathbb{Q}_p)$ with $p \geq 5$ and a coefficient field of characteristic $p$, we obtain more precise results when $\EuScript{U}$ is equal to an Iwahori subgroup $J$ or a hyperspecial maximal compact subgroup $K$.  In particular, we compute the structure of $E_J^*$ as an $E_J^0$-bimodule, obtain an explicit description of the center $\mathcal{Z}(E_J^*)$ of $E_J^*$, and construct a surjective morphism of algebras $\mathcal{Z}(E_J^*) \longrightarrow E_K^*$ (analogous to the compatibility between Bernstein and Satake isomorphisms in characteristic 0).  From this we deduce the (somewhat surprising) fact that $E_K^*$ is not graded-commutative, contrary to what happens for almost all $\ell$-modular characteristics.  
\end{abstract}

\setcounter{tocdepth}{4}

\tableofcontents

\section{Introduction}

In this introduction,  $\mathfrak{F}$ denotes a finite extension of $\mathbb Q_p$ with ring of integers $\mathfrak O$ and residue field $\mathbb F_q$, and $G$ the group of $\mathfrak{F}$-rational points of a connected reductive group $\mathbf{G}$ over $\mathfrak{F}$.  Given an open compact subgroup $\U$ of $G$ and a ring of coefficients $k$, we consider the Hecke $\Ext$-algebra
\[ E^*_\U:= \Ext^*_G(\X_\U, \X_\U)^{\textnormal{op}} := \left(\bigoplus_{i \in \bbZ} \Ext^i_G(\X_\U, \X_\U)\right)^{\textnormal{op}}\ \]
where $\X_\U$ denotes the compact induction of the trivial $k$-character of $\U$.  This algebra, which is also known in the literature as  a \emph{local derived Hecke algebra},  has appeared in recent years in  connection to the Langlands program.  Let us highlight some of its features:

\begin{enumerate}
\item In the work of Schneider (\cite{SDGA}), the  ring  $k$ is a field of characteristic $p$ and $\U$ is a  pro-$p$ Iwahori subgroup $I$ of $G$. The $\Ext$-algebra $E^*_I$ is then the cohomology algebra of a natural differential graded algebra (the \emph{Hecke DGA}) whose category of DG modules is shown (under certain conditions) to be equivalent to the derived category of smooth representations of $G$ on $k$-vector spaces.   When $\mathbf{G}$  is split over $\mathfrak{F}$, a general framework to study $E_I^*$ was developed in \cite{Ext}. Explicit computations when $G={\rm SL}_2(\mathbb Q_p)$, $p\geq 5$ can be found in \cite{OS2}, \cite{OS4} and \cite{bodon:thesis}, together with applications to the $I$-cohomology of smooth irreducible $k$-representations of $G$. These computations are also a crucial ingredient for the construction in \cite{ArSch} of localizing stable  subcategories in the  category of smooth $k$-representations  of ${\rm SL}_2(\mathbb Q_p)$.

\item \label{intro-2}In the work of Venkatesh (\cite{Ven}), the ring $k$  is $\mathbb Z/\ell^r\mathbb Z$ where $q\equiv1\bmod \ell^r$ and  $\U$ is  a maximal compact subgroup $K$. The \emph{degree increasing} action of the spherical Hecke $\Ext$-algebra  $E^*_K $ on the cohomology of a locally symmetric space is studied  in order to explain why certain Hecke eigensystems can occur in several different cohomological degrees. This action is related to certain motivic cohomology groups (see \cite{Ven}, \cite{PVen}).  Furthermore, when $\mathbf{G}$  is split over $\mathfrak{F}$ and $\ell$ does not divide the order of the Weyl group, there is a derived Satake isomorphism between $E_K^*$ and the Weyl-invariants of the spherical $\Ext$-algebra of the torus (\cite[\S3]{Ven}). This ensures that $E_K^*$ is graded-commutative.

Changing the level of the arithmetic manifold leads to considering the $\Ext$-algebra attached to other open compact subgroups: in the above context, the case when  $\U$ is  an Iwahori subgroup $J$ is also considered and some features of the connection between $E_K^*$ and $E_J^*$ are described in \cite[\S 4]{Ven}.

Lastly,  it is proved in \cite{Gehr} that $E_K^*$ is graded-commutative under the weaker assumption that $q\equiv 1\bmod \ell$ (and $\ell$ does not divide the order of the Weyl group). This uses a derived version of the Gelfand trick associated to the pair $(G,K)$.  
\end{enumerate}

Returning to the general setting given at the beginning of this introduction, the goal of this article is to examine structural properties of the algebras $E_\U^*$ for arbitrary open compact subgroups $\U$.  In order to do this, we examine the interplay between $E^*_\U$ and $E^*_\V$, where $\U\subseteq\V$ are two open compact subgroups of $G$.  In particular, we consider the natural $G$-equivariant maps $\iota_{\V, \U}: \X_\V  \longhookrightarrow  \X_\U$ and $ \pi_{\U, \V}:\X_\U  \longtwoheadrightarrow  \X_\V $
mapping respectively the characteristic function of $\V$ to itself (viewed as an element of $\X_\U$), and the characteristic function of $\U$ to the characteristic function of $\V$.  By suitable Yoneda compositions,  $\iota_{\V, \U}$  and $ \pi_{\U, \V}$  give rise to maps on cohomology
\begin{equation}\label{compaintro}
R_{\V,\U}^*\::\:  E_\V^*\longrightarrow  E^*_\U\qquad \text{ and }\qquad C_{\U,\V}^*\::\:  E^*_\U\longrightarrow E_\V^* \ .
\end{equation}

We study in Subsection \ref{subsec:Ext} various features of $E^*_\U$ (cup product, explicit description of the Yoneda product, involution).  Imposing the additional requirement that $k$ is a field of characteristic $p$ (rather than just a coefficient ring), we show in Subsection \ref{subsec:Ext-padic} how to extend various duality properties of Poincar\'e groups to general open compact subgroups $\U$, which may not be pro-$p$.  This requires some elements of Bruhat--Tits theory (namely, Moy--Prasad filtrations), which we provide in Appendix \ref{app:orientation} joint with David Schwein.  This allows us to deduce duality properties and results on the structure of the top cohomology space of $E_\U^*$.  We note that the features of $E_\U^*$ described in Subsections \ref{subsec:Ext} and \ref{subsec:Ext-padic} were described in \cite{Ext} for $E_I^*$ when $\mathbf G$ is $\mathfrak F$-split.  Proceeding further, we  establish  compatibility statements which  show how these features
 transfer back and forth between $E_\U^*$ and $E_\V^*$  via the maps $R_{\V,\U}^*$ and $C_{\U,\V}^*$ (see Subsection \ref{subsubsec:compa} and Proposition \ref{lemma:projfor}).  We mention the following example which justifies the chosen notation  for the  maps: we have cup products on  $E_\U^*$ and  on $E_\V^*$, which are compatible with each other in the sense that we have the following projection formula:
\[ C_{\U,\V}^{i+j}(A\cup R^{ j}_{\V,\U}(B))=C_{\U,\V}^i(A)\cup B \]
(here $A\in E^i_\U$ and $B\in  E^j_\V$).

We are particularly interested in the case when $\U$ and $\V$ are parahoric subgroups, which we explore when  $\mathbf G$ is $\mathfrak F$-split in Subsection \ref{sec:parah}.  Specializing the setting even further, we apply these results in Section \ref{sec:SL2}
to the case when $G= {\rm SL}_2(\mathfrak F)$. We let $K$ denote the maximal compact subgroup ${\rm SL}_2(\mathfrak O)$ and $J\subseteq K$ the upper Iwahori subgroup with pro-$p$ Sylow subgroup $I$. Here the field of coefficients $k$ has characteristic $p$.
When $\mathfrak F = \mathbb Q_p$ and $p\geq 5$, we use the explicit description of $E_I^*$ in \cite{OS4} and our compatibility maps \eqref{compaintro} to deduce the following results
\begin{enumerate}[i.]
 \item the structure of $E_J^*$ as an $E_J^0$-bimodule as well as the description of the Yoneda product in  this algebra (Subsection \ref{subsubsec:EJ}), 	
 \item  the explicit description of the center $\mathcal Z(E_J^*)$ of $E_J^*$  (by which we do \emph{not} mean the graded center, see \eqref{dcenter}) and the fact that $E_J^*$ is finitely generated over $\mathcal Z(E_J^*)$,
\item \label{intro-iii}  the existence of a surjective  homomorphism of $k$-algebras $\mathcal Z(E_J^*) \longrightarrow E_K^*$ with  a kernel of dimension $2$ (Proposition \ref{prop:squeeze}) which moreover splits.
\end{enumerate}
Point \ref{intro-iii}  implies  in particular that when  $G={\rm SL}_2(\mathbb Q_p)$ and $p\geq 5$, the spherical Hecke $\Ext$-algebra is commutative. In fact, we find an element in $E^1_K$ whose square is nonzero in $E^2_K$ (Remark \ref{rema:square}), which implies $E_K^*$ is not graded-commutative, contrary to the case of $\bbZ/\ell^r\bbZ$-coefficients discussed above.  We do not know in general if $E_K^*$ is commutative. However, in Subsection \ref{subsec:nonzerosq} we generalize Remark \ref{rema:square} by direct calculation of Frattini quotients and Yoneda products in $E_K^*$ (that is to say, without relying on previous results from \cite{OS4}) to deduce:
\begin{enumerate}[i.]
\setcounter{enumi}{3}
 \item  when $G={\rm SL}_2(\mathfrak F)$  where $\mathfrak F/\mathbb Q_p$ is unramified with $p$ odd and $q \neq 3$, the existence of an element
 in $E^1_K$ whose square is nonzero in $E^2_K$. In particular, $E_K^*$ is not graded-commutative.
\end{enumerate}

We return to the general context of the beginning of the introduction, and assume that $\mathbf G$ is $\mathfrak F$-split. If $k=\mathbb{C}$, the Satake isomorphism says  the spherical Hecke algebra  $E^0_K$ is  isomorphic to the algebra  of the Weyl-invariants $\mathbb{C}[X_*]^{W_0}$ of the cocharacter lattice of a split torus of $\mathbf G$. By a result of Bernstein, the latter is also isomorphic to the center $\mathcal{Z}(E_J^0)$ of the Iwahori--Hecke algebra. The Bernstein and Satake isomorphisms are compatible in the sense that there is a natural map inducing the isomorphism of algebras $\mathcal{Z}(E_J^0)\longrightarrow E^0_K$ 
 (\cite{Haines},  \cite{Dat}).  
 When $k$ has characteristic $p$,  there is also  a Satake isomorphism (\cite{Her}) and a natural isomorphism $\mathcal{Z}(E_J^0)\longrightarrow E^0_K$  (\cite{Ollcompa}). 
 In both cases, the definition of the map $\mathcal{Z}(E_J^0) \longrightarrow E^0_K$ is algebraically natural (it relies on the $C_{J,K}^0$ map introduced above). In contrast, the inverse map is more complicated and sometimes easier to describe geometrically. When $k=\mathbb C$, it can be realized  using nearby cycles of perverse sheaves on the Beilinson--Drinfeld Grassmanian (\cite{Gaitsgory}); when $k$ has characteristic $p$, this method has been adapted using mod $p$ perverse sheaves in \cite{Cass}.

The isomorphism between the  (usual) spherical Hecke algebra and the center of the (usual) Iwahori--Hecke  algebra is  what motivated us  to compare $E_K^*$ and $\mathcal Z(E_J^*)$. Contrary to the situation in item \eqref{intro-2} above,  there is no derived Satake isomorphism for $E_K^*$ when $k$  is a field of  characteristic $p$ (see \cite{ron}).
 However, in the case of ${\rm SL}_2(\mathbb Q_p)$, $p\geq 5$,  point \ref{intro-iii} above ensures the existence of an injective homomorphism of algebras $E_K^*\longrightarrow \mathcal{Z}(E_J^*)$ which goes in the same direction as the  \emph{more interesting} map $ E^0_K\longrightarrow \mathcal{Z}(E_J^0)$ in the picture above.

\subsection*{Acknowledgements}  We thank the anonymous referee for his/her thorough reading and helpful remarks.  During the preparation of this article, KK was supported by NSF grant DMS-2310225 and a PSC-CUNY Trad B award, and RO was supported by NSERC Discovery grant, the Fondation des Sciences Math\'ematiques de Paris and the Ecole Normale Sup\'erieure in Paris.

\section{Hecke $\Ext$-algebras}

\subsection{Hecke algebras relative to an arbitrary open compact subgroup}\label{subsec:Hecke}

\subsubsection{Generalities} 
\label{subsubsec:general}

Throughout this discussion, we fix a locally profinite group $G$ and a unital commutative ring $k$, which will serve as our ring of coefficients.  We let $\Mod(G)$ denote the abelian category of smooth $G$-representations over $k$.  For $\U$ an open compact subgroup of $G$, and $(\sigma, V_\sigma)$ a smooth representation of $\U$ over $k$, we consider the compact induction $\ind_\U^G(\sigma)$: it is the space of compactly supported functions 
 $f: G\longrightarrow  V_\sigma \text{ such that } f(gu)=\sigma(u^{-1})(f(g))$ for all $(u,g)\in \U\times G$. The action of $g\in G$ is given by $(g, f)\longmapsto f(g^{-1}{}_-)\ .$  In particular, taking $\sigma = 1$, the trivial $\U$-representation, gives the compact induction 
 \[\mathbf{X}_\U := \ind_\U^G(1),\] 
 which can be seen as the space $k[G/\U]$ of  compactly supported functions $G\longrightarrow k$ which are constant on the left cosets mod $\U$. It lies in $\Mod(G)$.

For $Y$ a compact subset of $G$ which is right invariant under $\U$,  the characteristic function $\chara_Y$ is an element of $\X_\U$. The Hecke $k$-algebra of $\U$ is defined as
\[H _\U:= \End_{k[G ]}(\mathbf{X}_\U)^{\mathrm{op}}\ .\]
We often will identify $H_\U$ with $\X_\U^\U $, the submodule of $\U $-fixed vectors, via the map $h  \longmapsto h(\mathrm{char}_\U)$.  (In fact, this isomorphism is equivariant for the right $H_\U$-action; see the next paragraph.)  The endomorphism algebra $H_\U$ is then seen as the convolution $k$-algebra $(k[\U\backslash G/\U], \cdot )$ where
\begin{equation}\label{convo} (\tau\cdot \tau')({}_-)=\sum_{x\in G/\U }\tau(x) \tau'(x^{-1}{}_-)=\sum_{\U \backslash G \ni x}\tau(_-x^{-1}) \tau'(x) \ ,\end{equation}
with $\tau,\tau' \in k[\U\backslash G/\U] \cong \X_\U^\U$.  
We will usually adhere to the following notational convention: whenever we view an element of $H_\U$ as an endomorphism we use the letter $h$, while when we view an element of $H_\U$ as a $\U$-bi-invariant function we use the letter $\tau$.

The space $\X_\U$ is tautologically a right $H_\U$-module: for $\phi\in \X_\U$ and $h\in H_\U$, the right action of $h$ on $\phi$ is simply given by $(\phi, h)\longmapsto h(\phi)$. Under the isomorphism $H_\U \cong \X_\U^\U \cong k[\U\backslash G/\U]$, this action translates into the following right action of $H_\U$ on $\X_\U=k[G/\U]$:
\[(\phi\cdot \tau')({}_-) := \sum_{x\in G/\U}\phi(x)  \tau'(x^{-1}{}_-)=\sum_{\U\backslash G \ni x} \phi(_-x^{-1})\tau'(x) \ .\]

We have the $\U$-equivariant decomposition
\begin{equation}\label{f:Xbruhat}
 \X_\U = \bigoplus_{g\in \U \backslash G/\U} \mathbf{X}_\U(g)  \quad\text{with}\quad  \mathbf{X}_\U(g) := \ind_\U^{\U g\U}(1) \ ,
\end{equation}
and where $\ind_{\U}^{\U g \U}(1)$ denotes the subspace of $\X_\U$ consisting of functions with support contained in $\U g\U$.  In particular, we have $\mathbf{X}(g)^\U = k \tau^\U_g$ where 
 \[\tau^\U_g := \mathrm{char}_{\U g \U} \ ,\] 
 and hence $H_\U \cong \X_\U^\U = \bigoplus_{g\in \U \backslash G/\U} k \tau^\U_g \ $ as a  $k$-module.
We use $h_g^\U:\X_\U \longrightarrow \X_\U$ to denote the associated $G$-equivariant morphism.  For $g\in G$, we will also use the notation 
\[\U_g:= g \U g^{-1}\cap \U \ .\] 
It obviously only depends on the left coset $g\U$.

More generally, if $\V$ is another open compact subgroup of $G$ and $\sigma$ is representation of $\U$, we may consider the smooth $\V$-representation $\ind_{\U}^{\V g \U}(\sigma)$ consisting of those functions in $\ind_{\U}^G(\sigma)$ with support contained in $\V g\U$.  We then have an isomorphism of $\V$-representations
\begin{eqnarray}
\label{f:doublecosetisom}
\rho_g: \ind_\U^{\V g \U}(\sigma) & \stackrel{\sim}{\longrightarrow} & \ind_{\V\cap g \U g^{-1}}^\V(g_*(\sigma|_{\U \cap g^{-1}\V g}))\\
 f & \longmapsto & \rho_g(f), \notag
\end{eqnarray}
 where $(\rho_g(f))(_-) =  f(_- g)$.  For a representation $\sigma'$ of $\U \cap g^{-1} \V g$, we use the notation $g_*\sigma'$ to denote the representation of $\V \cap g\U g^{-1}$ consisting of the same underlying module as $\sigma'$, but where the action of $v \in \V \cap g\U g^{-1}$ is given by $g^{-1}vg \in \U \cap g^{-1}\V g$.

The algebra $H_\U$ always possesses a distinguished character $\chi_{\triv}^\U:H_\U \longrightarrow k$, called the trivial character and defined by
\begin{equation}\label{trivU}\chi_{\triv}^\U(\tau^\U_g) = [ \U : \U_g] \ . \end{equation}
\begin{lemma}
  The above character is well-defined, that is, the relation $\chi_{\triv}^{\U}(\tau^{\U}_g)\chi_{\triv}^{\U}(\tau^{\U}_{g'}) = \chi_{\triv}^{\U}(\tau^{\U}_{g}\tau^{\U}_{g'})$ holds for all $g,g' \in G$.  
\end{lemma}

\begin{proof}   Suppose we are given a smooth representation $V$ of $G$.  By Frobenius reciprocity, the space $V^\U$ is naturally a left $H_\U$-module.  Explicitly, the action of $h\in H_\U$ on $v\in V^\U$ is given by
\[(h, v)\longmapsto  (\phi_v\circ h)(\chara_\U)\]
where $\phi_v: \X_\U \longrightarrow V$ is the $G$-equivariant map sending $\chara_\U$ to $v$.  (The morphism $\phi_v$ corresponds to $v$ under the isomorphisms $\Hom_G(\X_\U,V) \cong \Hom_\U(1,V) \cong V^\U$.)  In particular, if $\tau=\tau_g^\U$ is seen in $k[\U\backslash G/\U]$, it corresponds to the map $h_g^\U: \X_\U \longrightarrow \X_\U$ sending $\chara_\U$ to $\chara_{\U g \U}=\sum_{x\in \U/\U_g} xg\cdot \chara_\U$
 so 
 \[\tau^\U_g\cdot v = (\phi_v\circ h^\U_g)(\chara_\U)=\sum_{x\in \U/\U_g} xg\cdot v \ .\]

Suppose now that $V = k$ is the trivial representation of $G$. By the previous paragraph, the space $k = k^\U$ is naturally a left $H_\U$-module with action of $H_\U$ given by the character $\tau_g^\U\longmapsto [ \U:\U_g ]$.  This verifies the claim.
\end{proof}

\begin{remark}
	The above result is also discussed in \cite[\S I.3.5]{vigneras:book}.  
\end{remark}

\subsubsection{The maps $R_{\V,\U}: H_\V \longrightarrow H_\U$ and $C_{\U,\V}: H_\U \longrightarrow H_\V$\label{subsubsec:RUVCUV}}

Let $\U$ and $\V$ be two open compact subgroups of $G$ such that $\V\supseteq \U$.
We have a natural injective $G$-equivariant map
\begin{eqnarray}
\iota_{\V, \U}:\quad \X_\V & \longhookrightarrow & \X_\U \label{iUV}\\
  \chara_\V & \longmapsto &  \chara_\V=\sum_{v\in \V/\U} \chara_{v\U} = \sum_{v \in \V/\U} v\cdot \chara_\U\ , \notag 
\end{eqnarray}
and a surjective $G$-equivariant map
\begin{eqnarray}
\pi_{\U, \V}:\quad \X_\U & \longtwoheadrightarrow & \X_\V \label{piUV} \\
  \chara_\U & \longmapsto & \chara_\V \ . \notag
\end{eqnarray}
 The following diagram is clearly commutative:

  \begin{center}
\begin{tikzcd}
 H _\V= \Hom_G(\X_\V, \X_\V)^{\rm op} \ar[d, "h\mapsto h(\chara_\V)", "{\rotatebox{90}{$\sim$}}"'] \ar[rrrr, "h\mapsto \iota_{\V, \U}\circ h \circ\pi_{\U, \V}"]&&& &\Hom_G(\X_\U , \X_\U)^{\rm op} = H_\U\ar[d, "h\mapsto h(\chara_\U)", "{\rotatebox{90}{$\sim$}}"'] \\
 (k[\V\backslash G/\V], \: \cdot \:)
\ar[rrrr, "\iota_{\V, \U}", hook]&&&& ( k[\U\backslash G/\U], \: \cdot \:) 
\end{tikzcd}
\end{center}
We denote by
\begin{equation}
 R_{\V, \U} : H_\V  \longhookrightarrow  H _\U \label{f:RVU0}  \qquad
h    \longmapsto   \iota_{\V,\U} \circ h \circ \pi_{\U, \V} 
\end{equation}
the injective map described by the diagram above.
On the other hand,  we have a commutative diagram
 \begin{equation}
 \label{f:CUV0diag}
  \begin{tikzcd}
 H _\U= \Hom_G(\X_\U, \X_\U)^{\rm op} \ar[d, "h\mapsto h(\chara_\U)", "{\rotatebox{90}{$\sim$}}"'] \ar[rrrr, "h\mapsto \pi_{\U,\V}\circ h \circ\iota_{\V,\U}"] &&& & \Hom_G(\X _\V, \X_\V)^{\rm op} = H_\V \ar[d, "h\mapsto h(\chara_\V)", "{\rotatebox{90}{$\sim$}}"'] \\
 (k[\U\backslash G/\U], \: \cdot \:) \ar[rrrr, "\tau^\U_g\mapsto {[}\V_g: \U_g {]}  \tau_g^{\V}"] &&&&   (k[\V\backslash G/\V], \: \cdot \:)  
\end{tikzcd}
\end{equation}   We denote by
\begin{equation}
C_{\U,\V}: H _\U \longrightarrow  H_\V \label{f:CUV0} \qquad
h  \longmapsto  \pi_{\U,\V} \circ h \circ \iota_{\V,\U} 
\end{equation}
 the top map of the diagram.

 \begin{proof}[Proof of the commutativity of \eqref{f:CUV0diag}]
For  $g\in G$ we compute using \eqref{convo} that $(h^\U_g \circ \iota_{\V,\U})(\chara_\V) = \chara_\V\cdot \tau^\U_g\in H_\U$ has support $\V g \U$, and its value at $g$ is equal to the cardinality of the set 
\[ (g\U g^{-1} \U  \cap \V)/\U = (g \U g^{-1} \cap \V)\U/ \U \cong (g \U g^{-1} \cap \V)/ \U_{g} = g((\U\cap \V_{g^{-1}})/ \U_{g^{-1}})g^{-1} \ . \]
Thus
$(h^\U_g \circ \iota_{\V,\U})(\chara_\V) = [ (\U\cap \V_{g^{-1}}) : \U_{g^{-1}} ] \chara_{\V g \U}. $
It is easy to see that $\pi_{\U, \V}(\chara_{\V g \U})$ has support  $\V g \V$ and value at $g$ equal to $[ \V_g: (\V\cap g\U g^{-1}) ] = [ \V_{g^{-1}} :  (\U\cap\V_{g^{-1}})]$.
Thus, we have proved that  $\pi_{\U,\V}\circ h^\U_g \circ \iota_{\V,\U}$ maps $\chara_\U$ to $[\V_{g^{-1}}: \U_{g^{-1}}] \tau_g^\V= [\V_{g}: \U_{g}] \tau_g^\V$. 
\end{proof}

\textbf{Assume now that $[\V: \U]$ is invertible in $k$.}
We define the following element in $H_\U$:
\begin{equation}
e_{\U,\V} := \frac{1}{[\V:\U]} \chara_\V = \frac{1}{[\V:\U]} \sum_{v\in \U\backslash \V/\U}\tau_v^\U \ .\label{f:eUV}
\end{equation}
Using \eqref{convo}, it is easy to see that it is an idempotent element in $H_\U$.

\begin{remark}\phantomsection \label{rema:eUV}

\begin{enumerate}[i.]
\item \label{rema:eUV-i} We have $\chi_{\triv}^\U(e_{\U,\V})=1$.
\item If $\mathcal W$
is an open compact subgroup of $G$ such that
$\mathcal W\supseteq  \V\supseteq \U$ we have  $e_{\U,\mathcal W}= e_{\U,\mathcal W}\cdot e_{\U, \V}= e_{\U,\V}\cdot e_{\U,\mathcal W}$. 

\item $e_{\U,\V} H_\U=k[\V\backslash G/ \U]= \X_\U^\V$. 
For any $f\in \X_\U^\U=H_\U$ we have indeed
\begin{equation}\label{f:comp-l-UV}e_{\U,\V}\cdot f=\frac{1}{[\V:\U]} \sum_{v\in \V/\U } f(v^{-1}{}_-)=:\frac{1}{[\V:\U]} N_{\V/\U}(f)\ .\end{equation}
So $f\in \X_\U^\V$ if and only if $e_{\U,\V} \cdot f= f$.

\item \label{rema:eUV-iv}The restriction of  $\pi_{\U, \V}$ to $\X_\U e_{\U,\V}$ provides an isomophism $\X_\U e_{\U,\V}\stackrel{\sim}{\longrightarrow} \X_\V$.
\end{enumerate}
\end{remark}

\begin{remark} \label{rema:section0UV}  Still assuming that $[\V:\U]$ is invertible in $k$, one easily checks that we have the following relations:
\begin{equation}
\label{piiotaUV-props}
\pi_{\U,\V}\circ \iota_{\V,\U} =  [\V: \U]\id_{\X_\V}, \qquad \iota_{\V,\U} \circ \pi_{\U,\V} = [\V:\U]e_{\U,\V}, \qquad \pi_{\U, \V} \circ e_{\U,\V} = \pi_{\U,\V}, \qquad e_{\U,\V} \circ \iota_{\V,\U} = \iota_{\U,\V}.
\end{equation}
From these equalities, we deduce the following :

\begin{enumerate}[i.]
\item  The map $\frac{1}{[\V:\U]} \iota_{\V,\U}$ is a section for the surjection $\pi_{\U, \V}: \X_\U \longtwoheadrightarrow \X_\V$. 
\item  We have $C_{\U, \V}\circ R_{\V,\U}=[\V:\U]^2\id_{H_\V}$. In particular, $C_{\U,\V}$ is surjective. 
\item The linear map $\frac{1}{[\V:\U]}R_{\V,\U}: H_\V \longrightarrow H_\U $ preserves the product.  Note that it sends the unit $\chara_\V$ to the idempotent  $e_{\U,\V}$.  On the other hand, there exist groups $G \supseteq \V \supseteq \U$ such that no multiple of $C_{\U,\V}$ preserves the product; see the lemma below for an example.
\item \label{rema:section0UV-iv}  Using the last three equalities  in \eqref{piiotaUV-props}, we see that
the restricted maps
\[ H_\V\xrightarrow{\frac{1}{[\V:\U]}R_{\V,\U} }  e_{\U,\V} \, H_\U \, e_{\U,\V} \qquad \text{ and }\qquad e_{\U,\V} \, H_\U \, e_{\U,\V} \xrightarrow{\frac{1}{[\V:\U]}C_{\U,\V} }   H _\V \]
 are isomorphisms of  unital $k$-algebras which are inverse to each other.
\item \label{rema:section0UV-v} Using the last two equalities  in \eqref{piiotaUV-props}, we see that
$C_{\U,\V}(_-)=C_{\U,\V}(e_{\U,\V}\cdot {}_- \cdot e_{\U,\V})$.
\end{enumerate}
\end{remark}

\begin{lemma}
Suppose we are in the setting of Section \ref{sec:SL2}, so that $G = \textnormal{SL}_2(\mathfrak{F}) \supseteq K = \textnormal{SL}_2(\mathfrak{O}) \supseteq I$, where $\mathfrak{F}$ is a finite extension of $\bbQ_p$ and $I$ denotes the ``upper-triangular'' pro-$p$-Iwahori subgroup.  Suppose further that $k$ is a field of characteristic $p$. 
Then no multiple of the map $C_{I,K}: H_I \longrightarrow H_K$ preserves the product.
\end{lemma}

\begin{proof}
We freely use notation and results from Section \ref{sec:SL2}. 

Assume by contradiction that $aC_{I,K}$ preserves the product for some $a \in k$.  We first determine $a$.  We have
\begin{flushleft}
$\displaystyle{a[K:I]\textnormal{id}_{\X_K} \stackrel{\eqref{piiotaUV-props}}{=} a(\pi_{I,K} \circ \iota_{K,I})  \stackrel{\eqref{f:CUV0}}{=}   (aC_{I,K})(\textnormal{id}_{\X_I}) =  (aC_{I,K})(\textnormal{id}_{\X_I} \circ \textnormal{id}_{\X_I}) }$
\end{flushleft}
\begin{flushright}
$\displaystyle{ 
= (aC_{I,K})(\textnormal{id}_{\X_I}) \circ (aC_{I,K})(\textnormal{id}_{\X_I}) 
= (a[K:I]\textnormal{id}_{\X_K} ) \circ (a[K:I]\textnormal{id}_{\X_K}) 
= a^2[K:I]^2\textnormal{id}_{\X_K}.}$
\end{flushright}
This forces $a = [K:I]^{-1} = -1$ (by \eqref{f:square}).

On the other hand, let us define $w := s_1s_0s_1 \in {}_K\D_K$ (see \eqref{KDK}), so that $\tau_w^I = \tau_{s_1}^I \tau_{s_0}^I \tau_{s_1}^I$.  Using the explicit description of the set ${}_K\D_K$ we get
\[ C_{I,K}(\tau_w^I)  \stackrel{\eqref{f:CUV0diag}}{=}  [K_w:I_w]\tau_w^K \stackrel{\textnormal{Prop. \ref{prop:index}}}{=}  [K:I] \tau_w^K = -\tau_w^K \ , \]
and analogously
\[ C_{I,K}(\tau_{s_1}^I)C_{I,K}(\tau_{s_0}^I)C_{I,K}(\tau_{s_1}^I)  \stackrel{\eqref{f:CUV0diag}}{=} [K_{s_1}:I_{s_1}]^2[K_{s_0}:I_{s_0}]\tau_{s_1}^K \tau_{s_0}^K \tau_{s_1}^K  \stackrel{\textnormal{Prop. \ref{prop:index}}}{=} 0 \ . \]
The assumption that $-C_{I,K}$ preserves the product implies that these two expressions must be equal, and we arrive at a contradiction.  
\end{proof}

\medskip

We have the trivial character $\chi_{\triv}^\V: H_\V \longrightarrow k$ defined as in \eqref{trivU}. 
Since the map $\frac{1}{[\V:\U]}R_{\V,\U}:H_\V \longrightarrow e_{\U,\V} H_\U e_{\U,\V}$ is a homomorphism of unital algebras, we  also have the character obtained by composition as follows:
\begin{equation}
\label{f:othertrivUV} 
H_\V\xrightarrow{\frac{1}{[\V:\U]}R_{\V,\U} }e_{\U,\V} \, H_\U \, e_{\U,\V} \xrightarrow{\chi^\U_{\triv}}  k \ .
\end{equation}

\begin{lemma}\label{lemma:sametrivUV}
The character \eqref{f:othertrivUV} coincides with $\chi_{\triv}^\V$.
\end{lemma}

\begin{proof} Using Remark \ref{rema:section0UV}-\ref{rema:section0UV-iv}, we prove the result by verifying that the composite $e_{\U,\V} H_\U e_{\V, V}\xrightarrow{\frac{1}{[\V:\U]}{C_{\U,\V}}} H_\V\xrightarrow{\chi^\V_{\triv}}  k$ coincides with $\chi^\U_{\triv}|_{e_{\U,\V} \, H_\U \, e_{\U,\V} }$.  Let $g\in G$.  Then by the equalities \eqref{piiotaUV-props}, the formula in \eqref{f:CUV0diag}, and the definition of the trivial character, we have 
\begin{eqnarray*}
\frac{1}{[\V:\U]}(\chi_{\triv}^{\V}\circ C_{\U,\V}) (e_{\U,\V} \, \tau_g^\U e_{\U,\V}) & = & \frac{1}{[\V:\U]}(\chi_{\triv}^{\V}\circ C_{\U,\V}) (\tau_g^\U)  =  \frac{[\V_g: \U_g]}{[\V:\U]} \chi_{\triv}^\V(\tau^\V_g)\\
 & = & \frac{[\V_g:\U_g][\V: \V_g]}{[\V:\U]}
  =  [\U: \U_g]\ . 
\end{eqnarray*}
By Remark \ref{rema:eUV}-\ref{rema:eUV-i}, we have $\chi^\U_{\triv}(e_{\U,\V})=1$, and therefore the above expression is equal to $\chi^\U_{\triv}(e_{\U,\V} \, \tau_g^\U e_{\U,\V})$.
\end{proof}

\subsection{Hecke $\Ext$-algebras relative to an arbitrary open compact subgroup}\label{subsec:Ext}

\subsubsection{Generalities}

We once again assume $k$ is an arbitrary coefficient ring, and note that the category $\Mod(G)$ has enough injectives.  We form the graded $\Ext$-algebra
\begin{equation*}
  E_{\U}^* := \Ext_{\Mod(G)}^*(\X_\U, \X_\U)^{\mathrm{op}} := \left(\bigoplus_{i \in \bbZ} \Ext^i_{\Mod(G)}(\X_\U, \X_\U)\right)^{\textnormal{op}}
\end{equation*}
over $k$, with the multiplication being the (opposite of the) Yoneda product.  By construction, we have $E_{\U}^0 = H_\U$, and therefore each $E_\U^i$ is an $H_\U$-bimodule.  Using Frobenius reciprocity for compact induction and the fact that the restriction functor from $\Mod(G)$ to $\Mod(\U )$ preserves injective objects, we obtain the identification
\begin{equation}\label{f:frob}
  E_\U ^* = \Ext_{\Mod(G)}^*(\X_\U ,\X_\U )^{\mathrm{op}} \cong \Ext_{\Mod(\U)}^*(1, \X_\U|_{\U}) \cong H^*(\U ,\X_\U) \ .
\end{equation}
Noting that the cohomology of profinite groups commutes with arbitrary direct sums (see \cite[Prop. 1.5.1]{NSW}), the $\U $-equivariant decomposition 
\[ \X_\U = \bigoplus_{g\in \U \backslash G/\U } \X_\U (g) \]
induces a decomposition of $k$-modules
\begin{equation}\label{f:H*dec}
  H^*(\U ,\X_\U  ) = \bigoplus_{g\in \U \backslash G/\U } H^*(\U , \X_\U (g)) \cong \bigoplus_{g \in \U\backslash G / \U} H^*(\U_g,k) \,
\end{equation} 
where we recall that $\U _g = \U \cap g \U  g^{-1}$.  The latter isomorphism comes from an application of the Shapiro isomorphism  $\Sh_g^\U: H^*(\U , \X_\U (g) ) \stackrel{\sim}{\longrightarrow} H^*(\U_g, k)$, the explicit description of which we recall now.

Given another open compact subgroup $\V$ of $G$, the Shapiro isomorphism is defined as the composite
\begin{equation}\label{f:Shapirohybrid}
  \Sh_g^{\V,\U} : H^*(\V,\ind_\U^{\V g \U}(1)) \xrightarrow{\; \res \;} H^*(\V\cap g \U g^{-1},\ind_\U^{\V g \U}(1)) \xrightarrow{\; H^*(\V\cap g \U g^{-1},\ev_g) \;} H^*(\V\cap g \U g^{-1},k)
\end{equation} where $\ev_g : \X_\U \longrightarrow k$ denotes the $ g\U g^{-1}$-equivariant map defined by $\ev_g(f) = f(g)$ (or, here, its restriction to $\ind_\U^{\V g \U}(1)$).  The inverse of $\Sh_g^{\V,\U}$ is the map
\begin{equation}\label{f:Shapiro-inversehybrid}
  (\Sh_g^{\V,\U})^{-1} : H^*(\V\cap g \U g^{-1},k) \xrightarrow{\; H^*(\V\cap g \U g^{-1},\mathrm{i}_g) \;} H^*(\V\cap g \U g^{-1},\ind_\U^{\V g \U}(1)) \xrightarrow{\; \cores \;} H^*(\V ,\ind_\U^{\V g \U}(1)) \ ,
\end{equation}  
where $ \mathrm{i}_g : k  \longrightarrow  \X_\U$ is the $g\U g^{-1}$-equivariant map defined by $\mathrm{i}_g(a) = a \chara_{g\U}$ (which here we see as a map $k \longrightarrow \ind_\U^{\V g \U}(1)$).

In particular, when $\U = \V$, we obtain the Shapiro isomorphism 
\begin{equation}\label{f:Shapiro1}
  \Sh^\U_g : H^*(\U,\X_\U (g)) \xrightarrow{\; \res \;} H^*(\U _g,\X_\U (g)) \xrightarrow{\; H^*(\U _g,\ev_g) \;} H^*(\U _g,k)
\end{equation}

\begin{lemma}\label{lemma:shapindep}
 Let $g,g'\in G$ be such that $\U g\U =\U g'\U $ and $\alpha\in \U $ such that $g'\U =\alpha g\U $. Then $\U _{g'}= \alpha \U _g \alpha^{-1}$ and the composite $\Sh^\U_{g'}\circ  ({\Sh^\U_g})^{-1}$ is equal to $\alpha_*$, the map given by conjugation by $\alpha^{-1}$ on cocycles.  Otherwise said, the following diagram is commutative
 \begin{equation}\label{f:shapindep}
  \begin{tikzcd}
H^*(\U,\X_\U (g))  \ar[d, equals] \ar[rrr, "\Sh_g^\U", "\sim"'] && & H^*(\U _g,k)\ar[d, "\alpha_*", "{\rotatebox{90}{$\sim$}}"'] \\
H^*(\U,\X_\U (g'))
\ar[rrr, "\Sh_{g'}^\U", "\sim"'] &&&H^*(\U _{g'},k) 
\end{tikzcd} 
\end{equation} 
\end{lemma}

\begin{proof}
Suppose first that $\sigma$ is a smooth $\U$-representation, and let us write $g' = \alpha g \alpha'$ for $\alpha, \alpha' \in \U$.  We note that we have a commutative diagram of $\U$-representations
\begin{center}
\begin{tikzcd}
\ind_{\U}^{\U g\U}(\sigma) \ar[r, "\rho_g", "\sim"'] \ar[d, equals] & \ind_{\U_g}^{\U}(g_*(\sigma|_{\U_{g^{-1}}})) \ar[d, "\mathfrak{a}","{\rotatebox{90}{$\sim$}}"'] \\
\ind_{\U}^{\U g'\U}(\sigma) \ar[r, "\rho_{g'}", "\sim"']  & \ind_{\U_{g'}}^{\U}({g'}_*(\sigma|_{\U_{g'^{-1}}}))
\end{tikzcd}
\end{center}
where $\rho_g$ and $\rho_{g'}$ are as in equation \eqref{f:doublecosetisom}, and $(\mathfrak{a}(f))(u) = \alpha'^{-1}\cdot (f(u\alpha))$ for $f \in \ind_{\U_g}^\U(g_*(\sigma|_{\U_{g^{-1}}})$ and $u \in \U$.  (Here $\alpha'^{-1} \cdot_{-}$ denotes the action of $\alpha'^{-1}$ on $\sigma$ without a twist.)

Next, we claim that the following diagram is commutative:
\begin{equation}
\label{f:commdiag-for-shapiro}
\begin{tikzcd}
H^*\big(\U,\ind_{\U_g}^\U(g_*(\sigma|_{\U_{g^{-1}}}))\big) \ar[rrr, "H^*(\U_g{,} \textnormal{ev}_1) \circ \textnormal{res}"] \ar[d, "H^*(\U{,} \mathfrak{a})"] & &  & H^*(\U_g, g_*(\sigma|_{\U_{g^{-1}}})) \ar[d, "\mathfrak{a}'"] \\
H^*\big(\U,\ind_{\U_{g'}}^\U(g'_*(\sigma|_{\U_{g'^{-1}}}))\big) \ar[rrr, "H^*(\U_{g'}{,} \textnormal{ev}_1) \circ \textnormal{res}"] & & & H^*(\U_{g'}, g'_*(\sigma|_{\U_{g'^{-1}}}))
\end{tikzcd}
\end{equation}
Here, $\mathfrak{a}'$ denotes the map on cohomology induced by the compatible pair of homomorphisms (see \cite[\S I.5]{NSW})
\begin{center}
\begin{tabular}{rclcrcl}
$\varphi: \U_{g'}$ & $\longrightarrow$ & $\U_{g},$ & & $\mathtt{f}: g_*(\sigma|_{\U_{g^{-1}}})$ & $\longrightarrow$ & $g'_*(\sigma|_{\U_{g'^{-1}}})$ \\
$u'$ & $\longmapsto$ & $\alpha^{-1}u'\alpha,$ & & $v$ & $\longmapsto$ & $\alpha'^{-1}\cdot v$
\end{tabular}
\end{center}
where once again $\alpha'^{-1}\cdot_{-}$ denotes the untwisted action of $\U$ on $Y$.  (Recall that ``compatible pair'' means that $\mathtt{f}(\varphi(u')*v) = u' *' \mathtt{f}(v)$, where we now use $*$ and $*'$ to denote the twisted actions.)  By dimension shifting (since $\mathfrak{a}$ and $\mathfrak{a}'$ are functorial in $\sigma$), in order to prove the commutativity of the diagram \eqref{f:commdiag-for-shapiro}, it suffices to assume $* = 0$, in which case the commutativity follows from a straightforward calculation.

Suppose now that $\sigma = k$ is the trivial representation.  Combining the above two commutative diagrams gives 
\begin{center}
  \begin{tikzcd}
H^*(\U,\ind_{\U}^{\U g \U}(1))  \ar[d, equals] \ar[rrr, "\Sh_g^\U", "\sim"'] && & H^*(\U _g,k)\ar[d, "\mathfrak{a'}", "{\rotatebox{90}{$\sim$}}"'] \\
H^*(\U,\ind_{\U}^{\U g' \U}(1))
\ar[rrr, "\Sh_{g'}^\U", "\sim"'] &&&H^*(\U _{g'},k) 
\end{tikzcd} 
\end{center}
and we conclude by observing that for $\sigma = k$, the map $\mathfrak{a}'$ is equal to $\alpha_*$.  
\end{proof}

Let $\trace_\U\in \mathbf X_\U^\vee = \Hom_k(\X_\U,k)$ be the linear map defined by
\begin{equation}\label{f:trace}\trace_\U:= \sum_{g\in G/\U} \ev_g\ .\end{equation}
It is easy to check that $\trace_\U:\X_\U \longrightarrow k$ is left $G$-equivariant  when $k$ is endowed with the trivial action of $G$.

\begin{lemma}\label{lemma:SUtriv}
The map $\trace_\U$ is right $H_\U$-equivariant when $k$ is endowed with the trivial character $\chi^\U_{\triv}$ of $H_\U$.
\end{lemma}

\begin{proof}
Recall that $\chi_{\triv}^\U$ was defined in \eqref{trivU}. Since 
$\X_\U$ is a $(G,H_\U)$-bimodule, since it is generated by $\chara_\U$ under the action of $G$,  and since
$\trace_\U$ is left $G$-equivariant and satisfies $\trace_\U(\chara_\U)=1$,  it is enough to check that $\trace_\U\vert_{H_\U}=\chi_{\triv}^\U$. But for $g\in G$, we have
$\trace_\U(\tau^\U_g)=\sum_{x\in G/\U} \tau^\U_g(x) = [\U g \U :\U] = [ \U:\U_g] = \chi_{\triv}^\U(\tau_g^{\U})\ . $
\end{proof}
We denote by $\trace_\U^i := H^i(\U, \trace_{\U}): H^i(\U, \X_\U) \longrightarrow H^i(\U,k)$ the map on cohomology induced by $\trace_\U$. (When $G$ is a split $p$-adic reductive group and $\U$ a pro-$p$-Iwahori subgroup, this map is introduced (and simply denoted $\trace$) in \cite[\S 7.2]{Ext}.) The following is \cite[Rmk. 7.4]{Ext}, the proof of which goes through immediately in our more general setting.

\begin{remark}
We may decompose $\mathbf \trace_\U = \sum_{g\in \U\backslash G/\U} \trace_{\U,g}$ where $\trace_{\U,g} = \sum_{x\in \U g\U/\U} \ev_x$. Each summand $\trace_{\U,g}: \X_{\U} \longrightarrow k$ is $\U$-equivariant and satisfies $\trace_{\U,g}|_{\X_{\U}(g')} = 0$ if $\U g\U \neq \U g'\U$. Furthermore, the following diagram is commutative:
\begin{equation}\label{f:traceg}
\begin{tikzcd}
                &         H^i(\U,\X_{\U}(g))  \ar[dl, "\Sh^\U_g"']  \ar[dd, "H^i(\U {,} \trace_{\U , g})"] \\ 
  H^i(\U_g, k)     \ar[dr, "\cores_{\U}^{\U_g}"']              \\
                &         H^i(\U,k)                  
\end{tikzcd}
\end{equation}
\end{remark}

\subsubsection{The cup product\label{subsec:cup}}

Next, we copy \cite[\S 3.3]{Ext} with $\X_\U$ instead of $\X$ and define a cup product on $E^*_\U$.  Namely, by multiplying compactly supported functions pointwise, we obtain the $G$-equivariant map
\begin{eqnarray*}
  \mathbf{X}_\U \otimes_k \mathbf{X}_\U  \longrightarrow  \mathbf{X}_\U\, &\qquad
  f \otimes f'  \longmapsto  f f' \ .
\end{eqnarray*}
It gives rise to the cup product
\begin{equation}\label{f:cup}
  H^i(\U , \mathbf{X}_\U) \otimes_k H^j(\U ,\mathbf{X}_\U) \xrightarrow{\; \cup \;} H^{i+j}(\U ,\mathbf{X}_\U) \
\end{equation} which has the property that, for $g\in G$,
\begin{equation}\label{f:orth}
  H^i(\U , \mathbf{X}_\U(g)) \cup H^j(\U ,\mathbf{X}_\U(g')) = 0 \quad\text{whenever}\quad \U g \U\neq \U g'\U \ .
\end{equation} 
On the other hand, since $\ev_{{g}}(f f') = \ev_{{g}}(f) \ev_{{g}}(f')$ and since the cup product is functorial and commutes with cohomological restriction maps, we have a commutative diagram
\begin{equation}\label{f:cup+Sh}
\begin{tikzcd}
     H^i(\U , \mathbf{X}_\U (g)) \otimes_k H^j(\U ,\mathbf{X}_\U (g)) \ar[d, "\Sh_g^\U \otimes \Sh_g^\U"', "\rotatebox{90}{$\sim$}"] \ar[r, "\cup"] & H^{i+j}(\U ,\mathbf{X}_\U (g)) \ar[d, "\Sh^\U_g", "\rotatebox{90}{$\sim$}"'] \\
     H^i(\U _g,k) \otimes_k H^j(\U _g,k) \ar[r, "\cup"] & H^{i+j}(\U_g,k)  
     \end{tikzcd}
\end{equation}
where the bottom row is the usual cup product on the cohomology algebra $H^*(\U _g,k)$.  In particular, we see that the cup product \eqref{f:cup} is graded-commutative.

\subsubsection{The Yoneda product}

Recall that the $\Ext$-algebra $E_\U^* = \Ext_{\textnormal{Mod}(\U)}^*(\X_\U, \X_\U)^{\textnormal{op}}$ is a graded algebra endowed with the (opposite of the) Yoneda product.  We make this product more explicit below.

\begin{proposition}
\label{yoneda-product-U}
Let $\alpha \in H^i(\U,\X_\U(x))$ and $\beta \in H^j(\U, \X_\U(y))$ denote two cohomology classes.  We then have
\[ \alpha\cdot \beta = \sum_{\substack{z \in \U\backslash G /\U \\ \U z \U \subseteq \U x \U\cdot \U y \U}} \gamma_z, \]
where $\gamma_z \in H^{i + j}(\U, \X_\U(z))$ is determined by
\begin{equation}\label{f:yon}\Sh_z(\gamma_z) = \sum_{h \in \U_{x^{-1}} \backslash (x^{-1}\U z \cap \U y \U)/\U_{z^{-1}}} \textnormal{cores}^{\U_z \cap zh^{-1} \U hz^{-1}}_{\U_z}\left(\widetilde{\Gamma}_{z,h}\right)\end{equation}
where 
$\widetilde{\Gamma}_{z,h} := \textnormal{res}^{\U \cap zh^{-1} \U hz^{-1}}_{\U_z \cap zh^{-1} \U hz^{-1}}\left(a_* \Sh_{ {x}}(\alpha)\right)\cup \textnormal{res}^{z\U z^{-1} \cap zh^{-1} \U hz^{-1}}_{\U_z \cap zh^{-1} \U hz^{-1}}\left((a {x}c)_*\Sh_{ {y}}(\beta)\right),$
with $h = c {y}d =  {x}^{-1}a^{-1} z$, where $a,c,d \in \U$. 
\end{proposition}

\begin{proof}
This follows in exactly the same way as the proof of \cite[Prop. 5.3]{Ext}: one easily checks that the cited proof works in our more general locally profinite setting.

\end{proof}

\subsubsection{The anti-involution}
\label{subsec-antiinvolution}

Next, we examine an anti-involution on $E_{\U}^*$.  For $g\in G$, we have $\U_{g^{-1}} = \U \cap g^{-1}\U g = g^{-1}(\U \cap g\U g^{-1}) g = g^{-1}\U_g g$, and hence a linear isomorphism
\begin{equation*}
  (g^{-1})_*:  H^*(\U_g, k)\stackrel{\sim}{\longrightarrow}  H^*(\U_{g^{-1}}, k) \ .
\end{equation*}
 Via the Shapiro isomorphism \eqref{f:Shapiro1}, this induces the linear isomorphism  $\anti_{\U,g}$:
\begin{equation}
\begin{tikzcd}
H^*(\U, \X_\U(g)) \ar[d, "\Sh_g^{\U}"', "{\rotatebox{90}{$\sim$}}"] \ar[rrr, "\anti_{\U,g}", "\sim"'] && & H^*(\U, \X_\U(g^{-1})) \ar[d, "\Sh^{\U}_{g^{-1}}", "{\rotatebox{90}{$\sim$}}"'] \\
H^*(\U_g, k)\ar[rrr, "(g^{-1})_*", "\sim"'] & && H^*(\U_{g^{-1}}, k)  
\end{tikzcd}
\end{equation}
It easily follows from Lemma \ref{lemma:shapindep} that the map $\anti_{\U,g}$ defined above only depends on the double  coset $\U g \U$.
Hence, summing over a system of representatives  of $\U\backslash G/\U$, the maps $(\anti_{\U,g})_{g\in \U\backslash G/\U}$ induce a linear isomorphism
\begin{equation*}
  \anti_\U:= \bigoplus_{g\in \U\backslash G/\U} \anti_{\U,g}:  H^*(\U,\X_\U) \stackrel{\sim}{\longrightarrow} H^*(\U,\X_\U) \ .
\end{equation*}
(When $\U$ is the pro-$p$-Iwahori subgroup of a split $p$-adic reductive group, this is exactly the map  $\anti$ introduced in \cite[\S 6]{Ext}.)
The next lemma follows from a straightforward calculation using the definitions of $\Sh_g^{\U}$ (see \eqref{f:Shapirohybrid} and \eqref{f:Shapiro-inversehybrid}).
\begin{lemma}
\label{anti-inv-on-H}
The restriction of $\anti_{\U}$ to $E^0_{\U} \cong H^0(\U,\X_\U) \cong H_\U$ is given by
$\anti_{\U}(\tau_g^{\U}) = \tau_{g^{-1}}^{\U} \text{ for any $g \in G$. }$
\end{lemma}

\begin{lemma}
\label{lemma-JU-antiinvolution}
The map $\anti_{\U}$ defines an involutive anti-automorphism on $E_\U^*$: given $\alpha \in H^i(\U,\X_\U)$ and $\beta \in H^j(\U, \X_\U)$, we have
$\anti_{\U}(\alpha\cdot \beta) = (-1)^{ij}\anti_{\U}(\beta)\cdot \anti_{\U}(\alpha).$
\end{lemma}

\begin{proof}
Using the explicit product formula in Proposition \ref{yoneda-product-U}, this follows in exactly the same way as the proof of \cite[Prop. 6.1]{Ext}.
\end{proof}

\begin{lemma}
\label{lemma-JU-triv}
Suppose $G$ is unimodular.  We then have $\chi_{\triv}^{\U} \circ \anti_{\U}|_{E^0_{\U}} = \chi_{\triv}^{\U}$ (as linear maps).
\end{lemma}

\begin{proof}
By Lemma \ref{anti-inv-on-H} and the definition of $\chi_{\triv}^{\U}$ (equation \eqref{trivU}), it suffices to verify $[\U g \U:\U] = [\U g^{-1} \U:\U]$.  This follows from the proof of \cite[\S I, Prop. 3.6]{vigneras:book}.
\end{proof}

\subsubsection{\label{subsubsec:compa}The maps $R_{\V,\U}^*: E_\V^* \longrightarrow E_\U^*$ and $C_{\U,\V}^*: E_\U^* \longrightarrow E_\V^*$}

Let $\U\subseteq \V$ be two open compact subgroups of $G$. We first recall the following:

\begin{remark}\label{rema:coressurj} For   any smooth representation $\sigma$ of $\V$, we recall that 
the composite
\[ H^*(\V, \sigma)\xrightarrow {\res^{\V}_ {\U}} H^*(\U, \sigma)\xrightarrow {\cores^{\U}_ {\V}}
H^*(\V, \sigma) \]
is nothing but the multiplication by the scalar $[\V: \U]$ (\cite[Cor. 1.5.7]{NSW}).  So  if $[\V:\U]$ is invertible in $k$, we have
\begin{equation}\label{f:cosurj}\cores^{\U}_ {\V}:H^*(\U, \sigma)\longrightarrow H^*(\V, \sigma) \quad \text{is surjective, and}\end{equation} 
\begin{equation}\label{f:resinj}\res^{\V}_ {\U}:H^*(\V, \sigma)\longrightarrow H^*(\U, \sigma)\quad \text{ is injective.}\end{equation}
\end{remark}

We define next the homomorphisms of graded algebras 
\[ R_{\V,\U}^*: E_\V^*\longrightarrow E_\U^*,\qquad C_{\U,\V}^*: E_\U^* \longrightarrow E_\V^* \]
which extend
$R_{\V,\U}: H_\V \longrightarrow H_\U$ and  $C_{\U,\V}: H_\U \longrightarrow H_\V$ 
defined in Subsection \ref{subsubsec:RUVCUV}. The maps $\iota_{\V,\U}\in \Hom_G(\X_\V, \X_\U)$ and $\pi_{\U,\V}\in \Hom_G(\X_\U, \X_\V)$ were introduced respectively in \eqref{iUV} and \eqref{piUV}, and we view them respectively as elements of $\Ext^0_{\Mod(G)}(\X_\V, \X_\U)$ and of $\Ext^0_{\Mod(G)}(\X_\U, \X_\V)$.  We  introduce the maps $R_{\V,\U}^*$ and $C_{\U,\V}^*$ defined respectively by
\begin{equation}
\label{f:RVU*}
R_{\V,\U}^*\quad:\quad  E_\V^*=\Ext^*(\X_\V, \X_\V)^{\mathrm{op}}\xrightarrow{f\mapsto \iota_{\V,\U} \cdot f \cdot \pi_{\U,\V}}\Ext^*(\X_\U, \X_\U)^{\mathrm{op}}= E^*_\U\ ,
\end{equation}
\begin{equation}
\label{f:CUV*}
C_{\U,\V}^*\quad:\quad  E^*_\U=\Ext^*(\X_\U, \X_\U)^{\mathrm{op}}\xrightarrow{f\mapsto \pi_{\U,\V} \cdot f \cdot \iota_{\V,\U}}\Ext^*(\X_\V, \X_\V)^{\mathrm{op}}= E_\V^*
\end{equation}
where the symbol $\cdot$ denotes  the Yoneda composition.

\begin{remark}
\label{rema:RPCP*}
Suppose $[\V:\U]$ is invertible in $k$.  Echoing Remark \ref{rema:section0UV}, we have:
\begin{enumerate}[i.]
\item  $C_{\U, \V}^*\circ R_{\V,\U}^*=[\V:\U]^2\id_{E_\V^*}$. In particular, $C^*_{\U,\V}$ is surjective and $R_{\V,\U}^*$ is injective.
\item The restricted maps
\[ E_\V^*\xrightarrow{\frac{1}{[\V:\U]}R^*_{\V,\U} }  e_{\U,\V} \cdot E^*_\U  \cdot e_{\U,\V} \qquad \text{ and }\qquad e_{\U,\V} \cdot E^*_\U \cdot e_{\U,\V} \xrightarrow{\frac{1}{[\V:\U]}C^*_{\U,\V} }   E^* _\V \]
are well-defined and are homomorphisms of  unital $k$-algebras which are inverse to each other. \label{rema:RPCP*-ii}
\item   
$C^*_{\U,\V}(_-)=C^*_{\U,\V}(e_{\U,\V}\cdot{}\,_-\,  \cdot e_{\U,\V})$. \label{rema:RPCP*-iii}
\end{enumerate}
\end{remark}

We give below an explicit description of the maps $R_{\U,\V}^*$ and $C_{\V,\U}^*$  in terms of restriction and corestriction when   $E^*_\U$ (resp., $E^*_\V$) is written as a direct sum as in \eqref{f:H*dec}.   The proofs of the following propositions are given in Appendix \ref{app:RC} (where we first describe the effect of composing and precomposing by $\iota_{\V,\U}$ and $\pi_{\U,\V}$ on cohomology via the Shapiro isomorphism).

\medskip

To state the propositions, let $g\in G$ and $y\in \V g \V$. We choose an expression $y= v g v'$ and notice that $v \V_g v^{-1}= \V_y$. Furthermore, since conjugation by an element of $\V_{g}$ is trivial on $H^*(\V_{g}, k)$,  the map
\begin{eqnarray*}
H^*(\V_{g}, k) & \longrightarrow & H^*(\U_y,k), \\
 a & \longmapsto & \res^{ \V_{y}}_{\U_y} (v_*a)
\end{eqnarray*}
 is well-defined and independent of the choice of $v$.

\begin{proposition}
\label{prop:RP} 
With the notation as above, we have the following commutative diagram:
\begin{equation}
\label{f:defiRP*}
\begin{tikzcd}
E_\V^* \ar[d, equal] \ar[rrrr, "R_{\V,\U}^*"] &&& &E^*_\U \ar[d, equal]\\
H^*(\V, \X_\V) \ar[d, equal] \ar[rr, "H^*(\V{,} \iota_{\V, \U})" ] & & H^*(\V, \X_\U) \ar[rr, "\res^{\V}_{\U}"] && H^*(\U, \X_\U) \ar[d, equal] \\
H^*(\V, \X_\V) \ar[rr, "\res^{\V}_{\U}"] & & H^*(\U, \X_\V) \ar[rr, "\quad H^*(\U{,} \iota_{\V, \U})"] && H^*(\U, \X_\U) \\
H^*(\V, \X_\V(g)) \ar[rrrrd, "\bigoplus_y \res^{\V_y}_{\U_y}\circ\Sh^\V_y"] \ar[d, "\Sh^\V_{g}", "{\rotatebox{90}{$\sim$}}"'] \ar[u, hook] & &&& \bigoplus_{y\in  {\U\backslash \V g\V/\U}} H^*(\U ,\X_\U(y)) \ar[d, "\bigoplus_y \Sh^\U_{y}", "{\rotatebox{90}{$\sim$}}"'] \ar[u, hook]  \\
H^*(\V_{g},k) \ar[rrrr, "a\mapsto \bigoplus_y \res_{\U_y}^{ \V_y}(v_* a)"']  & &&& \bigoplus_{y\in  {\U\backslash \V g\V/\U}} H^*(\U_y, k)
\end{tikzcd}
\end{equation}
where for $y\in \V g\V$ we write $y$ as $y=vg v'$ with $v,v'\in \V$. 
\end{proposition}

\begin{proposition}\label{prop:CP}
Let $g\in G$.  The following diagram commutes:
\begin{equation}
\label{f:defiCP*}
 \begin{tikzcd}
 E_\U^* \ar[d, equals] \ar[rrr, "C_{\U,\V}^*"] &&& E^*_\V \ar[d, equals]\\
H^*(\U, \X_\U)& && H^*(\V, \X_ \V) \\
H^*( \U, \X_\U(g)) \ar[d, "\Sh^ \U_g", "{\rotatebox{90}{$\sim$}}"'] \ar[u, hook] & && \ar[d, "\Sh^ \V_g", "{\rotatebox{90}{$\sim$}}"'] \ar[u, hook]\ H^*( \V , \X_ \V(g)) \\
H^*( \U_g,k)\ar[rrr, "\cores^{ \U_g}_{ \V _g}"] & && H^*( \V _g, k) 
\end{tikzcd}
\end{equation}
\end{proposition}

\begin{proposition}
\label{lemma:projfor0}

We have the following commutative diagrams:
\begin{equation}
\label{f:projfor}
\begin{tikzcd}
E^i_\V \otimes_kE^j_\V  \ar[d, "R_{\V,\U}^{ j}", , shift left=3.5ex]    \ar[rr,"\cup"] && E^{i+j}_\V \\
 E^i_\U \otimes_k E^j_\U \ar[u, "C_{\U,\V}^i", , shift left=3.5ex] \ar[rr, "\cup"] && E^{i+j}_\U  \ar[u, "C_{\U,\V}^{i+j}"]
 \end{tikzcd}
\end{equation}

\begin{equation} 
\label{f:J+RC*} 
\begin{tikzcd}
E_\V^* \ar[d, "\anti_\V"] \ar[rrr, "R_{\V,\U}^*"] && & E^* _\U \ar[d, "\anti_\U"]\\
E_\V^*  \ar[rrr, "R^*_{\V,\U}"] &&&   E^*_\U \ .
\end{tikzcd}
\qquad\text{ and }\qquad
\begin{tikzcd}
E^*_\U \ar[d, "\anti_\U"] \ar[rrr, "C_{\U\V}^*"] & & & E^*_\V  \ar[d, "\anti_\V"]\\
E^*_\U  \ar[rrr, "C_{\U\V}^*"] & & & E^*_\V  \ .
\end{tikzcd}
\end{equation}

\end{proposition}

\begin{proof}
We prove \eqref{f:projfor}.  We want to show, for $A\in H^i (\U, \X_\U)$ and $B\in  H^j(\V, \X_\V)$ that
\begin{equation} C_{\U,\V}^{i+j}(A\cup R^{ j}_{\V,\U}(B))=C_{\U,\V}^i(A)\cup B \ .\label{f:projform}\end{equation}
Using \eqref{f:orth}, it is enough to show it for $B\in H^j(\V, \X_\V(g))$ and $A\in H^i(\U , \X_\U(g))$ for some $g\in G$.
By \eqref{f:defiRP*}, the  image under $\Sh_g^\U$ of the component in $H^j(\U, \X_\U(g))$ of  
$R^{ j}_{\V,\U}(B)$  is 
$\res ^{\V_g}_{\U _g} (\Sh^\V_{g}(B)) \ .$
Therefore, by \eqref{f:cup+Sh} the image under $\Sh_g^\U $ of $A\cup R_{\V,\U}^{ j}(B)$ is $\Sh_g^\U (A)\cup  \res ^{\V_g}_{\U _g} (\Sh^\V_{g}(B)) $ and thus
\begin{eqnarray*}
\Sh_g^\V\big(C_{\U,\V}^{i+j}(A\cup R^{ j}_\V(B))\big) & \stackrel{\eqref{f:defiCP*}~ \textnormal{and the above}}{=} & \cores_{\V_g}^{\U _g}\big(\Sh_g^\U (A)\cup  \res ^{\V_g}_{\U _g} (\Sh^\V_{g}(B))\big) \\
 & \stackrel{\textnormal{\cite[Prop. 1.5.3(iv)]{NSW}}}{=} & \cores_{\V_g}^{\U _g}(\Sh_g^\U (A))\cup  \Sh^\V_{g}(B) \\
 & \stackrel{\eqref{f:defiCP*}}{=} & \Sh_g^\V(C_{\U,\V}^i(A)) \cup  \Sh^\V_{g}(B) \\
 & \stackrel{\eqref{f:cup+Sh}}{=} & \Sh_g^\V\big(C_{\U,\V}^i(A)) \cup B\big)
\end{eqnarray*}
This proves the claim using \eqref{f:cup+Sh}.

The first commutative diagram in \eqref{f:J+RC*} follows easily from Subsection \ref{subsec-antiinvolution}, from \eqref{f:defiRP*}, and from the commutation of restriction with conjugation ($g^{-1}_*\circ \res^{\V_g}_{\U_g} = \res_{\U_{g^{-1}}}^{\V_{g^{-1}}}\circ g^{-1}_*$).  The second commutative diagram follows similarly from \eqref{f:defiCP*} and from the commutation of corestriction with conjugation.
\end{proof}

\subsection{Hecke $\Ext$-algebras for $p$-adic reductive groups\label{subsec:Ext-padic}}

We now specialize our choice of group $G$, and endow $E_\U^*$ with several additional pieces of structure.

Let $\mathfrak F$ be a finite extension of $\bbQ_p$ with ring of integers $\mathfrak O$, maximal ideal $\mathfrak M$, and residue field  $\mathbb F_q$ where $q$  is a power of  a prime number $p$. We fix a uniformizer $\pi$ of $\mathfrak O$ and choose the valuation $\val_{\mathfrak{F}}$ on  $\mathfrak{F}$  normalized by $\val_{\mathfrak{F}}(\pi)=1$.

We let $\mathbf{G}_{/\mathfrak{F}}$ denote a connected reductive group defined over $\mathfrak{F}$, and let $\mathbf{S}$ denote a maximal $\mathfrak{F}$-split torus of $\mathbf{G}$.  We use italicized letters to denote groups of $\mathfrak{F}$-valued points, so that $G := \mathbf{G}(\mathfrak{F}), S := \mathbf{S}(\mathfrak{F})$, etc.  Finally, we let $I$ denote a choice of pro-$p$-Iwahori subgroup of $G$, associated to a fixed chamber $C$ of the semisimple Bruhat--Tits building of $G$ contained in the apartment corresponding to $\mathbf{S}$.

\subsubsection{Background on Poincar\'e groups}
\label{poincare-groups}

In the next two subsections we collect several results regarding cohomological properties of open compact subgroups of $G$.  

\medskip

We set $d := \dim_{\bbQ_p}(G)$, and fix throughout this section an open compact subgroup $\U$ of $G$, which \textbf{we assume to be $p$-torsion-free}.  Let $\U_p$ be any choice of pro-$p$ Sylow subgroup of $\U$.  By \cite[Cor. (1)]{serre:cohdim}, we have
\begin{equation}
\label{cohdimU}
\textnormal{cd}_p(\U) = \textnormal{cd}_p(\U_p) = \textnormal{cd}(\U_p) = d < \infty.
\end{equation}
Consequently, by \cite[Thm. V.2.5.8]{lazard}, the pro-$p$ group $\U_p$ is a Poincar\'{e} group in the sense of \cite[\S V.2.5.7]{lazard} or \cite[\S I.4.5]{serre:galoiscoh}.  Precisely, this means that:
\begin{enumerate}
    \item $\dim_{\bbF_p}(H^i(\U_p,\bbF_p))$ is finite for all $i \geq 0$;
    \item $\dim_{\bbF_p}(H^d(\U_p,\bbF_p)) = 1$;
    \item the cup product
    \[ H^i(\U_p,\bbF_p) \otimes H^{d - i}(\U_p,\bbF_p) \longrightarrow H^{d}(\U_p,\bbF_p) \]
    is a non-degenerate bilinear form for all $i \geq 0$.
\end{enumerate}

Poincar\'e groups have an alternate characterization, which we now explain.  Define
\begin{equation}
\label{dualizingmodule}
\mathfrak{I} := \varinjlim_{m} \varinjlim_{\substack{\V  \subseteq G \\ \textnormal{open compact}}} H^d(\V,\bbZ/p^m\bbZ)^\vee,
\end{equation}
where $\vee$ denotes the Pontryagin dual (i.e., if $A$ is a discrete $p$-power torsion abelian group, then $A^\vee := \Hom(A,\bbQ_p/\bbZ_p)$), and the transition maps in the inner direct limit are duals of corestriction maps.  Thus $\mathfrak{I}$ is a torsion $\bbZ_p$-module with a discrete action of $G$.

Since the set of open normal subgroups of $\U_p$ is cofinal in the set of open compact subgroups of $G$, we have
\[ \mathfrak{I}|_{\U_p} \cong \varinjlim_{m} \varinjlim_{\substack{\V  \trianglelefteq \U_p \\ \textnormal{open}}} H^d(\V,\bbZ/p^m\bbZ)^\vee. \]
Thus, the restriction of $\mathfrak{I}$ to $\U_p$ gives the dualizing module of $\U_p$ (see \cite[Thm. 3.4.4]{NSW}).  By \cite[Prop. 3.7.6 and Def. 3.7.1]{NSW}, the fact that $\U_p$ is a Poincar\'e group (in the sense of \cite{lazard} or \cite{serre:galoiscoh}) is equivalent to the following two conditions:
\begin{enumerate}
    \item $\U_p$ is a duality group at $p$ of dimension $d$, that is, for all $i < d$ we have
    $\varinjlim_{\substack{\V \trianglelefteq \U_p \\ \textnormal{open}}}H^i(\V,\bbF_p)^\vee = 0;$
    \item we have $\mathfrak{I} \cong \bbQ_p/\bbZ_p$ as abelian groups.
\end{enumerate}
By equation \eqref{cohdimU} and \cite[Ch. III, \S 7, Exer. 1]{NSW}, we conclude that $\U$ is also a duality group at $p$ of dimension $d$, that is, for all $i < d$, we have 
\[ \varinjlim_{\substack{\V \trianglelefteq \U \\ \textnormal{open}}} H^i(\V,\bbF_p)^\vee = 0. \]
Since $\mathfrak{I}|_{\U}$ is the dualizing module for $\U$ (by the same logic as for $\U_p$), we conclude that $\U$ is a Poincar\'e group at $p$, in the sense of \cite[Def. 3.7.1]{NSW}.

\subsubsection{Duality properties of open compact subgroups}
\label{duality-props}

In order to proceed further, we analyze some properties of the dualizing module $\mathfrak{I}$.

Recall that $\U$ is a $p$-torsion-free open compact subgroup of $G$, and $\mathfrak{I}$ is defined in equation \eqref{dualizingmodule}.  In particular, $\mathfrak{I}|_{\U}$ is the dualizing module of $\U$.  According to the discussion preceding \cite[Thm. 3.4.6]{NSW}, we have a canonically defined trace map
\begin{equation}
\label{def-of-coh-tr}
\textnormal{tr}_{\U} :H^d(\U,\mathfrak{I}) \stackrel{\sim}{\longrightarrow} \bbQ_p/\bbZ_p \ .
\end{equation}
We note that by the proof of step (7) of \cite[\S I.4.5, Prop. 30]{serre:galoiscoh}, the map $\textnormal{tr}_{\U}$ is an isomorphism (the cited step does not require that $\U$ be pro-$p$).

\begin{lemma}
\label{tracecor}
Suppose $\U' \subseteq \U$ is an inclusion of $p$-torsion-free open compact subgroups of $G$.  Then the composition 
\[ H^d(\U',\mathfrak{I}) \xrightarrow{\textnormal{cores}^{\U'}_{\U}} H^d(\U,\mathfrak{I}) \xrightarrow{\textnormal{tr}_{\U}} \bbQ_p/\bbZ_p \]
is equal to $\textnormal{tr}_{\U'}$.  In particular, the maps
\begin{eqnarray*}
\textnormal{cores}^{\U'}_{\U}: H^d(\U',\mathfrak{I}) & \longrightarrow & H^d(\U,\mathfrak{I}), \\
\textnormal{cores}^{\U'}_{\U}: H^d(\U',\mathfrak{I}[p]) & \longrightarrow & H^d(\U,\mathfrak{I}[p]),
\end{eqnarray*}
are isomorphisms, where $\mathfrak{I}[p] := \{v \in \mathfrak{I}: pv = 0\}$.  
\end{lemma}

\begin{proof}
This follows in the exact same way as the proof of \cite[Lem. 4.9]{koziol:functorial} (the cited proof does not require that $\U$ and $\U'$ be pro-$p$).
\end{proof}

\begin{lemma}
\label{trivial-action-dualizing}
The group $G$ acts trivially on $\mathfrak{I}[p]$.  
\end{lemma}

\begin{proof}
    We first record the following fact: by \cite[Cor. 3.4.7]{NSW}, we have
    \[ \mathfrak{I}[p] = \varinjlim_{\substack{\V \subseteq G \\ \textnormal{open compact}}} H^d(\V,\bbF_p)^\vee. \]

    Now let $[v] \in \mathfrak{I}[p] = \varinjlim_{\V} H^d(\V,\bbF_p)^\vee$ and $g \in G$.  We want to show that $g$ acts trivially on $[v]$.  By the Bruhat decomposition with respect to $I$, it suffices to prove the claim for $g \in I$ and $g \in N_G(S)$.  Further, by \cite[Eqs. (53)]{vigneras:hecke1}, if $g \in N_G(S)$, we may assume either that $g$ stabilizes $C$, or that $g$ is a lift of $s_A$, the simple affine reflection corresponding to a simple affine root $A$.

    \begin{enumerate}
        \item Suppose $[v]$ is represented by $v \in H^d(\V,\bbF_p)^\vee$, and suppose either that $g \in I$ or that $g$ is an element of $N_G(S)$ which stabilizes $C$.  Let $x$ denote the barycenter of $C$ (so that $g\cdot x = x$), and choose a real number $r \gg 0$ such that the Moy--Prasad subgroup of depth $r$ satisfies $G_{x,r} \subseteq \V$ and is torsion-free.  Replacing $v$ by $(\textnormal{cores}_{G_{x,r}}^{\V})^\vee(v)$, we may assume $v \in H^d(G_{x,r},\bbF_p)^\vee$. The result now follows from Theorem \ref{mainthm:appendix}.  
        \item Suppose now that $g\in N_G(S)$ lifts a simple affine reflection $s_A$, and suppose again that $[v]$ is represented by $v \in H^d(\V,\bbF_p)^\vee$.  Let $F \subseteq \overline{C}$ denote a facet fixed by $g$, let $x \in F$ (so that $g\cdot x = x$), and choose a real number $r \gg 0$ such that the Moy--Prasad subgroup of depth $r$ satisfies $G_{x,r} \subseteq \V$ and is torsion-free.  Replacing $v$ by $(\textnormal{cores}_{G_{x,r}}^{\V})^\vee(v)$, we may assume $v \in H^d(G_{x,r},\bbF_p)^\vee$. Exactly as above, the result now follows from Theorem \ref{mainthm:appendix}.  
    \end{enumerate}
\end{proof}

\begin{corollary}
\label{cor:cores-isom}
Suppose $\U' \subseteq \U$ is an inclusion of $p$-torsion-free open compact subgroups of $G$.  Then the corestriction map 
$\textnormal{cores}^{\U'}_{\U}: H^d(\U',\bbF_p)  \longrightarrow  H^d(\U,\bbF_p)$
is an isomorphism.
\end{corollary}

\begin{proof}
This follows from Lemmas \ref{tracecor} and \ref{trivial-action-dualizing} (and the fact that $\mathfrak{I}[p] \cong \bbF_p$ as abelian groups).
\end{proof}

\begin{proposition}
    \label{Uduality}
Suppose $\U$ is a $p$-torsion-free open compact subgroup of $G$.  Then:
    \begin{enumerate}
        \item $\dim_{\bbF_p}(H^i(\U,\bbF_p))$ is finite for all $i \geq 0$;
        \item $\dim_{\bbF_p}(H^d(\U,\bbF_p)) = 1$;
        \item the cup product
        \[ H^i(\U,\bbF_p) \otimes H^{d - i}(\U,\bbF_p) \longrightarrow H^d(\U,\bbF_p) \]
        is a non-degenerate bilinear form for all $i \geq 0$.
    \end{enumerate}
\end{proposition}

\begin{proof}
    This follows exactly as in the proof of ``(i) $\Longrightarrow$ (ii)'' in \cite[Prop. 3.7.6]{NSW}, noting that $\U$ being a Poincar\'e group (in the sense of \cite[Def. 3.7.1]{NSW}) implies $\mathfrak{I} \cong \bbQ_p/\bbZ_p$ as abelian groups, while Lemma \ref{trivial-action-dualizing} implies that $\mathfrak{I}[p] \cong \bbF_p$ as $\U$-representations.
\end{proof}

We assume henceforth that $\U$ is $p$-torsion-free.  Given $g \in G$, we have a chain of isomorphisms 
\[ \bbF_p \xrightarrow{\textnormal{tr}^{-1}_{\U_{g^{-1}}}} H^d(\U_{g^{-1}},\mathfrak{I}[p]) \xrightarrow{g_*} H^d(\U_g,\mathfrak{I}[p]) \xrightarrow{\textnormal{tr}_{\U_g}} \bbF_p. \]
(By abuse of notation, we let $\textnormal{tr}_\U$ denote the mod $p$ reduction of the map \eqref{def-of-coh-tr}.)  The resulting element of $\bbF_p^\times = \textnormal{Aut}(\bbF_p)$ is denoted $\chi_{G,\U}(g)$.  

When $\U$ is pro-$p$ and torsion-free, the following two results are contained in \cite[Lem. 2.6]{schneidersorensen}.

\begin{lemma}
    The association $G \ni g \longmapsto \chi_{G,\U}(g) \in \bbF_p^\times$ is a homomorphism.
\end{lemma}

\begin{proof}
    This follows in exactly the same way as the proof of \cite[Lem. 4.11]{koziol:functorial}, using Lemma \ref{tracecor} above.
\end{proof}

\begin{lemma}
    The character $\chi_{G,\U}$ is independent of the choice of $p$-torsion-free open compact subgroup $\U$ of $G$.
\end{lemma}

\begin{proof}
    Suppose $\U'$ is another $p$-torsion-free open compact subgroup, and let $\chi_{G,\U'}$ denote its associated character.  The result follows from the commutativity of the following diagram.  All maps appearing are isomorphisms.  The fact that the corestriction maps are isomorphisms which compose appropriately with the trace maps follows from Lemma \ref{tracecor}.  (We omit decorations on the corestriction maps for readability.)
    \begin{center}
        \begin{tikzcd}
            & & & H^d(\U_{g^{-1}},\mathfrak{I}[p]) \ar[r, "g_*"] \ar[ddlll,"\textnormal{tr}_{\U_{g^{-1}}}"', bend right = 15] & H^d(\U_g,\mathfrak{I}[p]) \ar[rrrdd, "\textnormal{tr}_{\U_g}", bend left=15] & & & \\
            & & & & & & & \\
             \bbF_p & & & H^d(\U_{g^{-1}} \cap \U'_{g^{-1}},\mathfrak{I}[p]) \ar[r, "g_*"] \ar[lll,"\textnormal{tr}_{\U_{g^{-1}} \cap \U_{g^{-1}}'}~~"', near start] \ar[uu, "\textnormal{cores}"] \ar[dd, "\textnormal{cores}"'] & H^d(\U_g \cap \U'_g,\mathfrak{I}[p]) \ar[rrr, "\textnormal{tr}_{\U_g \cap \U'_g}", near start] \ar[uu, "\textnormal{cores}"] \ar[dd, "\textnormal{cores}"'] & & & \bbF_p \\
                         & & & & & & & \\
            &  & & H^d(\U'_{g^{-1}},\mathfrak{I}[p]) \ar[r, "g_*"] \ar[llluu,"\textnormal{tr}_{\U'_{g^{-1}}}", bend left = 15] & H^d(\U'_g,\mathfrak{I}[p]) \ar[rrruu, "\textnormal{tr}_{\U'_g}"', bend right = 15] & & &
        \end{tikzcd}
    \end{center}
\end{proof}

In light of the above lemma, we use the notation $\chi_G$ to denote the character $\chi_{G,\U}$ for any fixed choice of $p$-torsion-free open compact subgroup $\U$ of $G$.  By Lemma \ref{trivial-action-dualizing}, the character $\chi_G$ may equivalently be described by the composition of the isomorphisms
\[ \bbF_p \xrightarrow{\textnormal{tr}^{-1}_{\U_{g^{-1}}}} H^d(\U_{g^{-1}},\bbF_p) \xrightarrow{g_*} H^d(\U_g,\bbF_p) \xrightarrow{\textnormal{tr}_{\U_g}} \bbF_p. \]
for any fixed choice of $p$-torsion-free open compact subgroup $\U$. Hence, by the same argument as in Theorem \ref{mainthm:appendix}, we obtain the following (compare with \cite[Lem. 2.10]{schneidersorensen} and \cite[Cor. 5.2]{kohlhaase}):

\begin{proposition}
\label{prop-chiGtrivial}
    The character $\chi_G$ is trivial.  
\end{proposition}

We deduce the corollary below, which is a generalization to an arbitrary $\U$ and general (not necessarily split) reductive group $\mathbf G$ of the first part of the statement in \cite[Prop. 7.16]{Ext}.

\begin{corollary}
\label{cor-conjugationcorestriction}
Suppose $\U$ is a $p$-torsion-free open compact subgroup of $G$, and let $g \in G$.  Then the diagram
\begin{center}
\begin{tikzcd}
H^d(\U_{g^{-1}}, \bbF_p) \ar[rr, "g_*"] \ar[dr, "\textnormal{cores}^{\U_{g^{-1}}}_\U"'] & & H^d(\U_g, \bbF_p) \ar[dl, "\textnormal{cores}^{\U_g}_{\U}"]\\ 
& H^d(\U,\bbF_p) & 
\end{tikzcd}
\end{center}
commutes.
\end{corollary}

\begin{proof}
We consider the following diagram:
\begin{center}
\begin{tikzcd}
H^d(\U_{g^{-1}}, \bbF_p) \ar[rr, "g_*"] \ar[dr, "\textnormal{cores}^{\U_{g^{-1}}}_\U"] \ar[ddr, "\textnormal{tr}_{\U_{g^{-1}}}"', bend right]& & H^d(\U_g, \bbF_p) \ar[dl, "\textnormal{cores}^{\U_g}_{\U}"'] \ar[ddl, "\textnormal{tr}_{\U_g}", bend left]\\ 
& H^d(\U,\bbF_p) \ar[d, "\textnormal{tr}_\U"]& \\
& \bbF_p & 
\end{tikzcd}
\end{center}
The two lower triangles commute by Lemmas \ref{tracecor} and \ref{trivial-action-dualizing}, while the outer triangle commutes by Proposition \ref{prop-chiGtrivial}.  Since all maps appearing are isomorphisms, we conclude that the upper triangle commutes as well.
\end{proof}

\begin{remark}\label{rema:Fpk}
Suppose $k$ is a field of characteristic $p$.  By applying the base change $_-\otimes_{\bbF_p}k$ and using the universal coefficient theorem (\cite[Ch. II, \S 3, Exercise]{NSW}), the above results remain valid with $\bbF_p$ replaced by $k$.  In particular, by a further abuse of notation, we shall view $\textnormal{tr}_\U$ as an isomorphism $\textnormal{tr}_\U:H^d(\U,k) \stackrel{\sim}{\longrightarrow} k$.
\end{remark}

\subsubsection{Duality properties of the $\Ext$-algebra}

We now consider the $\Ext$-algebra $E_\U^*$ relative to an open compact subgroup $\U$ of a $p$-adic reductive group $G$.  \textbf{We shall henceforth assume that $k$ is a field.}

Throughout this section, we will let $d := \dim_{\bbQ_p}(G)$ denote the dimension of $G$ as a $p$-adic manifold.  We begin by examining the degrees in which $E_\U^*$ is concentrated.

\begin{lemma}
\label{lemma:EUdegrees}
Let $\ell$ denote the characteristic of $k$.  
\begin{enumerate}
\item Suppose $\ell$ does not divide the pro-order of $\U$.  Then the $\Ext$-algebra $E_\U^*$ is supported only in degree $0$, that is, $E_\U^i \neq 0$ if and only if $i = 0$.
\item Suppose $\ell \neq p$, and that $\ell$ divides the pro-order of $\U$.  Then the $\Ext$-algebra $E_\U^*$ is nonzero in infinitely many non-negative degrees.
\item Suppose $\ell = p$, and that $\U$ is $p$-torsion-free.  Then the $\Ext$-algebra $E^*_{\U}$ is supported in degrees $0$ to $d$.
\end{enumerate}
\end{lemma}

\begin{proof}
We shall repeatedly use the isomorphism coming from equation \eqref{f:H*dec}:
\[ E_\U^* \cong \bigoplus_{g \in \U\backslash G / \U} H^*(\U_g,k)\ . \]
\begin{enumerate}
\item If $\ell$ does not divide the pro-order of $\U$, then the pro-$\ell$-Sylow subgroup of $\U_g$ is trivial.  Proposition 14 and Corollary 1 of \cite[Ch. I, \S 3.3]{serre:galoiscoh} then imply $\textnormal{cd}_\ell(\U_g) = 0$, which gives the claim.
\item Suppose $\ell \neq p$ and that $\ell$ divides the pro-order of $\U$.  Let $\U_+$ denote an open, normal, pro-$p$ subgroup of $\U$.  Since $\textnormal{cd}_\ell(\U_+) = 0$, the Hochschild--Serre spectral sequence
\[ H^i(\U/\U_+, H^j(\U_+,k)) \Longrightarrow H^{i + j}(\U,k) \]
collapses to give
\[ H^i(\U/\U_+,k) \cong H^i(\U,k). \]
It therefore suffices to show that $H^i(\U/\U_+,k) \neq 0$ for infinitely many $i$.  Since the $\ell$ divides the pro-order of $\U$ and since $\U_+$ is pro-$p$, we see that $\ell$ also divides the order of the finite group $\U/\U_+$, and therefore $\U/\U_+$ possesses elementary abelian $\ell$-groups of positive rank.  The desired claim may now be deduced from a result of Quillen; for example, see \cite[Cor. 8.3.3, \S 8.1]{evens}.  
\item It follows from equations \eqref{cohdimU} that $E_{\U}^*$ is supported in degrees $0$ to at most $d$.  To show that the upper bound is sharp, we use the displayed isomorphism above to see that $H^d(\U,k)$ injects into $E_{\U}^d$, and conclude using Proposition \ref{Uduality}.
\end{enumerate}
\end{proof}

Next, we discuss duality properties, invoking the results of the previous section.  Therefore, we assume henceforth that \textbf{the characteristic of $k$ is equal to $p$}.  Supposing further that $\U$ is $p$-torsion-free, we can obtain more refined information about the cup product \eqref{f:cup}.  Namely, recall that the trace map $\trace_\U: \X_\U \longrightarrow k$ was defined in equation \eqref{f:trace}, and we set $\trace_\U^i := H^i(\U,\trace_\U)$.

\begin{lemma}
\label{lemma-XUduality}
Suppose $k$ is of characteristic $p$, and that $\U$ is $p$-torsion-free.  Then, for $0 \leq i \leq d$, the bilinear map defined by the composite 
\[ H^i(\U,\X_\U) \otimes_k H^{d - i}(\U, \X_\U) \stackrel{\cup}{\longrightarrow} H^d(\U,\X_\U) \stackrel{\trace_\U^d}{\longrightarrow} H^d(\U, k) \stackrel{\textnormal{tr}_\U}{\longrightarrow} k \]
is nondegenerate. 
\end{lemma}

\begin{proof}
Let $g \in G$.  We have the following diagram:

\begin{center}
\begin{tikzcd}
     H^i(\U , \mathbf{X}_\U (g)) \otimes_k H^{d - i}(\U ,\mathbf{X}_\U (g)) \ar[dd, "\Sh_g^\U \otimes \Sh_g^\U"', "\rotatebox{90}{$\sim$}"] \ar[r, "\cup"] & H^{d}(\U ,\mathbf{X}_\U (g)) \ar[dd, "\Sh^\U_g", "\rotatebox{90}{$\sim$}"'] \ar[r, "\trace_\U^d"] & H^d(\U, k) \ar[dd, equals] \\
     & & \\
     H^i(\U _g,k) \otimes_k H^{d - i}(\U _g,k) \ar[r, "\cup"] & H^{d}(\U_g,k)  \ar[r, "\textnormal{cores}^{\U_g}_{\U}","\sim"'] & H^d(\U,k)
     \end{tikzcd}
\end{center}
The left square is commutative by the discussion preceding diagram \eqref{f:cup+Sh}, while the right square is commutative by the commutativity of diagram \eqref{f:traceg}.  Moreover, the lower corestriction map is an isomorphism by Corollary \ref{cor:cores-isom} and Remark \ref{rema:Fpk}.  Therefore, by Proposition \ref{Uduality} (and Remark \ref{rema:Fpk}), we see that the lower row gives a nondegenerate pairing, and therefore the same is true of the upper row.  Using the property \eqref{f:orth} finishes the claim.
\end{proof}

When $\U \subseteq \V$, we also have the following compatibilities among the Hecke $\Ext$-algebras $E_\U^*$ and $E_\V^*$ (compare Proposition \ref{lemma:projfor0}):

\begin{proposition}
\label{lemma:projfor}
Suppose $k$ is of characteristic $p$, and that $\U \subseteq \V$ are $p$-torsion-free.  We have the following commutative diagram:
\begin{equation}
\label{f:dualitycompa}
\begin{tikzcd}
  H^i(\V, \mathbf{X}_\V) \otimes_k H^{d - i}(\V,\mathbf{X}_\V)  \ar[d, "R_{\V,\U}^{d - i}", shift left = 7ex]  \ar[rr, "\cup"] && H^{d}(\V,\mathbf{X}_\V) \ar[rr, "\trace_\V^d"] & & H^{d}(\V,k) \ar[r, "{\textnormal{tr}_\V}", "{\sim}"'] & k\\
 H^i(\U, \mathbf{X}_\U) \otimes_k H^{d - i}(\U,\mathbf{X}_\U) \ar[u, "C_{\U,\V}^i", shift left=7ex] \ar[rr, "\cup"] && H^{d}(\U,\mathbf{X}_\U) \ar[u, "C_{\U,\V}^d"] \ar[rr, "\trace_\U^d"] &&H^{d}(\U,k) \ar[r, "{\textnormal{tr}_\U}", "{\sim}"'] \ar[u, "\cores_\V^\U", "{\rotatebox{90}{$\sim$}}"'] & k \ar[u, equals]
\end{tikzcd}
\end{equation}
\end{proposition}

\begin{proof} 
We already have the commutativity of the left-hand square by \eqref{f:projfor} and the commutativity of the right-hand square follows from Lemma \ref{tracecor}, Lemma \ref{trivial-action-dualizing}, and Remark \ref{rema:Fpk}.  The commutativity of the remaining square is a direct consequence of the  explicit description of $C_{\U,\V}^d$ in \eqref{f:defiCP*} and of the transitivity of corestriction and of the commutativity of \eqref{f:traceg}. 
\end{proof}

Next, recall that in Subsection \ref{subsec-antiinvolution} we have defined the anti-involution $\anti_\U$.  The next result shows how $\anti_\U$ interacts with $\trace_\U^i$.

\begin{lemma}
\label{lemma-traceantiinvolution}
Suppose $k$ is of characteristic $p$, and that $\U$ is $p$-torsion-free.  Then we have $\trace_\U^d\circ\anti_\U = \trace_\U^d$ as linear maps $H^d(\U,\X_\U) \longrightarrow H^d(\U,k)$.  
\end{lemma}

\begin{proof}
This follows exactly as in the proof of \cite[Cor. 7.17]{Ext}; we briefly give the details.  Given $g \in G$, it suffices to prove the commutativity of the following diagram:
\begin{center}
\begin{tikzcd}
H^d(\U,\X_{\U}(g)) \ar[rr, "\anti_{\U,g}"] \ar[dd, "\textnormal{Sh}_g^{\U}"] \ar[dddd, "H^d(\U{,}\trace_{\U,g})"', bend right=60] & & H^d(\U, \X_{\U}(g^{-1})) \ar[dd, "\textnormal{Sh}_{g^{-1}}^{\U}"'] \ar[dddd, "H^d(\U{,}\trace_{\U,g^{-1}})", bend left=60]\\
 & & \\
H^d(\U_g, k)  \ar[rr, "(g^{-1})_*"]  \ar[dd, "\textnormal{cores}^{\U_{g}}_{\U}"] & & H^d(\U_{g^{-1}}, k) \ar[dd, "\textnormal{cores}^{\U_{g^{-1}}}_{\U}"'] \\
& & \\
H^d(\U, k) \ar[rr, equals] &  & H^d(\U, k)
\end{tikzcd}
\end{center}
The upper square commutes by construction of $\anti_\U$, the outer ``triangles'' commute by diagram \eqref{f:traceg}, and the lower square commutes by Corollary \ref{cor-conjugationcorestriction} and Remark \ref{rema:Fpk}.  
\end{proof}

To proceed further, we recall some general properties of duality, following \cite[\S 7.1]{Ext}.  Given a vector space $Y$, we denote by $Y^\vee$ the $k$-linear dual space $Y^\vee := \Hom_k(Y, k)$.  If $Y$ is a left module over a $k$-algebra $A$, then $Y^\vee$ is naturally a right module over $A$.  Given an anti-involution $j:A\stackrel{\sim}{\longrightarrow} A$, we may twist the action of $A$ on a left module $Y$ by $j$ and thus obtain the right module $Y^j$ via the twisted action $(y,a) \longmapsto j(a)\cdot y$ for $a \in A, y \in Y$.  These statements hold \textit{mutatis mutandis} with ``left'' and ``right'' interchanged, and ${}^jY$ in place of $Y^j$.  If $Y$ is an $A$-bimodule, then we define the twisted $H$-bimodule ${}^jY^j$ the evident way.

Suppose that $Y$ decomposes as $Y = \bigoplus_{d \in \EuScript{D}} Y_d$ for some set $\EuScript{D}$.  We denote by $Y^{\vee, f}$ the so-called finite dual of $Y$, which is defined to be 
\begin{equation}\label{f:finitedual}Y^{\vee,f} := \textnormal{im}\left(\bigoplus_{d \in \EuScript{D}} Y_d^\vee \longrightarrow \prod_{d\in \EuScript{D} } Y_d^\vee = Y^\vee\right).\end{equation}

\begin{remark} \label{etavee}
For a left, resp., right, resp., bi-, $A$-module, the identity map yields an isomorphism of right, resp., left,  resp., bi-, $A$-modules  
\[ ({}^j Y)^\vee = (Y^\vee)^{j},\quad \textrm{resp.,}~ (Y^j)^\vee = {}^j(Y^\vee), \quad \textrm{resp.,}~~({}^j Y^j)^\vee={}^j(Y^\vee)^j. \]
If $Y^{\vee,f}$ is a sub right, resp., left, resp., bi-, $A$-module of $Y^\vee$ then it also yields an isomorphism of left, resp., right, resp., bi-, $A$-modules  
\[ ({}^j Y)^{\vee,f} = (Y^{\vee,f})^{j},\quad \textrm{resp.,}~~ (Y^j)^{\vee,f} = {}^j(Y^{\vee,f}), \quad \textrm{resp.,}~~ ({}^j Y^j)^{\vee,f}={}^j(Y^{\vee,f})^j. \]
\end{remark}

We now follow \cite[\S 7.2.4]{Ext} and apply the above discussion to our setup.  Let us suppose $\U$ is $p$-torsion free.  By Lemma \ref{lemma-XUduality}, the map
\begin{eqnarray}
\label{f:UDelta}
\Delta^i_{\U} : E_{\U}^i = H^i(\U,\mathbf{X}_{\U}) & \longrightarrow & H^{d - i}(\U, \X_{\U})^\vee = (E_{\U}^{d-i})^\vee \\
\alpha & \longmapsto & \textnormal{tr}_{\U} \circ \trace_{\U}^d (\alpha\cup _-)   \notag
\end{eqnarray} 
is injective.  The next result analyzes its $H_\U$-equivariance.

\begin{lemma} 
\label{lemma:duality}
Suppose $k$ is of characteristic $p$, and that $\U$ is $p$-torsion-free.  Then the map $\Delta_\U^i$  yields an injective morphism of $H_\U$-bimodules 
\begin{equation}
  \Delta_\U^i: E_\U^i \longhookrightarrow ({}^{\anti_\U} (E_\U^{d-i})^{\anti_\U})^\vee\  
  \end{equation}  
  with image $({}^{\anti_\U} (E_\U^{d-i})^{\anti_\U})^{\vee,f}$ (relative to the decomposition $E^{d - i}_\U = \bigoplus_{g \in \U \backslash G / \U} H^{d - i}(\U, \X_{\U}(g))$).
\end{lemma}

\begin{proof}
The proof is similar to that of \cite[Prop. 7.18]{Ext}.  The fact that the map is injective with image $( E_\U^{d-i})^{\vee,f}$ comes from the proof of Lemma \ref{lemma-XUduality}.  The fact that $\Delta_\U^i$ respects the action of $H_\U$ follows exactly as in \textit{op. cit.}, using Lemma \ref{lemma-JU-antiinvolution} in place of Proposition 6.1 and Lemma \ref{lemma-traceantiinvolution} in place of Corollary 7.17.
\end{proof}

\subsubsection{The top cohomology}

We now obtain some results about the top cohomology $E_\U^d$.  We assume throughout this section that $k$ is of characteristic $p$ and that $\U$ is $p$-torsion-free.

Recall from Lemma \ref{lemma:SUtriv} that we view the one-dimensional vector space $k$ as a $(G,H_\U)$-bimodule, with $G$ acting trivially on the left, and $H_\U$ acting by $\chi_{\triv}^\U$ on the right.

\begin{lemma}\label{lemma:Hdtriv}
Suppose $k$ is of characteristic $p$, and that $\U$ is $p$-torsion-free.  As either a left or right $H_\U$-module, the space $H^d(\U,k)$ is equal to the character $\chi_{\triv}^\U$.
\end{lemma}

\begin{proof}
The statement about the right $H_\U$-action follows by functoriality in the coefficients, and therefore we focus on the left $H_\U$-action.  In this case, the action of $\tau_g^\U$ is given by the composition
\[ H^d(\U,k) \xrightarrow{\textnormal{res}^{\U}_{\U_{g^{-1}}}} H^d(\U_{g^{-1}}, k) \stackrel{g_*}{\longrightarrow} H^d(\U_g,k) \xrightarrow{\textnormal{cores}^{\U_g}_{\U}} H^d(\U,k) \ . \]
By Corollary \ref{cor-conjugationcorestriction} and Remark \ref{rema:Fpk}, this composition is equal to
\[ H^d(\U,k) \xrightarrow{\textnormal{res}^{\U}_{\U_{g^{-1}}}} H^d(\U_{g^{-1}}, k)  \xrightarrow{\textnormal{cores}^{\U_{g^{-1}}}_{\U}} H^d(\U,k) \ , \]
which is equal to multiplication by $[\U : \U_{g^{-1}}] = \chi_{\triv}^{\U}(\tau^{\U}_{g^{-1}}) = \chi_{\triv}^{\U}(\anti_{\U}(\tau^{\U}_{g})) = \chi_{\triv}^{\U}(\tau^{\U}_g)$ (by Lemmas \ref{anti-inv-on-H} and \ref{lemma-JU-triv}).
\end{proof}

\begin{corollary}
\label{cor:top-coh-ses}
Suppose $k$ is of characteristic $p$, and that $\U$ is $p$-torsion-free.  Then the $(G,H_\U)$-equivariant map $\trace_\U: \X_{\U} \longrightarrow k$ induces an exact sequence of $H_\U$-bimodules
\begin{equation}
 \label{cor:top-coh-ses-eqn}
 0 \longrightarrow \ker(\trace_\U^d) \longrightarrow E^d_\U \stackrel{\trace_\U^d}{\longrightarrow} H^d(\U,k) \longrightarrow 0 \ . 
\end{equation}
\end{corollary}

\begin{proof}
It suffices to show that $\trace_\U^d$ is surjective, which follows (for example) from taking $g = 1$ in the diagram \eqref{f:traceg}.
\end{proof}

We now find a sufficient condition for the splitting of the short exact sequence \eqref{cor:top-coh-ses-eqn}.  Let us choose a set $\EuScript{D}$ of double coset representatives of $\U \backslash G / \U$ containing $1$, and define a subset $\EuScript{E} \subseteq \EuScript{D}$ by
\[ \EuScript{E} := \{d \in \EuScript{D}: [\U:\U_d] \neq 0 \quad \textnormal{in $k$}\}\ . \]
Equivalently, we have $d \in \EuScript{E}$ if and only if $\chi_{\triv}^{\U}(\tau^{\U}_d) \neq 0$.

Recall that $E_\U^0 = \bigoplus_{d\in \EuScript{D}}k\tau_{d}^{\U}$, so that $(E^0_\U)^\vee = \prod_{d\in \EuScript{D}}k\tau^{\U,*}_d$, where $\tau^{\U,*}_d$ is the linear form sending $\tau^{\U}_d$ to $1$ and $\tau^{\U}_{d'}$ to $0$ for $d' \neq d$.  Each $\tau_d^{\U,*}$ lies in $(E_\U^0)^{\vee,f}$, and by equation \eqref{f:UDelta}, the element $\tau_d^{\U,*}$ may be characterized by the property
\begin{equation}
\label{eqn:topcohbasis}
\textnormal{tr}_\U\circ \trace_\U^d\left((\Delta_\U^d)^{-1}(\tau_d^{\U,*}) \cup \tau_{d'}^\U\right) = \begin{cases}1 & \textnormal{if $d' = d$}, \\ 0 & \textnormal{if $d' \neq d$}.\end{cases} 
\end{equation}
We define
\[ \varphi := \sum_{d \in \EuScript{E}} [\U:\U_d]\tau_{d}^{\U,*} \in \prod_{d\in \EuScript{D}}k\tau^{\U,*}_d  = (E^0_{\U})^\vee\ . \]

Exactly as in \cite[\S 8]{Ext}, the short exact sequence of $H_\U$-bimodules
\begin{equation}
\label{f:ses-triv}
0 \longrightarrow \ker(\chi_{\triv}^{\U}) \longrightarrow E^0_{\U} \xrightarrow{\chi_{\triv}^{\U}}\chi_{\triv}^{\U}\longrightarrow 0 \ 
\end{equation}
gives a surjection of $H_\U$-bimodules
\begin{equation}
\label{f:phi-in-ker}
(E^0_{\U})^{\vee} \longtwoheadrightarrow \ker(\chi_{\triv}^{\U})^{\vee} \ .
\end{equation}

\begin{lemma}
Suppose $k$ is of characteristic $p$, and that $\U$ is $p$-torsion-free.  Then the kernel of \eqref{f:phi-in-ker} is a one-dimensional vector space with basis given by $\varphi$.  In particular, by dualizing \eqref{f:ses-triv} we see that $H_\U$ acts on $\varphi$ by the character $\chi_{\triv}^{\U,\vee} = \chi_{\triv}^{\U}$.
\end{lemma}

\begin{proof}
Let $\psi = \sum_{d \in \EuScript{D}} a_d \tau_d^{\U,*} \in \prod_{d\in \EuScript{D}}k \tau^{\U,*}_d$ denote a basis for the kernel of \eqref{f:phi-in-ker}.  Since $\tau_d^{\U} - [\U:\U_d] \tau_1^{\U}$ lies in the kernel of $\chi_{\triv}^{\U}$ for every $d \in \EuScript{D}$, and since the restriction of $\psi$ to $\ker(\chi_{\triv}^{\U})$ is 0, we get
\[ 0 = \psi\left(\tau_d^{\U} - [\U:\U_d] \tau_1^{\U}\right) = a_d - [\U:\U_d]a_1. \]
Hence, we see that $\psi = a_1\varphi$, which gives the claim.
\end{proof}

Suppose now that $\EuScript{E}$ is finite, so that $\varphi \in (E^0_{\U})^{\vee,f}$ (relative to the decomposition $E_\U^0 = \bigoplus_{d\in \EuScript{D}}k\tau_{d}^{\U}$).  The surjection \eqref{f:phi-in-ker} and the above lemma give a short exact sequence of $H_\U$-bimodules
\[ 0 \longrightarrow k\varphi \longrightarrow (E^0_{\U})^{\vee,f} \longrightarrow \ker(\chi_{\triv}^{\U})^{\vee,f} \longrightarrow 0 \ , \]
where $\ker(\chi_{\triv}^{\U})^{\vee,f}$ denotes the image of $(E^0_\U)^{\vee,f} \subseteq (E^0_\U)^\vee$ under the surjection $(E^0_{\U})^\vee \longtwoheadrightarrow \ker(\chi_{\triv}^{\U})^\vee$.

\begin{proposition} \label{prop:decEd}Suppose $k$ is of characteristic $p$, and that $\U$ is $p$-torsion-free.  Suppose further that $\EuScript{E}$ is finite and $\sum_{d\in \EuScript{E}} [\U:\U_d]$ is invertible in $k$.  We then have an isomorphism of $H_\U$-bimodules
\[ E^d_\U\cong H^d(\U,k)\oplus  \ker(\trace_\U^d) \stackrel{\textnormal{Lem.}~\ref{lemma:Hdtriv}}{\cong} \chi_{\triv}^{\U} \oplus  \ker(\trace_\U^d)\ . \]
 \end{proposition}

\begin{proof}
This is identical to the proof of \cite[Prop. 8.6]{Ext}.  Namely, we have a diagram of $H_\U$-bimodules
\begin{center}
\begin{tikzcd}
0 \ar[r] & {}^{\anti_\U}(k\varphi)^{\anti_\U} \ar[d, "\partial"] \ar[r] & ({}^{\anti_\U}(E_\U^0)^{\anti_\U})^{\vee, f} \ar[d, "(\Delta_\U^d)^{-1}", "\rotatebox{90}{$\sim$}"'] \ar[r] & ({}^{\anti_\U}\ker(\chi_{\triv}^{\U})^{\anti_\U})^{\vee,f} \ar[r] & 0 \\
0 & H^d(\U,k)\ar[l] & E^d_{\U}  \ar[l, "\trace_\U^d"'] & \ker(\trace_\U^d) \ar[l] \ar[u, "\varepsilon"] & 0 \ar[l]
\end{tikzcd}
\end{center}
Here $\partial$ and $\varepsilon$ are chosen so as to make the diagram commutative; in particular, we have
\begin{eqnarray*}
\textnormal{tr}_\U \circ\partial(\varphi) & = & \sum_{d \in \EuScript{E}} [\U:\U_d]\textnormal{tr}_\U \circ\trace_\U^d\left((\Delta_\U^d)^{-1}(\tau_d^{\U,*})\right) \\
  & = & \sum_{d \in \EuScript{E}} [\U:\U_d]\textnormal{tr}_\U \circ\trace_\U^d\left((\Delta_\U^d)^{-1}(\tau_d^{\U,*}) \cup \tau_d^{\U}\right) \\
 & = & \sum_{d\in \EuScript{E}} [\U:\U_d]\neq 0  \ .
\end{eqnarray*}
Since $\textnormal{tr}_\U$ is an isomorphism, we conclude that $\partial$ is an isomorphism (since both the domain and codomain are one-dimensional over $k$).  Consequently, $\varepsilon$ is also an isomorphism, and moreover $\partial^{-1}$, composed with the inclusion and $(\Delta_{\U}^d)^{-1}$, gives a splitting of $\trace_\U^d$.  This gives the claim.
\end{proof}

\subsection{Parahoric Hecke $\Ext$-algebras for split groups\label{sec:parah}}

Next, we obtain even more explicit formulas for various products in $E_\U^*$.  For this purpose, we suppose from now on that $\mathbf G$ is $\mathfrak F$-split.

\subsubsection{Bruhat--Tits theory\label{subsec:bt}} 

We now recall the elements of Bruhat--Tits theory which we will use, following \cite[\S I.1]{Sch-St}. See also \cite[\S\S 2, 4]{OS1}.

Let $\mathscr{X}$ (resp.,   $\mathscr X^{\textrm{ext}}$) be the semisimple  (resp., extended)  building of $G$ and  $\rm pr: \mathscr X^{ext} \longrightarrow \mathscr X$  the canonical projection map.  We fix an $\mathfrak{F}$-split maximal torus $\mathbf{T}$ in $\mathbf{G}$, put $T := \mathbf{T}(\mathfrak{F})$, and let $T^0$ denote the maximal compact subgroup of $T$ and $T^1$ its pro-$p$ Sylow subgroup.  The choice of $T$ is equivalent to the choice of  an apartment $\Aa$ in $\mathscr{X}$. We fix a chamber $C$ in $\Aa$  as well as a hyperspecial vertex $x_0$ of $C$.  Associated with each facet $F$ is, in a $G$-equivariant way, a smooth affine $\mathfrak{O}$-group scheme $\mathbf{G}_F$ whose generic fiber is $\mathbf{G}$ and such that $\mathbf{G}_F(\mathfrak{O})$ is the pointwise stabilizer in $G $ of $\pr^{-1}(F)$. Its identity component is denoted by $\mathbf{G}_F^\circ$ so that the reduction $\overline{\mathbf{G}}_F^\circ := \mathbf{G}_F^\circ \times_{\mathfrak{O}}\bbF_q$ is a connected smooth algebraic group over $\bbF_q$. The subgroup $\mathbf{G}_F^\circ(\mathfrak{O})$ of $G $ is open compact, and is called a parahoric subgroup of $G$.  The stabilizer of $x_0$ in $G$ contains a good maximal compact subgroup $K$ of $G$ which in turn contains an Iwahori subgroup $J$ that fixes $C$ pointwise. We have 
\[ \mathbf{G}_C^\circ(\Oo) = J \qquad \textrm{and} \qquad \mathbf{G}_{x_0}^\circ(\Oo) = K \ . \]
(Note that $\mathbf{G}_{x_0}^\circ = \mathbf{G}_{x_0}$ since the special fiber is connected; see \cite[\S 2]{OS1}.)

Denote by $\overline{\mathbf{B}}$ the Borel subgroup of $\overline{\mathbf{G}}_{x_0}$ which is the image of 
the natural morphism $\overline{\mathbf{G}}_C \longrightarrow \overline{\mathbf{G}}_{x_0}$.  Further, we let $\overline{\mathbf{N}}$ denote the unipotent radical of $\overline{\mathbf{B}}$ and $\overline{\mathbf{T}}$ its Levi subgroup. Set 
\[ K^1 :=  \Ker \big(\mathbf{G}_{x_0}(\mathfrak O) \xrightarrow{\textnormal{proj}} \overline{\mathbf{G}}_{x_0} (\mathbb{F}_q) \big) \qquad \textrm{and} \qquad  I := \{g \in K : \textnormal{proj}(g) \in \overline{\mathbf{N}}(\mathbb{F}_q) \} \ . \]
We then have a chain $K^1 \subseteq  I \subseteq  J  \subseteq  K$ of open compact subgroups in $G$ such that
\[ K / K ^1  =   \overline{\mathbf{G}}_{x_0} (\mathbb{F}_q) ~\supseteq~  J / K ^1 = \overline{\mathbf{B}}(\mathbb{F}_q) ~\supseteq~ I/ K ^1 = \overline{\mathbf{N}}(\mathbb{F}_q) \ . \]
The subgroup $I$ is pro-$p$ and is called the pro-$p$ Iwahori subgroup. It is a maximal  pro-$p$ subgroup in $ K $. We have $T\cap J= T^0$ and $T\cap I= T^1$. The quotient $ J /I\cong T^0/T^1$ identifies with $ \overline{\mathbf{T}}(\mathbb{F}_q)$.  In particular, its order is prime-to-$p$.

Consider the associated root datum $(X^*({T }), \Phi, X_*({T }), \check{\Phi})$ where  $X^*({T })$ and $X_*({T })$ denote respectively the set of algebraic characters and cocharacters
of $T$.  Similarly, let $X^*({ Z})$ and $X_*({Z})$ denote respectively the set of algebraic characters and cocharacters of the connected center ${Z}$ of $G$. We note that this root system is reduced by our assumption that the group $\mathbf{G}$ is $\mathfrak F$-split. 
Denote by 
\[ \lp \, {}_- \,, {}_-\, \rp :X_*({T })\times X^*({T })\longrightarrow \mathbb Z \]
the natural perfect pairing. Its $\mathbb{R}$-linear extension is also denoted by
$\lp \, {}_- \,, {}_-\, \rp$.
The vector space
$\mathbb R\otimes _{\mathbb Z}(X_*({T })/X_*({Z})) $  (resp. $ \mathbb R\otimes _{\mathbb Z} X_*({T })$)
considered as an affine space on itself identifies with the standard apartment $\mathscr{A}$ (resp.\ $\mathscr{A}^{\textnormal{ext}}$) of the building $\mathscr{X}$ (resp.\ $\mathscr{X}^{\textnormal{ext}}$).

A root $\alpha \in \Phi$ defines a function on $\mathscr{A}$ which we denote  by $x\longmapsto \alpha(x) $. The choice of chamber $C$ amounts to choosing a subset $\Phi^+$ of $\Phi$ of roots $\alpha$ satisfying $\alpha(C)\geq 0$, and we denote by $\Pi $ a basis for $\Phi$ relative to $\Phi^+$. There is a partial order on $\Phi $ given by $\alpha \preceq \beta$ if and only if $\beta -\alpha $ is a linear combination with (integral) nonnegative coefficients of elements in $\Pi$. To $\alpha\in \Phi$ is associated a coroot $\check\alpha\in\check\Phi$ such that $\lp\check\alpha, \alpha   \rp=2$ and a reflection $s_\alpha$ of $\mathscr{A}$ defined by
\[ s_\alpha: x\longmapsto x- \alpha(x)\check\alpha \mod X_*({Z})\otimes_{\mathbb Z}\mathbb R \ . \]
The subgroup of transformations of $\mathscr{A}$ generated by these reflections identifies with the finite Weyl group $\mathfrak W$, defined to be the quotient $N_G (T)/T$. The group ${\mathfrak W}$ is a Coxeter group generated by the set $S := \{s_\alpha : \alpha \in \Pi\}$, and is endowed with a length function $\ell: \mathfrak{W} \longrightarrow \mathbb{Z}_{\geq 0}$ relative to $S$.

To an element $t\in T $ corresponds a vector $\nu(t) \in \mathbb R\otimes _{\mathbb Z}X_*({T })$ characterized by
\[ \lp\nu(t),\, \chi\rp =-\val_{\mathfrak F}(\chi(t))\]  
for all $\chi\in X^*(T).$  The kernel of $\nu$ is the maximal compact subgroup $T^0$ of $T$. The quotient  $\Lambda := T/T^0$ is a free abelian group  with rank equal to $\dim(T)$, and $\nu$ induces an isomorphism $\Lambda \stackrel{\sim}{\longrightarrow} X_*(T)$. The group $\Lambda$ acts by translation on $\Aa$ via $\nu$.
The extended affine Weyl group $W$ is defined to be the quotient of $N_G (T)/T^0$.  The actions of ${\mathfrak W}$ and $\Lambda$ combine into an action of ${W}$ on $\Aa$ as recalled in \cite[p. 102]{Sch-St}.   The extended affine Weyl group $W$ is then the semi-direct product ${\mathfrak W}\ltimes\Lambda$ (\cite[\S 1.9]{Tits}): ${\mathfrak W}$ identifies with the subgroup of $W$ that fixes any point of $\mathscr{A}^{\textnormal{ext}}$ that lifts $x_0$.

The set of affine roots is ${\Phi_{\aff}} := \Phi\times \mathbb Z ={\Phi_{\aff}^+} \sqcup {\Phi_{\aff}^-}$ where
\[ {\Phi_{\aff}^+}:=\{(\alpha, \mathfrak h) ~:~ \alpha \in\Phi, \,\mathfrak h>0\}\cup\{(\alpha ,0) ~:~ \alpha \in\Phi^+\} \] 
and 
\[ {\Phi_{\aff}^-}:=\{(\alpha, \mathfrak h) ~:~ \alpha \in\Phi, \,\mathfrak h<0\}\cup\{(\alpha ,0) ~:~ \alpha \in\Phi^-\} \ . \]
The extended affine Weyl group $W$ acts on $\Phi_{\aff}$ by 
\[ w\lambda \cdot (\alpha, \mathfrak h) = \left(w(\alpha) ,\,\, \mathfrak h-\lp \lambda, \alpha \rp \right) \ ,\] 
where we denote by $w(\alpha)$ the natural action of ${\mathfrak W}$ on roots.
Denoting by $\Pi_{\min}$ the set of roots which are minimal for the partial order $\preceq$, we define the set of simple affine roots by  
\[ \Pi _{\aff} := \{(\alpha , 0) ~:~ \alpha \in\Pi \}\cup\{(\alpha ,1) ~:~ \alpha \in\Pi _{\min}\} \ . \]
We refer to \cite[\S 4.3]{OS1} for the definition of the reflection $s_A$ on $\Aa$ attached to an affine root $A\in \Phi_{\aff}$. The length function $\ell$ on the Coxeter system ${\mathfrak W}$ extends to $W$ by the formula
\[ \ell(w) = \#\{ A\in \Phi_{\aff}^+ ~ : ~ w(A) \in  \Phi_{\aff}^- \} \]
for $w \in W$.

The affine Weyl group is defined as the subgroup
\[ W _{\aff} := \langle s_A  ~:~ A \in \Phi_{\aff}\rangle \]
of $W $. Let $  S_{\aff} := \{s_A ~:~ A \in \Pi _{\aff}\}$. The pair $(W _{\aff}, S_{\aff})$ is a Coxeter system (\cite[V.3.2 Thm.\ 1(i)]{Bki-LA}), and the length function $\ell$ of $W$ restricted to $W _{\aff}$ coincides with the length function of this Coxeter system. Recall (\cite[\S 1.5]{Lu}) that $W _{\aff}$ is a normal subgroup of $W$:   the set $\Omega $ of elements  with length zero  is an abelian subgroup of $W $ and 
 $W $ admits a semi-direct product decomposition $W = \Omega \ltimes W _{\aff}$.  Further, the length $\ell$ is constant on the double cosets $\Omega w \Omega$ for $w \in W $. In particular $\Omega$ normalizes $S_{\aff}$.

\subsubsection{Parahoric combinatorics}  Let $\P$ be a parahoric subgroup of $G$ containing $J$. Thus, $\P$ is of the form $\P = {\mathbf G}_{F_\P}^\circ(\mathfrak O)$ for $F_\P$ a facet contained in $\overline C$. Attached to this subgroup is the subset $\Pi_\P$ of the affine roots in $\Pi_{\aff}$ taking value zero on $F_\P$, or equivalently the subset $S_\P$ of those reflections in $S_{\aff}$ fixing $F_\P$ pointwise. We let
\[ \Phi_\P := \{ (\alpha, \mathfrak h) \in \Phi_{\aff} ~:~ (\alpha, \mathfrak h)|_{F_\P} = 0\} \ , \] 
\[ \Phi_\P^+ := \Phi_\P \cap \Phi^+_{\aff} \ , \quad  \Phi_{\P}^- : = \Phi_{\P} \cap \Phi^-_{\aff} \ , \]
and
\begin{equation}\label{f:WP}
    {\mathfrak W}_\P := \langle s_{(\alpha, \mathfrak h)} ~:~(\alpha, \mathfrak h)|_{F_\P} = 0\rangle \subseteq W_{\aff} \ . 
\end{equation}
The pair $({\mathfrak W}_\P, S_\P)$ is a Coxeter system with ${\mathfrak W}_{\P}$ being the pointwise stabilizer in $W$ of $\pr^{-1}(F_\P)$ (\cite[V.3.3, IV.1.8 Thm.\ 2(i)]{Bki-LA}), the restriction $\ell|_{{\mathfrak W}_\P}$ coincides with its length function (\cite[IV.1.8 Cor.\ 4]{Bki-LA}), and ${\mathfrak W}_\P$ is finite (\cite[V.3.6 Prop.\ 4]{Bki-LA}).

Next, we recall the Bruhat decomposition relative to the Iwahori subgroup $J$: the group $G$ admits a decomposition
\[ G = \bigsqcup_{w\in W}JwJ \ , \]
where the double coset $JwJ$ is defined as $J\dot{w}J$ for any choice of lift $\dot{w}\in N_G(T)$ of $w$.  (Note that this double coset is independent of the choice of lift $\dot{w}$.)  By \cite[Lem. 4.9]{OS1},  the parahoric subgroup $\P$ admits a decomposition
\[ \P = \bigsqcup_{w \in \mathfrak{W}_\P} JwJ \ . \]

We define the group $\widetilde{W} := N_G(T)/T^1$. In particular, it contains  $T^0/T^1$. We have the exact sequence
\begin{equation} \label{f:esW}
0 \longrightarrow T^0/T^1\longrightarrow \widetilde W\longrightarrow  W\longrightarrow 0\ .
\end{equation}
For any subset $Y$ of $W$ we will denote by $\widetilde Y$ its preimage in $\widetilde W$. Given an element $w\in W$, we denote by $\tilde w$ a lift for $w$ in $\widetilde W$. The length $\ell$ and Bruhat order $\leq$ on $W$ inflate respectively to a length function and order on $\widetilde W$ (see \cite[\S 2.1]{vigneras:hecke1}).

For $w\in W$, we denote by $\dot w$ a lift for $w$ in $N_G(T)$.
The groups $J_{\dot w}$, $\P_{\dot w}$ and $I_{\dot w}$ (defined in Subsection \ref{subsubsec:general}) do not depend on the choice of this lift and we will simply denote them by $J_w$, $\P_{w}$  and $I_w$. 

\begin{lemma}\label{coro25}
Let $v,w\in W$ and $s\in S_{\aff}$ with respective lifts $\tilde v, \tilde w$ and $\tilde s$ in $\widetilde W$. We then have:
\begin{enumerate}
\item $\vert I/I_w\vert= \vert J/J_w\vert = q^{\ell(w)}$; \label{coro25-i}
\item  if $\ell(vw)=\ell(v)+\ell(w)$ then $I \tilde v I \cdot I \tilde w I= I \tilde v \tilde w I$ and  $J v J \cdot J  w J= J  v  w J$ ;\label{coro25-ii}
\item in general, $J uJvJ\subseteq \bigcup _{z\leq u} J zvJ$ and $JuJvJ\subseteq \bigcup _{z\leq v} J u zJ$. \label{coro25-iii}
\end{enumerate}

\end{lemma}
\begin{proof}
The proofs of Points \eqref{coro25-i} and \eqref{coro25-ii} are recalled in \cite[Cor. 2.5, pts. i. and ii.]{Ext} in the case of $I$ (from which the case of $J$ follows easily).  For Point \eqref{coro25-iii},  we verify  $J uJvJ\subseteq \bigcup _{z\leq u} J zvJ$  by induction on $\ell(u)$. When $\ell(u)=0$ the claim is clear.  When $s\in S_{\aff}$ then it is well-known (or follows easily from Point \eqref{coro25-ii} and \cite[Eqn. (20)]{Ext}) that
\[ J s JvJ=\begin{cases} J svJ &\text{if $\ell(sv)=\ell(v)+1$},\cr JvJ\sqcup JsvJ 
&\text{if $\ell(sv)=\ell(v)-1$},\end{cases} \]
so the claim is true when $\ell(u)=0,1$. Now suppose $s$ is such that $\ell(su)=\ell(u)+1$. Then by induction we have
\[ J su J v J=JsJuJvJ\subseteq \bigcup _{z\leq u} JsJ zvJ\subseteq \bigcup _{z\leq u} JszvJ\cup JzvJ \] 
and the claim is proved because for $z$ as above, $z\leq u\leq su$ and
$sz\leq su$ (either $sz\leq z\leq u\leq su$ or $z\leq sz$ in which case $\ell(sz)=\ell(z)+1$ and
$sz\leq su$).
\end{proof}

\begin{remark}\label{rema:square}We have 
$\P/J = \bigsqcup_{w\in \WP} J w J/J$  and  $ JwJ/J \cong J/J_w$ where the latter bijection is given by $jwJ \longmapsto jJ_w$.  By Lemma \ref{coro25}, we have $\vert J/J_w\vert = q^{\ell(w)}$, and since $\WP$ is a subgroup of $W_{\aff}$, the only element of length zero contained therein is $w=1$. Therefore $[\P:J] \equiv 1~ \bmod p$ and 
\begin{equation}
\label{f:square}
[\P:I] \equiv [J:I] \equiv [T^0:T^1]\equiv (-1)^{\textnormal{rk}(\mathbf{T})}~ \bmod p \ .
\end{equation}
\end{remark}

We now construct some distinguished coset representatives for $\mathfrak{W}_\P$ in $W$.  The following results are well-known in the context of finite groups (see \cite[\S\S 2.3, 2.7]{Carter}).  We define 
\begin{equation}\label{representants}
   \D_\P := \{d\in W ~:~ d(\Phi  _\P^+)\subseteq {\Phi}_{\aff}^+\} \ ,
\end{equation}
which is a system of representatives of the left cosets $W / \mathfrak{W}_{\P}$. It satisfies
\begin{equation}\label{additive0}
\ell(du) = \ell(d) + \ell(u)
\end{equation}
for any $u\in \WP$ and $d\in \D_\P$. In particular, $d$ is the unique element with minimal length in $d\mathfrak{W}_\P$ (see \cite[Prop. 4.6]{OS1}).
Note that the set $ \D _\P\subseteq N_G(T)/T^0$ is a system of representatives of the double cosets in $J\backslash G/\P$ and also of $I\backslash G/\P$ (since $T^0\subseteq J\subseteq \P$).

\begin{remark} \label{rema:disId}     
For $d\in \D_\P$, we have $I\cap \P_d= I_d$ and $J\cap \P_d= J_d$.  Indeed, it follows from the first identity of \cite[Prop. 4.13-ii]{OS1} that $I\cap d \P d^{-1}\subseteq d I d^{-1}$, and therefore $I\cap \P_d= I_d$. We then obtain $J\cap \P_d= T^0 (I \cap \P_d)= T^0 I_d= J_d$.
\end{remark}

Next, we define 
\begin{equation}\label{f:DPP}
\DPP := \D_\P\cap \D_\P^{-1} .
\end{equation}

\begin{proposition} \hfill
\label{prop:doublereps}
\begin{enumerate}
\item Each double coset in $ \WP\backslash W/{\mathfrak W}_\P$ contains a unique element in $\DPP$. \label{prop:doublereps-1}
\item If ${d}\in  \DPP$ then $d$ is the unique element of minimal length in its coset $\WP\, d\,\WP$.
\end{enumerate}
\end{proposition}

\begin{proof} In the context of finite groups, this is exactly \cite[Prop. 2.7.3]{Carter}, the proof of which is based on two Lemmas. They translate into our context with similar proofs as follows. 

\begin{sublemma}
\label{sublemma1}
Suppose $xd(A) \in \Phi_{\aff}^-$, where $x\in \WP$, $d\in \DPP$ and $A\in \Pi_\P$. Then $d(A)\in \Pi_\P$.
\end{sublemma}

\begin{proof}[Proof of Sublemma \ref{sublemma1}]
By definition of $\D_\P$, we know that $d(A)\in \Phi_{\aff}^+$. But $x\in \WP$, so the only positive roots it potentially transforms into negative roots are the ones in $\Phi_\P^+$. This is because $\WP$ is generated by $S_\P$ and for $A\in \Pi_{\aff}$, we have $s_A(\Phi_{\aff}^+\smallsetminus\{A\}) = \Phi_{\aff}^+\smallsetminus\{A\}$. Thus $d(A)\in \Phi_{\P}^+$.  This root decomposes uniquely as $d(A)=\sum_{B\in \Pi_\P} m_B B$ with $m_B$ nonnegative integers  (\cite[Lem. 4.5]{OS1}). Hence $A=\sum_{B\in \Pi_\P} m_B d^{-1}(B)$ and since $d^{-1}\in \D_\P$ we have $d^{-1}(B)\in\Phi^+_{\aff}$. But $A\in \Pi_{\P}$ so there can only be one nonzero summand and there exists  $B\in \Pi_\P$ such that $d^{-1}(B)=A$. 
\end{proof}

\begin{sublemma}
\label{sublemma2}
Let $v,w\in \DPP$. Then  ${\mathfrak W}_\P\cap v {\mathfrak W}_\P w^{-1}$ is empty unless $v=w$.
\end{sublemma}

\begin{proof}[Proof of Sublemma \ref{sublemma2}]
The proof of \cite[Lem. 2.7.2]{Carter} goes through readily.
\end{proof}
\end{proof}

\begin{proposition}\label{prop:doubleco}
The set $\DPP$ yields a system of representatives of the double cosets $\P\backslash G/\P$.
\end{proposition}

\begin{proof} 
It follows from Lemma \ref{coro25}\eqref{coro25-iii} that  for any $w\in W$, we have  $\P w \P=J \WP w \WP J$.  This implies directly that any system of representatives of $\WP\backslash W/\WP$  yields a system of representatives of $\P\backslash G/\P$, and we conclude using Proposition \ref{prop:doublereps}\eqref{prop:doublereps-1}.  
\end{proof}

\begin{proposition} \hfill
\label{prop:index}
\begin{enumerate}
\item For $d\in \DPP$, we have $[\P_d: I_d]\equiv [\P:I]\bmod p$.
\item For  $w\in W\smallsetminus \DPP$, we have $[\P_w: I_w]\equiv 0\bmod p$.
\end{enumerate}
\end{proposition}

\begin{proof} 
\begin{enumerate}
 \item Let $d\in \DPP$, and let $F$ the facet of $C$ such that $\P={\mathbf G}_{F}^\circ(\mathfrak O)$.  We consider the pro-unipotent radical $I_F$ of $\P$ as introduced in \cite[\S 3.1]{OS1}.  It is a normal subgroup of $I$ and of $\P$. By Lemma 4.7 of \textit{op. cit.}, there is a subgroup $\EuScript U_F^0$ of $I$ which satisfies $\EuScript U_F^0\subseteq I_d$ (because  $d^{-1}\in \D_\P$), and is such that $I = \EuScript U_F^0 I_F$ (\textit{op. cit.}, Lemma 4.8).  Thus, we have $I= I_d I_F$.

Notice that the normalizer of $I$ in $\P$ is equal to $J$. This is because the  only element with length $0$ in $\mathfrak W_\P$ is the trivial element.  Let $x\in \P_d$.  Since $I_d$ is a pro-$p$-group, we have
\[ |I_d x I_d/ I_d| \equiv [I_d: (I_d)_x] \equiv \begin{cases} 1 \bmod p & \textnormal{if}~(I_d)_x = I_d \ ,\\ 0 \bmod p & \textnormal{if}~(I_d)_x \neq I_d \ .\end{cases} \]
Note that $(I_d)_x = I_d$ if and only if $x$ normalizes $I_d$.  Therefore, to calculate $[\P_d:I_d]$, it suffices to decompose $\P_d$ into $I_d$ double cosets, and focus only on those double cosets $I_dxI_d$ where $x$ normalizes $I_d$.  Thus, let us suppose that $x$ normalizes $I_d$.  Then $x$ normalizes $I= I_d I_F$ so  $x\in J$. Therefore $x\in J\cap \P_d= J_d$ by Remark \ref{rema:disId} (this uses $d\in \D_\P$), and we have (see \eqref{f:square}):
\[ [\P_d: I_d] \equiv [J_d:I_d] \equiv [J:I] \equiv [T^0:T^1] \equiv [\P: I] \bmod p\ . \]
 \item  Let $w\in W$, and let $d\in \D_P$ be such that $w\in d \WP$. We have 
 \begin{equation}\label{preli}
 I_w\subseteq I_d= I\cap \P_d= I\cap \P_w\subseteq \P_w
 \end{equation} 
 where the first equality is justified Remark \ref{rema:disId} and the first inclusion by \cite[Lem. 2.2]{Ext}, noticing that the inclusion is strict if $w\neq d$ by Lemma \ref{coro25}\eqref{coro25-i} and equation \eqref{additive0}. In particular, if $w\neq d$ we see that $I_w$ is a strict pro-$p$ subgroup of $I\cap \P_w$ and of $\P_w$. In particular, $[\P_w: I_w] \equiv 0\bmod p$.  Thus, we have shown that if $w\notin \D_\P$, then $[\P_w: I_w] \equiv 0\bmod p$.  On the other hand, we have $\P_{w^{-1}}= w^{-1} \P_w w$ and $I_{w^{-1}}= w^{-1} I_w w$, which shows that if $w^{-1}\notin \D_\P$, then $[\P_w: I_w] \equiv 0\bmod p$.
 \end{enumerate}
\end{proof}

\subsubsection{Pro-$p$ Iwahori and parahoric Hecke $\Ext$-algebras: examples and notation.\label{pro-p-hecke-alg}\label{connect0}}

We now turn our attention to specific examples of the algebras $H_\U$ and $E_\U^*$.

\begin{example}
\phantomsection
 \begin{enumerate}

\item When $\U=I$, the algebra $H_I$ will simply be denoted by $H$ as in \cite{Ext}. It is known as the \textbf{pro-$p$ Iwahori--Hecke algebra of $G$}.  We will similarly denote other objects introduced above when $\U = I$ e.g., we will write $\X$ instead of $\X_I$, $\tau_g$ instead of $\tau_g^I$ and  $E^*$ instead of $E_I^*$ for the \textbf{pro-$p$ Iwahori--Hecke $\Ext$-algebra of $G$}.

Recall that the group $\widetilde W = N_G(T)/T^1$ provides a system of representatives of the double cosets $I \backslash G/ I$ given by $w\longmapsto I \dot{w} I$, where $\dot{w}$ denotes any choice of lift of $w$ to $G$.  In particular, the double coset $I \dot{w} I$ is independent of the choice of lift, and we write it as $IwI$.  A basis for $H$ is given by the set of all 
\[ \tau_w := \chara_{I w I}\text{ for $w\in \widetilde{W}$} \]
The defining relations of $H$ (i.e., the braid and quadratic relations) are recalled in \cite[\S 2.2]{Ext}.

\item For $\P$ a parahoric subgroup containing $I$ such that $[\P:I]$ is invertible in $k$, we will denote simply by $e_{\P}$ the idempotent
$e_{I, \P}$  introduced in \eqref{f:eUV}. With the notation of Subsection \ref{subsec:bt}, it can also be written as
\[ e_\P=\dfrac{1}{[\P: I]}\sum_{w\in \widetilde {\mathfrak W}_\P} \tau_w\ . \]
By Proposition \ref{prop:doubleco}, a system of representatives of $\P\backslash G/\P$  is given by the subset $\DPP$ of $W$ (see \eqref{f:DPP}).  Thus, a basis for  $H_\P$ is given by the set of all 
\[ \tau_d^\P=\chara_{\P d\P}\text{ for $d\in \DPP$}\ .\]
We simply denote by $\iota_\P$ and $\pi_\P$ respectively the natural injective $G$-equivariant map $\iota_{\P}: \X_\P  \longhookrightarrow  \X$ and the surjective $G$-equivariant map $\pi_{\P}: \X  \longtwoheadrightarrow  \X_\P$ defined in \eqref{iUV} and \eqref{piUV}.  As in \eqref{f:RVU*} and \eqref{f:CUV*} (see also \eqref{f:RVU0} and \eqref{f:CUV0}) they induce maps
 \begin{equation*}
 R^*_\P : E^*_\P  \longrightarrow  E^*  \qquad \textrm{ and } \qquad
C_\P^*: E^*  \longrightarrow  E_\P^* \ .
\end{equation*} 
Recall  (Remark \ref{rema:RPCP*}) that $E^*_\P\xrightarrow{\frac{1}{[\P:I]}R^*_\P }  e_\P E^* e_\P$ and $e_\P E^* e_\P \xrightarrow{\frac{1}{[\P:I]}C_\P^* }   E^* _\P $ are homomorphisms of unital $k$-algebras which are inverse to each other.

 \item \label{subsubsec:JK-1}
 Suppose $\P = J$, so that $\P$ is associated to the chamber $C$.  Then $H_J$ is known as the \textbf{Iwahori--Hecke algebra}, and has the following description (for more details, see \cite{vigneras:hecke1}).  As a vector space, we have
 $H_J = \bigoplus_{w \in W} k\tau^J_w,$
 subject to the defining relations
\begin{alignat*}{2}
 \tau_w^J \tau^J_{w'} & =  \tau^J_{ww'} && \textnormal{if}~  \ell(w) + \ell(w') = \ell(ww'),  \\
 (\tau^J_{s_A})^2 & =  (q - 1)\tau^J_{s_A} + q\tau^J_{1} \qquad && \textnormal{if}~A \in \Pi_{\aff}.
\end{alignat*}

 \item \label{subsubsec:JK-2}
 Suppose $\P = K$, so that $\P$ is associated to the vertex $x_0$.  Then $H_K$ is known as the \textbf{spherical Hecke algebra}, and has the following description.  As a vector space, we have
 \[ H_K = \bigoplus_{w \in {}_K\D_K} k\tau^K_w = \bigoplus_{\lambda \in X_*(T)_+} k\tau^K_{\lambda(\pi)}, \]
  where $X_*(T)_+$ denotes the cone of dominant cocharacters.  The product structure is most naturally described via the Satake transform; for more details, see \cite{henniartvigneras:satake}.

 \end{enumerate}
\label{subsubsec:JK}
 \end{example}

\begin{remark} When $k$ has characteristic $p$, we have, by diagram \eqref{f:CUV0diag}, Proposition \ref{prop:index} and 
Remark \ref{rema:section0UV}-\ref{rema:section0UV-v}:
\begin{equation}
\label{f:formulaCP}
C^*_\P(e_\P \tau_{w} e_\P)=C^0_\P(\tau_{w})= {[\P_w :I_w]}\tau_w^\P =\begin{cases} 0&\text{ if $w\notin \widetilde \DPP$ \ ,}
\cr[\P:I] \tau_w^\P&\text{ if $w\in \widetilde \DPP$ \ .}\end{cases}
\end{equation} 
Consequently, using Remark \ref{rema:section0UV}-\ref{rema:section0UV-iv}, we obtain:
\begin{equation}
\label{f:formulaRP}
R_{\P}^*(\tau_w^\P) = R^0_\P(\tau_w^\P)= {[\P :I]} e_\P\tau_{\tilde w} e_\P \quad \textnormal{where $w\in \DPP$ and $\tilde w\in \widetilde W$ is a lift of $w$.}
\end{equation}
\end{remark}

\begin{remark}
\label{rema:res-DPP} 
Suppose that $k$ is a field of characteristic $p$. By Remark \ref{rema:coressurj}, \eqref{f:defiCP*} and Proposition \ref{prop:index},
the restriction of  $C_{J, \P}^*: E_J^*\longrightarrow  E_\P^*$ (resp.  $C_{ \P}^*: E^*\longrightarrow  E_\P^*$) to 
$\bigoplus_{d \in  \DPP} H^*(J, \X_J(d )) \quad (\text{resp.}, ~\bigoplus_{d \in  \DPP} H^*(I, \X(\tilde{d})))$ is still surjective. 
\end{remark}

\subsubsection{Some Yoneda products in the Iwahori--Hecke $\Ext$-algebra}

In this paragraph we consider the case  of the parahoric subgroup $\P=J$.

\begin{remark}
\label{rema:E*}
Let  $w\in \widetilde W= N_G(T)/T^1$ (resp. $w\in  W= N_G(T)/T^0$) with lift $\dot w$ in $N_G(T)$. As with the double coset $IwI$ (resp., $J w J$) , the group $I_{\dot w}$  (resp., $J_{\dot{w}}$) does not depend on the choice of the lift, and we simply denote it by $I_w$ (resp., $J_w$).  Since $T^1 \subseteq I_w$ (resp., $T^0 \subseteq J_w$), the conjugation by an element of $T^1$ induces the trivial action on $H^i(I_w,k)$ (see \cite[Prop. 1.6.3]{NSW}). Therefore, by Lemma \ref{lemma:shapindep}, the Shapiro isomorphism
$
  \Sh^I_{\dot w} : H^*(I,\X(w)) \longrightarrow H^*(I_w,k)
$ (resp. $
  \Sh^J_{\dot w} : H^*(J,\X(w)) \longrightarrow H^*(J_w,k)
$)
does not depend on the lift chosen in $N_G(T)$.  Hence, we often simply denote it by $\Sh_w$ (resp. $\Sh_w^J$).
\end{remark}

Recall that \cite[Prop. 5.3]{Ext} makes explicit the product structure in $E^*$,  which for a general open compact subgroup $\U$ is given in Proposition \ref{yoneda-product-U} of the present article.  In particular, Corollary 5.5 and Proposition 5.6 of \textit{op. cit.} make certain Yoneda products more explicit in the case when $\U$ is the pro-$p$ Iwahori subgroup $I$. It is straightforward to check that \cite[Cor. 5.5]{Ext} holds \textit{mutatis mutandis} when $I$ is replaced by $J$.
We therefore get Proposition  \ref{coro:prod-goodlength} (we refer to Remark \ref{rema:E*}).

\begin{proposition}
\label{coro:prod-goodlength}
Let $v,w \in  W$ such that  $\ell(vw)=\ell(v)+\ell(w)$. For any  cohomology classes $\upalpha \in H^i(J,\mathbf{X}_J(v))$ and $\upbeta \in H^j(J,\mathbf{X}_J(w))$ we have $\upalpha\cdot \upbeta \in H^{i+j}(J, \X_J(vw))$, and
\begin{equation}\label{f:prod-goodlength}
 \upalpha\cdot \upbeta = (\upalpha\cdot \tau^J_w) \cup  (\tau^J_v\cdot \upbeta) \ ,
\end{equation}
where we use the cup product in the sense of Subsection \ref{subsec:cup}; moreover
\begin{equation}\label{f:prod-goodlength2}
  \Sh^J_{vw}(\upalpha\cdot \tau^J_w) =
 \res^{J _v}_{J_{vw}} \big(\Sh^J_v(\upalpha)
   \big)  \quad\textrm{ and }\quad  \Sh^J_{vw}( \tau^J_v\cdot \upbeta)   =
   \res^{ v J_w v^{-1}}_{J_{vw}} \big(v_*\Sh^J_w(\upbeta)\big)\ .
\end{equation}
\end{proposition}

\begin{remark}
\label{f:rightomega} 
We easily deduce the following. Let $\upalpha\in H^i(J, \X_J(w))$ with $w\in W$ and $i \geq 0$. For $\omega \in \Omega$, we have $	 \upalpha\cdot\tau^J_\omega \in H^i(I, \X_J( w\omega))$ and
\begin{equation}\label{f:rightomega}
  \Sh^J_{w\omega }(\upalpha	\cdot\tau^J_\omega )=\Sh^J_w(\upalpha) \ .
\end{equation}
\end{remark}

To state the next proposition (compare with \cite[Prop. 5.6]{Ext}), we refer to the notations introduced in \cite[\S 2.1.6, (14) and (15)]{Ext}.  Let $\mathbf{G}_{x_0}$ denote the Bruhat-Tits group scheme over $\mathfrak{O}$ corresponding to the hyperspecial vertex $x_0$. As part of a Chevalley basis we have (cf.\ \cite[\S 3.2]{BT2}), for any root $\alpha \in \Phi$, a homomorphism $\varphi_\alpha : {\rm SL}_2 \longrightarrow \mathbf{G}_{x_0}$ of $\mathfrak{O}$-group schemes which restricts to isomorphisms
\begin{equation*}
    \{ \begin{pmatrix}
    1 & \ast \\ 0 & 1
    \end{pmatrix}
    \}
    \stackrel{\sim}{\longrightarrow} {\EuScript U}_\alpha
    \qquad\text{and}\qquad
    \{ \begin{pmatrix}
    1 & 0 \\ \ast & 1
    \end{pmatrix}
    \}
    \stackrel{\sim}{\longrightarrow} {\EuScript U}_{-\alpha} \ ,
\end{equation*} where ${\EuScript U}_\alpha$ and ${\EuScript U}_{-\alpha}$ denote the root subgroups as introduced in \cite[\S 2.1.2]{Ext}.  We then define the additive isomorphism $x_\alpha : \mathfrak{F} \stackrel{\sim}{\longrightarrow} \EuScript{U}_\alpha$ by
$
  x_\alpha (u) := \varphi_\alpha( \left(
  \begin{smallmatrix}
    1 & u \\ 0 & 1
  \end{smallmatrix}
  \right) ) \ .
$ Moreover, note that we have $\check\alpha (x) = \varphi_\alpha( \bigl( \begin{smallmatrix}
x & 0 \\ 0 & x^{-1}
\end{smallmatrix}
\bigr) )$.

We let the subtorus $\mathbf{T}_{s_\alpha} \subseteq \mathbf{T}$ denote the image (in the sense of algebraic groups) of the cocharacter $\check\alpha$. We always view these as being defined over $\mathfrak{O}$ as subtori of $\mathbf{G}_{x_0}$. The group of $\mathbb{F}_q$-rational points $\mathbf{T}_{s_\alpha} (\mathbb{F}_q)$ can be viewed as a subgroup of $T ^0 / T ^1 \stackrel{\sim}{\longrightarrow} \mathbf{T}(\mathbb{F}_q)$ (and is abstractly isomorphic to $\mathbb{F}_q^\times$). Given $z\in \mathbb F^\times_q$, we consider $[z]\in\mathfrak O^\times$ the Teichm\"uller representative  and consider the  morphism of groups
\[ \check\alpha([_-])\::\:\mathbb F_q^\times\xrightarrow{[_-]} \mathfrak O^\times \stackrel{\check\alpha}{\longrightarrow} \mathbf{T}(\mathfrak{O}) = T^0 \ . \] 
For $(\alpha, \mathfrak h)\in \Pi_{\aff}$ and $s= s_{(\alpha,\mathfrak h)}$, we put
$    n_s := \varphi_\alpha (
    \begin{psmallmat}
    0 & \pi^\mathfrak{h} \\
    - \pi^{-\mathfrak{h}} & 0
    \end{psmallmat}
    ) \in N_G(T ) \ .$
We then have $n_s^2 = \check\alpha(-1) \in T ^0$ and $n_s T ^0 = s \in W $.

\begin{proposition}
\label{prop:explicitleftaction}  
Let $\upbeta\in H^j(J, \X_J(w))$ with $w\in W$ and $j\geq 0$. For $\omega \in \Omega$, we have $\tau^J_\omega	\cdot \upbeta\in H^j(J, \X_J(\omega w))$ and
\begin{equation}\label{f:leftomega}
  \Sh^J_{\omega w}(\tau^J_\omega	\cdot \upbeta)=\omega_*\Sh^J_w(\upbeta) \ .
\end{equation}
Assume  that  $q-1$ is invertible in $k$. For $ s\in S_{\aff}$ corresponding to the affine root $(\alpha, \mathfrak h)$,  we have

\begin{itemize}
\item  either $\ell( s w)=\ell(w)+1$
and $\tau^J_{ s} \cdot \upbeta\in H^j(J, \X_J( s w))$  with
\begin{equation}\label{f:leftsgood}
  \Sh^J_{  s w}(\tau^J_{  s}	\cdot \upbeta)=\res^{  s  J_w   s^{-1}}_{ J_{sw}} \big(  s_* \Sh^J_w(\upbeta)\big)\ ,
\end{equation}
\item  or  $\ell(  s w)=\ell(w)-1$ and   $\tau^J_{ s} \cdot \upbeta = \upgamma_w+\upgamma_{sw}\in H^j(J, \X_J(w)) \oplus H^j(J, \X_J( s w))$,
where 
\begin{equation}\label{gsw}
\Sh_{ s w}^J(\upgamma_{s w}) = \cores_{J _{ sw}}^{ s J _w s^{-1}}\big( s_* \Sh_w^J(\upbeta)\big)
\end{equation}
and  
\begin{equation}\label{gw}
\Sh_w^J(\upgamma_w)=\sum_{z\in \mathbb F_q^\times}  (n_s x_{\alpha}(\pi^{\mathfrak h}[z]) n_s^{-1})_* \Sh^J_w(\upbeta). 
\end{equation}
\end{itemize}
\end{proposition}

\begin{proof}
The results for $\omega \in \Omega$ and $s \in S_{\aff}$ which satisfies $\ell(sw) = \ell(w) + 1$ follow from Proposition \ref{coro:prod-goodlength} (in particular formula \eqref{f:prod-goodlength2}).  Thus, let us assume $\ell( s w) = \ell(w) - 1$, which implies $JsJwJ   \subseteq  JswJ  \sqcup JwJ$ (see Lemma \ref{coro25}\eqref{coro25-iii}).  From Proposition \ref{yoneda-product-U}, we obtain
\[ \tau^J_{ s} \cdot \upbeta = \upgamma_w + \upgamma_{sw} \in H^j(J, \X_J(w)) \oplus H^j(J, \X_J(sw)) \ \] 
and we want to compute
$\upgamma_w$ and  $\upgamma_{sw}$.
\noindent The proof  of   \eqref{gsw} is the same as the proof of \cite[(67)]{Ext} where one simply needs to replace $\tilde s $ by $s$ and $I$ by $J$.

We will now deduce \eqref{gw} from the proof of \cite[(68)]{Ext} using Remark \ref{rema:RPCP*}-\ref{rema:RPCP*-ii}, which says that the restrictions
\[ r^*:={\dfrac{R^*_{J}}{[J:I]} }\quad :E_J^*\longrightarrow  e_J  \cdot E^* \cdot e_{J} \quad \textnormal{ and } \quad c^*:={\dfrac{C^*_{J}}{[J:I]} }\quad:e_J\cdot E^* \cdot e_J \longrightarrow   E^*_J \]
are isomorphisms of $k$-algebras which are inverse to each other. We compute:
\begin{itemize}
\item By \eqref{f:formulaCP}, we have
\[ {{C^0_{J}}}(\tau_{n_s})={[J:I]} \tau^J_{s}, \quad \textnormal{so}\quad 
c^0(e_J\tau_{n_s} e_J)=\tau^J_{s} \textnormal{ and } r^0(\tau^J_{s})=e_J\tau_{n_s} \ . \]
\item Further, from \eqref{f:defiRP*} we deduce
\[ r^*(\upbeta)={\frac{1}{[J:I]} }\sum_{a\in {T^0/T^1}} \beta_ {a \tilde w} \ , \]
 where $\tilde w\in \widetilde W$ is a lift for $w$ and where $\beta_ {a \tilde w}\in H^j(I, \X(a\tilde w))$ satisfies \begin{equation}
 \label{shb}
 \Sh_{a\tilde w}^I(\beta_ {a \tilde w})=\res^{J_w}_{I_w} \big( a_* \Sh^J_w(\upbeta) \big) = \res^{J_w}_{I_w}   \big(\Sh^J_w(\upbeta) \big)
 \end{equation}
(because $a\in J_w$).  It is then easy to check  (using \cite[Cor. 5.5, (64)]{Ext}) that $r^*(\upbeta)= e_J\cdot \beta_{\tilde w} \cdot e_J$.
\end{itemize}
Next, we compute $\upgamma_w$. Using the above two bullet points, we have 
\[ \tau_s^J\cdot \upbeta= c^*( r^*(\tau_s^J)\cdot  r^*(\upbeta))=c^*(e_J\cdot( \tau_{n_s} \cdot \beta_{\tilde w}) \cdot e_J)\ . \]
By \cite[Prop. 5.6]{Ext}, we know that
\[ \tau_{n_s} \cdot \beta_{\tilde w} = \gamma_{n_s\tilde w} + \sum_{t\in \check\alpha([\mathbb F_q^\times])}\gamma_{t\tilde w} \]
where $\gamma_{t\tilde w} \in H^j(I, \X(t\tilde w))$, $\gamma_{n_s\tilde w} \in H^j(I, \X(n_s\tilde w))$ and 
\[ \Sh^I_{{t}  \tilde{w}}(\gamma_{{t} \tilde{w}})=  \sum_{z\in \mathbb F_q^\times,\: \check\alpha([z])=t}(n_s t^{-1} x_{\alpha}(\pi^{\mathfrak h}[z]) n_s^{-1})_*\Sh^I_{\tilde{w}}(\beta_{\tilde w})\ .\] 
(It is implicit in this formula (and verified \textit{op. cit.}) that $(n_s t^{-1} x_{\alpha}(\pi^{\mathfrak h}[z]) n_s^{-1})$ normalizes $I_{\tilde w}$.) It follows that $(n_s  x_{\alpha}(\pi^{\mathfrak h}[z]) n_s^{-1})$ normalizes $I_{\tilde w}$ and $J_w$. We have 
\[ \upgamma_{w}=c^*\Bigg(e_J\cdot\big( \sum_{t\in \check\alpha([\mathbb F_q^\times])}\gamma_{t\tilde w}\big)\cdot e_J\Bigg) \ . \]
Let us write $e_J\cdot (\sum_{t\in \check\alpha([\mathbb F_q^\times])}\gamma_{t\tilde w})\cdot e_J \ $  as $\sum_{u\in T^0/T^1} \delta_{u \tilde w}$ for some $\delta_{u\tilde w}\in H^j(I,\X(u\tilde w))$. Then 
\[ \Sh_{u\tilde w}^I(\delta_{u\tilde w}) =  \frac{1}{[J:I]^2}\sum_{a\in T^0/T^1,  t\in \check\alpha([\mathbb F_q^\times]) }a_*\Sh_{t\tilde w}^I (\gamma_{t\tilde w}) \]
(using \cite[Cor. 5.5, (64)]{Ext}), and therefore
\begin{eqnarray*}
\Sh_{u\tilde w}^I(\delta_{u\tilde w}) & = & \frac{1}{[J:I]^2} \sum_{a\in T^0/T^1,  t\in \check\alpha([\mathbb F_q^\times]) }\:
\sum_{z\in \mathbb F_q^\times,\: \check\alpha([z])=t}(a n_s t^{-1} x_{\alpha}(\pi^{\mathfrak h}[z]) n_s^{-1})_*\Sh^I_w(\beta_{\tilde w}) \\
& \stackrel{\eqref{shb}}{=} &  \frac{1}{[J:I]^2} \sum_{a\in T^0/T^1,  t\in \check\alpha([\mathbb F_q^\times]) }\:
\sum_{z\in \mathbb F_q^\times,\: \check\alpha([z])=t}(a n_s t^{-1} x_{\alpha}(\pi^{\mathfrak h}[z]) t n_s^{-1}a^{-1})_*\res^{J_w}_{I_w}  ( \Sh^J_w(\upbeta))
\end{eqnarray*}
where we have also used the fact that $a n_st^{-1} n_s^{-1}\in J_w= T^0I_w$.  Continuing:
\begin{eqnarray*}
\phantom{\Sh_{u\tilde w}^I(\delta_{u\tilde w})}  & = & \frac{1}{[J:I]^2} \sum_{t\in \check\alpha([\mathbb F_q^\times]), b\in T^0/T^1 }\:
\sum_{z\in \mathbb F_q^\times,\: \check\alpha([z])=t}(n_s b x_{\alpha}(\pi^{\mathfrak h}[z]) b^{-1} n_s^{-1})_*  \res^{J_w}_{I_w}   (\Sh^J_w(\upbeta)) \\
 & = & \frac{1}{[J:I]^2} \sum_{ b\in T^0/T^1 }\:
\sum_{z\in \mathbb F_q^\times}(n_s b x_{\alpha}(\pi^{\mathfrak h}[z]) b^{-1} n_s^{-1})_*  \res^{J_w}_{I_w}   (\Sh^J_w(\upbeta)) \\
 & = &  \frac{1}{[J:I]^2} \sum_{ b\in T^0/T^1 }\:
\sum_{z\in \mathbb F_q^\times}(n_s x_{\alpha}(\pi^{\mathfrak h}[z]\alpha(b)) n_s^{-1})_*  \res^{J_w}_{I_w}   (\Sh^J_w(\upbeta))\\ 
 & = & \frac{1}{[J:I]}\sum_{z\in \mathbb F_q^\times}(n_s x_{\alpha}(\pi^{\mathfrak h}[z]) n_s^{-1})_*  \res^{J_w}_{I_w}  (\Sh^J_w(\upbeta)) \\
 & = & \frac{1}{[J:I]}\sum_{z\in \mathbb F_q^\times}  \res^{J_w}_{I_w}  \left((n_s x_{\alpha}(\pi^{\mathfrak h}[z]) n_s^{-1})_* \Sh^J_w(\upbeta)\right) \ .
\end{eqnarray*}
Thus, using the above calculation we finally obtain
\begin{eqnarray*}
\Sh^J_w(\upgamma_w) & = & \Sh^J_w\Bigg(c^*\bigg(\sum_{u\in T^0/T^1} \delta_{u \tilde w}\bigg)\Bigg) \\
& = & \frac{1}{[J:I]}\sum_{u\in T^0/T^1}\Sh^J_w(C_J^*( \delta_{u \tilde w})) \\
& \stackrel{\eqref{f:defiCP*}}{=} & \frac{1}{[J:I]}\sum_{u\in T^0/T^1}\cores_{J_w}^{I_{\tilde w}}( \Sh^I_{u\tilde w}(\delta_{u \tilde w})) \\
 & = & \frac{1}{[J:I]^2}\sum_{u\in T^0/T^1,z\in \mathbb F_q^\times}\cores_{J_w}^{I_{\tilde w}}\Big(   \res^{J_w}_{I_w}  \big((n_s x_{\alpha}(\pi^{\mathfrak h}[z]) n_s^{-1})_* \Sh^J_w(\upbeta)\big)\Big) \\
 & \stackrel{\textnormal{Rmk.}~\ref{rema:coressurj}}{=} & \frac{1}{[J:I]}\sum_{u\in T^0/T^1,z\in \mathbb F_q^\times}  (n_s x_{\alpha}(\pi^{\mathfrak h}[z]) n_s^{-1})_* \Sh^J_w(\upbeta) \\
 & = & \sum_{z\in \mathbb F_q^\times}  (n_s x_{\alpha}(\pi^{\mathfrak h}[z]) n_s^{-1})_* \Sh^J_w(\upbeta) \ .
\end{eqnarray*}

\end{proof}

\begin{corollary}
\label{prop:explicitleftaction-cor}  
Suppose $q \neq 2,3$ and that $q - 1$ is invertible in $k$.  Let $\upbeta\in H^1(J, \X_J(w))$, and suppose $ s\in S_{\aff}$ satisfies $\ell(sw) = \ell(w) - 1$.  
\begin{enumerate}
\item We then have $\tau^J_{ s} \cdot \upbeta = \upgamma_w+\upgamma_{sw}\in H^1(J, \X_J(w)) \oplus H^1(J, \X_J( s w))$, where 
\begin{equation}\label{cor-gsw}
\Sh_{ s w}^J(\upgamma_{s w}) = \cores_{J _{ sw}}^{ s J _w s^{-1}}\big( s_* \Sh^J_w(\upbeta)\big)
\end{equation}
and  
\begin{equation}\label{cor-gw}
\Sh_w^J(\upgamma_w) = (q - 1)\Sh^J_w(\upbeta). 
\end{equation}
Thus, $\upgamma_w = (q - 1)\upbeta$.  
\item Suppose furthermore that $k$ is of characteristic $p$.  We then have $\tau^J_s \cdot \upbeta = -\upbeta$.  \label{prop:explicitleftaction-cor-2}
\end{enumerate}
\end{corollary}

\begin{proof}
\begin{enumerate}
\item We have 
\[ \Sh_w^J(\upbeta) \in H^1(J_w,k) = \Hom(J_w,k) = \Hom(J_w/[J_w,J_w], k) \ , \]
and therefore we identify $\Sh_w^J(\upbeta)$ with a homomorphism $J_w \longrightarrow k$.  By Lemma \ref{lem:comm-PY}, any such homomorphism is supported on the integral torus $T^0$.  If $s$ corresponds to the simple affine root $(\alpha, \mathfrak{h})$, then we have (for $z\in \bbF_q^\times$)
\begin{eqnarray*} 
\left((n_s x_\alpha(\pi^{\mathfrak{h}}[z])n_s^{-1})_*\Sh^J_w(\upbeta)\right)(t) & = & \Sh^J_w(\upbeta)\left(n_s x_\alpha(-\pi^{\mathfrak{h}}[z])n_s^{-1}tn_s x_\alpha(\pi^{\mathfrak{h}}[z])n_s^{-1}\right) \\
& = & \Sh^J_w(\upbeta)\left((n_s x_\alpha(-\pi^{\mathfrak{h}}[z])n_s^{-1}tn_s x_\alpha(\pi^{\mathfrak{h}}[z])n_s^{-1}t^{-1}) t\right) \\
& = & \Sh^J_w(\upbeta)(t) ,
\end{eqnarray*}
where the last equality follows from the fact that $(n_s x_\alpha(-\pi^{\mathfrak{h}}[z])n_s^{-1}tn_s x_\alpha(\pi^{\mathfrak{h}}[z])n_s^{-1}t^{-1})$ lies in $[J_w, J_w]$ (see the proof of Lemma \ref{lem:comm-PY}).  We conclude using the previous proposition.
\item By the above point, it suffices to prove the map $\cores_{J _{ sw}}^{ s J _w s^{-1}}: H^1(sJ_ws^{-1},k) \longrightarrow H^1(J_{sw},k)$ is identically zero.  By duality, it suffices to show that the transfer map on Frattini quotients $\textnormal{tr}:(J_{sw})_\Phi \longrightarrow (sJ_w s^{-1})_\Phi$ is zero (see \cite[\S 1.5.4, Prop. 1.5.9]{NSW}).  Furthermore, by Lemma \ref{cor:frattini-PY}, we may compute the transfer on elements $t \in T^1$.

By \cite[Lem. 2.2]{Ext} and Lemma \ref{coro25}\eqref{coro25-i}, the subgroup $J_{w^{-1}}$ has index $q$ in $J_{w^{-1}s^{-1}}$, so that $sJ_w s^{-1}$ has index $q$ in $J_{sw}$.  Moreover, by the discussion preceding \cite[Rmk. 2.10]{Ext}, we have that $\dot{w}^{-1}n_s^{-1}x_\alpha(b\pi^{\mathfrak{h}})n_s\dot{w}$ is a system of representatives for $sJ_ws^{-1}\setminus J_{sw}$, where $b$ ranges over a set $S \subseteq \mathfrak{O}$ of representatives of $\mathfrak{O}/\mathfrak{M}$ and $\dot{w} \in N_G(T)$ is a lift of $w$.  (As in the previous point, we let $(\alpha, \mathfrak{h})$ denote the simple affine root corresponding to $s$.)  For $t \in T^1$, we then obtain
\begin{eqnarray*}
 \textnormal{tr}(t) & \equiv &  \prod_{b\in S} \left(\dot{w}^{-1}n_s^{-1}x_\alpha(b\pi^{\mathfrak{h}})n_s\dot{w}\right)t\left(\dot{w}^{-1}n_s^{-1}x_\alpha(-b\pi^{\mathfrak{h}})n_s\dot{w}\right) \\
  & \equiv & \prod_{b\in S} t\cdot \left(\dot{w}^{-1}n_s^{-1}x_\alpha\left(b\pi^{\mathfrak{h}}(\check{\alpha}(t')-1)\right)n_s\dot{w}\right)~ \bmod \Phi(sJ_ws^{-1}) \ , 
 \end{eqnarray*}
where $t' := n_s\dot{w}t^{-1}\dot{w}^{-1}n_s^{-1} \in T^1$.  
This is a product of $t^q$ (which is a $p^{\textnormal{th}}$ power, hence lies in $\Phi(sJ_{w}s^{-1})$) and a matrix in $\dot{w}^{-1}n_s^{-1}x_\alpha(\pi^{\mathfrak{h} + 1}\mathfrak{O})n_s\dot{w}$ (since $\check{\alpha}(t') \in 1 + \mathfrak{M}$), which is contained in $sJ_w s^{-1}$.  By the proof of Lemma \ref{lem:comm-PY}, the latter also lies in $\Phi(sJ_ws^{-1})$, and therefore $\textnormal{tr}(t) = 0$.  
\end{enumerate}
\end{proof}

\section{\label{sec:SL2}Parahoric Hecke $\Ext$-algebras for ${\rm SL}_2$ in characteristic $p$}

In this section, we specialize further to the case $G={\rm SL}_2(\mathfrak F)$ in order to obtain even more precise information about the parahoric Hecke $\Ext$-algebras.  We maintain the notation introduced in previous sections.

We let $T \subseteq G$ be the torus of diagonal matrices. We choose the positive root with respect to $T$ to be $\alpha( \left(
\begin{smallmatrix}
t & 0 \\ 0 & t^{-1}
\end{smallmatrix}
\right) ) := t^2$, which corresponds to the Borel subgroup of upper triangular matrices. We then have
\[ I = \begin{pmatrix} 1+\mathfrak{M} & \mathfrak{O} \\ \mathfrak{M} & 1+\mathfrak{M} \end{pmatrix}
\qquad  \textnormal{and}  \qquad J= I T^0= \begin{pmatrix} \mathfrak{O}^\times & \mathfrak{O} \\ \mathfrak{M} & \mathfrak{O}^\times \end{pmatrix} \] 
(by abuse of notation, here and below, all matrices are understood to have determinant 1) and $K={\rm SL}_2(\mathfrak O)$.

Recall the short exact sequence 
\begin{equation}
\label{f:esW}
  0 \longrightarrow T^0/T^1 \longrightarrow \widetilde{W} = N(T)/T^1 \longrightarrow W = N(T)/T^0 \longrightarrow 0 \ .
\end{equation}
Additionally, we set
\[ s_0 := s_\alpha = \begin{pmatrix} 0 & 1 \\ -1 & 0 \end{pmatrix}, \qquad s_1 := s_{(-\alpha,1)} = \begin{pmatrix} 0 & -\pi^{-1} \\ \pi & 0 \end{pmatrix}, \qquad \theta := \begin{pmatrix} \pi & 0 \\ 0 & \pi^{-1} \end{pmatrix}, \qquad \varpi := \begin{pmatrix} 0 & 1 \\ \pi & 0 \end{pmatrix} \ .\]
We have $s_0, s_1, \theta \in \textnormal{SL}_2(\mathfrak{F})$ and $\varpi \in \textnormal{GL}_2(\mathfrak{F})$.  Moreover, note that $s_0 s_1 = \theta$, $s_1=\varpi  s_0  \varpi^{-1}$, and that the matrix $\varpi$ normalizes $I$ and $J$.  The first three are elements in $N_G(T)$ which we will often identify with their respective images in $W$ or in $\widetilde W$.  The images of $s_0$ and $s_1$ in $W$ are the two reflections corresponding to the two vertices of the standard edge fixed by $I$ in the Bruhat--Tits tree of $G$. They generate $W$, i.e., we have $W = \langle s_0,s_1 \rangle = \theta^{\mathbb{Z}} \sqcup s_0 \theta^{\mathbb{Z}}$. We have $
  \ell(\theta^i) = |2i|$ and $\ell(s_0\theta^i) = |1 - 2i|.$

\medskip

We introduce  the following subsets of $W$:
\[ {  W}^0:=\{w\in {  W}, \: \ell(s_0 w)=\ell(w)+1\} \qquad \text{and} \qquad {  W}^1:=\{w\in {  W}, \: \ell(s_1w)=\ell(w)+1\}. \]
Note that the intersection of these  two subsets is $\{1\}$.  Analogously to \cite[\S 3.2.1]{OS4}, we define for $\ell\geq 0$ the subgroups
\begin{equation}
\label{defiJn}
  J_\ell^+ := \begin{pmatrix} \mathfrak{O}^\times & \mathfrak{O} \\ \mathfrak M^{\ell +1} & \mathfrak{O}^\times \end{pmatrix} \qquad\text{and}\qquad J_\ell^- =\varpi J_\ell ^+ \varpi^{-1}= \varpi^{-1} J_\ell^+\varpi = \begin{pmatrix} \mathfrak{O}^\times & \mathfrak M^\ell \\ \mathfrak M & \mathfrak{O}^\times \end{pmatrix}
\end{equation}
of $J$  
 and recall that
\begin{equation}
\label{f:cap}
J_w=  J \cap wJw^{-1} =
  \begin{cases}
  J_{\ell(w)}^+ & \text{if $w\in {  W}^0$}, \\
  J_{\ell(w)}^- & \text{if $w\in {  W}^1$}.
  \end{cases}
\end{equation}

Notice that $\mathfrak W_K=\mathfrak W$ (see \eqref{f:WP}).
The system of representatives $\D_K$  of $W/\mathfrak W_K$ satisfying \eqref{representants} is given by 
\[ \lbrace 1, \: \theta^n, \: s_0 \theta^n, \text{ for } n\geq 1\rbrace\ . \]
Therefore, the minimal system of double coset representatives $_K\D_K$ as defined in Proposition \ref{prop:doublereps} is the set 
\begin{equation}
\label{KDK}
_K\D_K = \lbrace  1, \: s_0 \theta^n \text{ for } n\geq 1\rbrace\ .
\end{equation}  Compare with the Cartan decomposition  which gives
$W = \bigsqcup_{n \geq 0} \mathfrak{W}_K \theta^n \, \mathfrak{W}_K.$

\subsection{Parahoric Hecke algebras of ${\rm SL}_2$\label{subsec:parahSL2}}
From this point forward, we take $k$ to be a field of  characteristic $p$.  Recall that we simply denote by $H$ (instead of $H_I$) the pro-$p$ Iwahori--Hecke algebra over $k$.
The idempotent $e_J = - \sum_{t \in T^0/T^1 } \tau_t$ was introduced in Subsection \ref{pro-p-hecke-alg}; it is central in $H$. We may then describe the relations in $H$ as
\begin{alignat*}{2}
\tau_w\tau _v & = \tau_{wv} & & \textrm{for $w,v\in\widetilde W$ such that}\ \ell(w)+\ell(v)=\ell(wv), \\
 \tau_{s_i}^2 & = -e_J\tau_{s_i} \qquad & &\textrm{for $i=0,1$.}
\end{alignat*} 
We have the following descriptions of the parahoric Hecke algebras $H_J$ and $H_K$:
\begin{itemize}

\item The Iwahori--Hecke algebra $H_J$ was described in Example \ref{subsubsec:JK}\eqref{subsubsec:JK-1}. Recall that it  has basis  $\{\tau_w^J\}_{w\in W}$
with the relations
\label{f:relJ}
\begin{alignat*}{2}
\tau^J_w\tau^J _v & = \tau^J_{wv} & & \textrm{for $w,v\in W$ such that}\ \ell(w)+\ell(v)=\ell(wv), \\
 ({\tau^J_{s_i}})^2 & = -\tau^J_{s_i} \qquad & & \textrm{for $i=0,1$.}
\end{alignat*} 
 It identifies  with the algebra $e_J H e_J$ with unit $e_J$.  Let 
\begin{equation}
\zeta_J := (\tau^J_{s_0}+1)(\tau^J_{s_1}+1)+\tau^J_{s_1}\tau^J_{s_0}=(\tau^J_{s_1}+1)(\tau^J_{s_0}+1)+\tau^J_{s_0}\tau^J_{s_1}\label{f:zeta} \ .
\end{equation} 
The subalgebra $k[\zeta_J]$ of $H_J$ is a polynomial algebra; it is the center of $H_J$.  Finally, we define a decreasing filtration $(F^nH_J)_{n\geq 0}$ on $H_J$ by the $H$-bimodules 
\[ F^n H_J:= \bigoplus_{w\in  W, \, \ell(w)\geq n} k\tau^J_w. \]

\item The spherical Hecke algebra $H_K$ is described in Example \ref{subsubsec:JK}\eqref{subsubsec:JK-2}. It identifies  with  the algebra $e_{J,K}  H_J e_{J,K}$ with unit $e_{J,K}$, where
$e_{J,K}= \tau^J_{s_0}+1\in H_J\ .$
The composition
\begin{equation} 
\label{f:compacenter}
k[\zeta_J]\xrightarrow{z\mapsto e_{J,K}z } e_{J,K}k[\zeta_J] e_{J,K}=e_{J,K}H_J e_{J,K}\underset{\sim}{\xrightarrow{C_{J,K}}} H_K
\end{equation} 
is an isomorphism (\cite[Thm. 4.3]{Ollcompa}).  Denote by $T$ the image of $\zeta_J$ by this map. We verify below the following identity:
\begin{equation}
\label{tauT}
\tau^K_{s_0\theta^n}= T^n-T^{n-1}
\end{equation}
\begin{proof}
For $n\geq 1$,
recall
$ K_{s_0\theta^n}= s_0 K_{\theta^n} s_0^{-1}= J_{s_0\theta^n}\ .$ Therefore, using \eqref{f:defiCP*}, we have $C_{J,K}(\tau_{s_0\theta^n}^J)=\tau_{s_0\theta^n}^K$.
Now $C_{J,K}(e_{J,K}\zeta_J e_{J,K})=C_{J,K}((\tau^J_{s_0}+1)(\tau^J_{s_1}+1)(\tau^J_{s_0}+1))=
C_{J,K}(\tau^J_{s_1}+1)=
\tau_{s_1}^K+ 1=\tau^K_{s_0\theta}+1$ which proves the statement at $n=1$. Then proceed by induction
using $ \zeta_J \tau_{s_0\theta^n}^J = \tau_{s_0\theta^{n+1}}^J$.

\end{proof}

\end{itemize}

\subsection{Frattini quotients}

We now consider Frattini quotients of various compact open subgroups of $G$.  For the basic definitions, see Appendix \ref{appendix:frattini}. We assume in this subsection that $q\neq 2,3$.

\begin{lemma} 
\label{lem:frattini-J-SL2}
Assume $q\neq2,3$.
\begin{enumerate}
\item The Frattini quotient of $J$ is trivial.
\item For $w\in W$ with length $\ell(w)\geq 1$, we have an isomorphism of abelian groups
\begin{align}
\label{iso:frat} 
(1+\mathfrak{M}) \big/ (1+\mathfrak{M}^{\ell(w)+1})(1+\mathfrak{M})^p & \stackrel{\sim}\longrightarrow  (J_w)_\Phi \\
1+\pi x\bmod (1+\mathfrak{M}^{\ell(w)+1})(1+\mathfrak{M})^p & \longmapsto \begin{pmatrix} 1+\pi x & 0 \cr 0 & (1+\pi x)^{-1} \end{pmatrix} \bmod \Phi(J_w) .
\end{align}
\end{enumerate}\label{lemma:fratJ}
\end{lemma}

\begin{proof}
This follows from Corollary \ref{cor:frattini-Jw}.
\end{proof}

\begin{lemma}
\label{lemma:JKd}
Assume $q \neq 2,3$.
\begin{enumerate}
\item The Frattini quotient of $K$ is trivial.
\item For $x=\theta ^n$ or $x= s_0\theta ^n$ with $n\geq 1$, we have an isomorphism of abelian groups
\begin{align*} (1+\mathfrak{M}) \big/ (1+\mathfrak{M}^{2n})(1+\mathfrak{M})^p
&\overset{\cong}\longrightarrow  (K_x)_\Phi \cr
1+\pi x\bmod (1+\mathfrak{M}^{2n})(1+\mathfrak{M})^p&\longmapsto \begin{pmatrix} 1+\pi x & 0 \cr 0 & (1+\pi x)^{-1} \end{pmatrix} \bmod \Phi(K_x)\end{align*}
\end{enumerate}
\end{lemma}
\begin{proof}
This follows from Corollary \ref{cor:frattini-Klambda} (see also Corollary \ref{cor:frattini-K} for the first point).   Note that we can also obtain the second point from Lemma \ref{lem:frattini-J-SL2} using the identity
\begin{equation}
\label{JKd} 
K_{s_0\theta^n}= s_0 K_{\theta^n} s_0^{-1}= J_{s_0\theta^n}\ .
\end{equation}
\end{proof}

\begin{remark} \hfill
\label{kapparem}
\begin{enumerate}
\item There  exists  a minimal $L>1$ such that $1+\mathfrak M^L\subseteq (1+\mathfrak M)^p$.
\item If $\mathfrak{F}$ is unramified over $\bbQ_p$ and $p\neq 2$, we have $(1+\mathfrak M)^p= 1+\mathfrak M^2$ so $L=2$. 
\end{enumerate}
\end{remark}

\noindent Recall from \cite[\S 3.2.1]{OS4} the following description of the Frattini quotient of $I_w$ for $w\in \widetilde W$. When $w\in \widetilde {W^0}$:
\begin{align}
\label{f:I+ab}
 (I_{\ell(w)}^+)_\Phi&\xrightarrow{\; \sim \;}  \mathfrak{O}/\mathfrak{M} \times (1+\mathfrak{M}) \big/ (1+\mathfrak{M}^{\ell(w)+1})(1+\mathfrak{M})^p \times \mathfrak{O}/\mathfrak{M} \\
 \begin{pmatrix} 1+\pi x & y \cr \pi^{\ell(w)+1} z & 1+\pi t \end{pmatrix} \bmod \Phi(I_w)&\longmapsto (z\bmod \mathfrak M, \,1+\pi x\bmod (1+\mathfrak{M}^{\ell(w)+1})(1+\mathfrak{M})^p,\, y\bmod \mathfrak M)
\end{align} 
and when $w\in \widetilde {W^1}$:
\begin{align}
\label{f:I-ab}
 (I_{\ell(w)}^-)_\Phi&\xrightarrow{\; \sim \;}  \mathfrak{O}/\mathfrak{M} \times (1+\mathfrak{M}) \big/ (1+\mathfrak{M}^{\ell(w)+1})(1+\mathfrak{M})^p \times \mathfrak{O}/\mathfrak{M} \\
 \begin{pmatrix} 1+\pi x & \pi ^{\ell(w)} y \cr \pi z & 1+\pi t \end{pmatrix} \bmod \Phi(I_w)&\longmapsto (z\bmod \mathfrak M, \,1+\pi x\bmod (1+\mathfrak{M}^{\ell(w)+1})(1+\mathfrak{M})^p,\, y\bmod \mathfrak M).
 \end{align}

\begin{lemma}\label{lemma:coresIJw}
Let $w\in \widetilde W$. Via the isomorphisms \eqref{iso:frat} and \eqref{f:I+ab}, \eqref{f:I-ab} the transfer map  $(J_w)_\Phi\longrightarrow (I_w)_\Phi$ is given by:
\[ 1+\pi x \bmod (1+\mathfrak{M}^{\ell(w)+1})(1+\mathfrak{M})^p \longmapsto  (0, (1+\pi x)^{-1} \bmod (1+\mathfrak{M}^{\ell(w)+1})(1+\mathfrak{M})^p, 0) \ . \]
\end{lemma}

\begin{proof}
It follows immediately from the fact that a system of representatives of $J_w/I_w$ is given by $\{\begin{psmallmat} [a]&0\cr 0&[a^{-1}]\end{psmallmat}\}_{a\in (\mathfrak O/\mathfrak M)^\times}$.
\end{proof}

\subsection{Iwahori--Hecke $\Ext$-algebra}
\subsubsection{The first cohomology space}Here the field $\mathfrak F$ is arbitrary with the condition $q\neq 2,3$. 
Let $w\in W$.
Recall from Subsection \ref{subsec:Ext} that the Shapiro isomorphism gives
\[ H^1(J, \X_J(w))\cong H^1(J_w,k) = \Hom((J_w)_\Phi,k) \ , \] 
where $(J_w)_\Phi$ denotes the Frattini quotient of $J_w$.  By Lemma \ref{lemma:fratJ}, we have $H^1(J, \X_J(1))=0$ and for $w\neq 1$, 
$(1+\mathfrak{M}) \big/ (1+\mathfrak{M}^{\ell(w)+1})(1+\mathfrak{M})^p
\overset{\sim}\longrightarrow  (J_w)_\Phi  \ .$
Therefore, we will identify  an element of $H^1(J, \X_J(w))$ with a group homomorphism 
$1+\mathfrak M \longrightarrow k$ which is trivial on 
$(1+\mathfrak{M}^{\ell(w)+1})(1+\mathfrak{M})^p$. Given $w\in W$ and such a homomorphism $c^0$, we will denote by $(c^0)_w$ the corresponding element of $H^1(J, \X_J(w))$.

\begin{proposition}
\label{prop:theformulas}
Let $\epsilon \in \{0,1\}$, and choose $w\in  W$ with length $\geq 1$ and $(c^0)_w\in H^1(J, \X_J(w))$.  We have:
\begin{align*}
  \tau^J_{s_\epsilon} \cdot (c^0)_w =  
  \begin{cases}
 -(c^0)_{s_\epsilon w}  & \text{if $w \in {W}^\epsilon$,}\cr
  - (c^0)_{ w} & \text{if  $w \in {W}^{1-\epsilon}$}.\cr
  \end{cases}\quad \quad
 (c^0)_w   \cdot  \tau^J_{s_\epsilon}=  
  \begin{cases}
 (c^0)_{ w s_\epsilon}  & \text{if $w^{-1} \in {W}^\epsilon$,}\cr
  - (c^0)_{ w} & \text{if  $w^{-1} \in {W}^{1-\epsilon}$}.\cr
  \end{cases}
  \end{align*}
\end{proposition}

\begin{proof} 
Using the involution $\anti_J$ (see Subsection \ref{subsec-antiinvolution}) which satisfies $\anti_J((c^0)_w)=(c^0)_{w^{-1}}$ (resp., $(-c^0)_{w^{-1}}$) if $\ell(w)$ is even (resp., odd)  and $\anti_J(\tau_{s_\epsilon}^J) = \tau_{s_\epsilon}^J$ it is enough to check the first formula.
We write the proof for $\epsilon=0$ using Proposition \ref{prop:explicitleftaction}.  Recall that $\alpha$ denotes the corresponding simple root.
\begin{itemize}
\item Suppose first $\ell(s_0w)=\ell(w)+1$.  We use  \eqref{f:leftsgood} and the fact that the map 
\[ H^1(J_w,k) \xrightarrow{{s_0}_*}  H^1(s_0J_ws_0^{-1},k) \xrightarrow{\res}  H^1(J_{s_0w},k) \]
sends $(c^0)_w$ to $-(c^0)_{s_0w}$ since conjugation by $s_0$ maps $\begin{psmallmat} 1+\pi x & 0 \cr 0 & (1+\pi x)^{-1} \end{psmallmat} $ onto its inverse (and use \eqref{iso:frat}). 
\item  The case $\ell(s_0w)=\ell(w)-1$ follows directly from Corollary \ref{prop:explicitleftaction-cor}\eqref{prop:explicitleftaction-cor-2}.
\end{itemize}
\end{proof}

\begin{lemma}
\label{lemma:E1bimo}
Suppose that $\mathfrak F$ is unramified over $\mathbb Q_p$, and that $p\neq 2$ and $q\neq 3$.
As $H_J$-bimodules, we have 
\[ E_J^1\cong F^1 H_J\otimes _k \Hom(\mathfrak O/\mathfrak M, k) \ , \]
where the bimodule structure on the right-hand side comes from left tensor factor.
\end{lemma}

\begin{proof}  By Remark \ref{kapparem}, we have $(1+\mathfrak M)^p= 1+\mathfrak M^2$. 
Let $c^0\in  \Hom(1+\mathfrak M/1+\mathfrak M^2,k)$.
Define $x^J_{s_0}:=- (c^0)_{s_0}$ and $x^J_{s_1}:= (c^0)_{s_1}$. 
They satisfy the relations: 
\begin{itemize}
  \item $\tau^J_{s_i} \cdot  x^J_{s_i} = - x^J_{s_i} = x^J_{s_i}\cdot   \tau^J_{s_i}$ for $i \in \{0,1\}$;
  \item $\tau^J_{s_0} \cdot  x^J_{s_1} = x^J_{s_0} \cdot  \tau^J_{s_1}$ and $\tau^J_{s_1}\cdot   x^J_{s_0} = x^J_{s_1} \cdot  \tau^J_{s_0}$.
\end{itemize}
Therefore (compare with \cite[\S 3.7.1]{OS4}), we have a well-defined homomorphism of $H_J$-bimodules $F^1H_J \longrightarrow  E_J^1$ given by
\begin{align}
f_{(x^J_{s_0}, x^J_{s_1})}:\:\:F^1H_J\longrightarrow E^1_J, \quad  \tau_{s_0}^J\longmapsto x^J_{s_0}, \: \tau_{s_1}^J\longmapsto x^J_{s_1} \ .
\end{align}
For $w\in W$ with $\ell(w)\geq 1$,  it satisfies 
\[ f_{(x^J_{s_0}, x^J_{s_1})}(\tau^J_w)=\begin{cases} (c^0)_w&\text{ if $w\in W^0$}\cr  -(c^0)_w&\text{ if $w\in W^1$}\cr\end{cases} \ . \] 
(Compare with \cite[Prop. 3.18]{OS4}.) In particular, $f_{(x^J_{s_0}, x^J_{s_1})}$ is an injective homomorphism of $H_J$-bimodules.  Since $E^1_J$ is the direct sum of the images of such homomorphisms when $c^0$ ranges over a basis of the $k$-vector space $\Hom(1+\mathfrak M/1+\mathfrak M^2,k)\cong \Hom(\mathfrak O/\mathfrak M,k)$, we obtain the lemma.
\end{proof}

\begin{remark} Let $w\in W$ with lift $\tilde w\in \widetilde W$. Suppose $\ell(w)\geq 1$. As in \cite[\S 3.2.1]{OS4}, we may consider the element $(0, c^0, 0)_{\tilde{w}} \in H^1(I , \X(\tilde{w}))$. 
Via the Shapiro isomorphism $H^1(I , \X(\tilde w))\cong H^1(I_{\tilde w},k)$, this element maps the left hand side matrix in \eqref{f:I+ab} or \eqref{f:I-ab} onto $c^0(1+\pi x~ \bmod (1+\mathfrak{M}^{\ell(w)+1})(1+\mathfrak{M})^p)$.
By \eqref{f:defiCP*} and Lemma \ref{lemma:coresIJw}, we have
$ C^1_{J}((0,c^0,0)_{\tilde w})=-(c^0)_{w}$. Then by Remark \ref{rema:RPCP*}-\ref{rema:RPCP*-ii} and \ref{rema:RPCP*-iii}, we have 
\begin{equation}\label{f:Rc0}
R^1_J((c^0)_{w})=- e_J\cdot (0,c^0,0)_{\tilde w} \cdot e_J\ .
\end{equation}

\end{remark}

\subsubsection{The top cohomology space\label{sec:Ed}} 
We suppose here that $I$ is torsion-free, so that it is a Poincar\'e group of dimension $d = \dim_{\bbQ_p}(\textnormal{SL}_2(\mathfrak{F})) = 3[\mathfrak{F}:\mathbb{Q}_p]$. (In particular, this holds if $p > 2e + 1$, where $e$ denotes the ramification degree of $\mathfrak{F}$ over $\bbQ_p$; see \cite[\S III.3.2.7, Eqn. (3.2.7.5)]{lazard}.)
By Lemma \ref{lemma:duality} and by applying $\anti_J$, we obtain
\[ \Delta_J^d:E_J^d\stackrel{\sim}{\longrightarrow} ({}^{\anti_J}{(H_J)}^{\anti_J})^{\vee, f}\cong (H_J)^{\vee, f} \ . \] 
Given $w\in W$, we denote by $\tau^{J,*}_w$ the element in ${(H_J)}^{\vee, f}$ such that $\tau_w^{J,*}(\tau_v^J)=\delta_{v,w}$ for any $v\in W$.  A basis for $E^d_J$ is given by $\{\upphi^J_w\}_{w\in W}$ where $\upphi^J_w$ is the preimage of $\tau^{J,*}_w\in  ({}^{\anti_J}{(H_J)}^{\anti_J})^{\vee, f}$ under the above isomorphism $\Delta_J^d$.  In particular, for $w\in W$, the element $\upphi^J_w\in H^d(J, \X_J(w))$ is determined by the condition $\textnormal{tr}_J\circ \trace_J^d(\upphi_w^J)=1$ (compare with \cite[\S 8]{Ext} and equation \eqref{eqn:topcohbasis}).

Using the isomorphism $\Delta_J^d$, we see that the action of $\tau^J_{s}$ on these basis elements (for $s\in\{s_0,s_1\}$) is given by:
\begin{equation}
\label{f:leftrightHd}
 \upphi^J_w\cdot \tau^J_{s}= \begin{cases}  \upphi^J_{w s}-\upphi^J_w & \text{ if $\ell(w s)=\ell(w)-1$,}\cr 0& \text{ if $\ell(w s)=\ell(w)+1$,}\end{cases}
\quad \quad \tau^J_{s}\cdot \upphi^J_w= \begin{cases}\upphi^J_{s w}- \upphi^J_w& \text{ if $\ell( sw)=\ell(w)-1$,}\cr 0& \text{ if $\ell(sw)=\ell(w)+1$.}\end{cases}
\end{equation}
For  $w\in W$ satisfying $\ell(w)\geq 1$, we define 
\[ \uppsi^J_w:=\tau^J_{s_{1-\epsilon}}\cdot \upphi_w^J \ , \] 
where $w\in W^\epsilon$.  Recall the isomorphism of $H_J$-bimodules $E^d_J{\cong} \chi_{\triv}^{J} \oplus  \ker(\trace_J^d)\ $ of Proposition \ref{prop:decEd}. 
By arguments analogous to \cite[Rmk. 2.16]{OS4}, a basis 
 of the vector space $  \ker(\trace_J^d)$  is given by  $\{{\uppsi^J_w}\}_{w\in W, \ell(w)\geq 1}$ while $\upphi_1^J$ supports the character $\chi_{\triv}^J$.  Further, as $H_J$-bimodules, we have $\ker(\trace_J^d)\cong \bigcup_{n\geq 1} (H_J/\zeta_J^n H_J) ^\vee$ (the argument is the same as \cite[Prop. 2.4]{OS4}). Thus, we obtain
 \[ E_J^d=k\upphi^J_1\oplus \bigoplus_{w\in W, \ell(w)\geq 1} k\uppsi^J_w\cong \chi_{\triv}^J\oplus
 \bigcup_{n\geq 1} (H_J/\zeta_J^n H_J) ^\vee\ . \]

\subsubsection{Center of the  Iwahori--Hecke $\Ext$-algebra of $\textnormal{SL}_2(\bbQ_p)$, $p\geq 5$ \label{subsubsec:EJ}}
In this subsection, we take $\mathfrak F=\mathbb Q_p$ with $p\geq 5$.  This assumption guarantees that $J$ is $p$-torsion-free.  We begin by describing the $H_J$-bimodule structure on $E_J^*$.

\begin{proposition}\label{prop:fullbimo}
The Iwahori--Hecke $\Ext$-algebra $E_J^*$ is supported in degrees $0$ to $3$.  We have $E_J^0= H_J$ and, as $H_J$-bimodules:
\begin{itemize}
\item $E_J^1\cong F^1 H_J$,
\item $E_J^2\cong (F^1 H_J)^{\vee, f}\cong \bigcup_{n\geq 1}(F^1H_J/\zeta_J^n F^1H_J)^\vee$,
\item $E_J^3\cong ({H_J})^{\vee, f}\cong  \chi^J_{\triv}\oplus \bigcup_{n\geq 1} (H_J/\zeta_J^n H_J) ^\vee $,
\end{itemize}
where the finite duals are defined as in \eqref{f:finitedual} with respect to  $H_J = \bigoplus_{w\in W} k \tau^J_w$ and  $F^1H_J = \bigoplus_{w\in W, \ell(w)\geq 1} k \tau^J_w$.
\end{proposition}

\begin{proof} See Lemma \ref{lemma:E1bimo} and Subsection \ref{sec:Ed}. 
We fix a basis $\c^0$ of the $k$-vector space  $\Hom(1+p\mathbb Z_p/1+p^2\mathbb Z_p,k)$ and consider the  associated elements 
$\x^J_{s_0}:=- (\c^0)_{s_0}$ and $\x^J_{s_1}:= (\c^0)_{s_1}$. It gives an associated isomorphism of $H_J$-bimodules 
\begin{align}
\label{f:f}
f_{(\x^J_{s_0}, \x^J_{s_1})}:\:\:F^1H_J\longrightarrow E^1_J, \quad  \tau_{s_0}^J\longmapsto \x^J_{s_0}, \: \tau_{s_1}^J\longmapsto \x^J_{s_1} \ .
\end{align} 
as in the proof of Lemma \ref{lemma:E1bimo}.  For $E^2_J$, notice that $f_{(\x^J_{s_0}, \x^J_{s_1})}$ preserves the decompositions of $F^1H_J$ and $E^1_J$ with respect to the natural bases indexed by $\{w \in W: \ell(w) \geq 1\}$. Therefore, the isomorphism of $H_J$-bimodules $E_J^2\cong ({}^{\anti_J}(F^1 H_J)^{\anti_J})^{\vee, f}\cong  ({F^1 H_J})^{\vee, f}$ is a direct consequence of  Lemma  \ref{lemma:duality} (and where the second isomorphism follows from the isomorphism of $H_J$-bimodules ${}^{\anti_J}(F^1H_J)^{\anti_J} \stackrel{\sim}{\longrightarrow} F^1H_J$ obtained by applying $\anti_J$).
We then notice that $\zeta_J ^n$ maps  $F^1H_J$ to $F^{2n+1} H_J$, hence the second isomorphism  for $E^2_J$ in the lemma.
\end{proof}

In order to calculate the center of $E_J^*$, we use introduce the following notation.

\begin{notn}[\textbf{Basis for $E^*_J$}]\label{not:basis}\hfill
\begin{itemize}
\item A basis for $E^0$ is given by  $\{\tau^J_w\}_{w\in W}$.
\item A basis for $E^3$ is given by  $\{\upphi^J_w\}_{w\in W}$ or $\{\upphi^J_1\}\cup\{\uppsi^J_w\}_{w\in W, \ell(w)\geq 1}$ (see Subsection \ref{sec:Ed}).
\item A basis for $E^1_J$ is given by  $\{\x^J_w\}_{w\in W, \ell(w)\geq 1}$ where $\x^J_w:= f_{(\x^J_{s_0}, \x^J_{s_1})}(\tau^J_w)$ (see the proof of Proposition \ref{prop:fullbimo}).
\item A basis for  $E^2_J$ is given by  $\{\upalpha^J_w\}_{w\in W, \ell(w)\geq 1}$ where $\upalpha^J_w$ is the  preimage of $\tau^{J,*}_w\vert _{F^1H}\in  (F^1 H_J)^{\vee, f}$ in the isomorphism $E_J^2\cong ({}^{\anti_J}(F^1 H_J)^{\anti_J})^{\vee, f}$.
It is the unique element in $H^2(I, \X_J(w))$ such that  $\upalpha^J_w\cup \x^J_w=\upphi^J_w$.
For $s\in \{s_0, s_1\}$, the action of $\tau_s^J$ on these basis elements is given by
\begin{equation}\label{f:leftrightd-1}
 \upalpha^J_w\cdot \tau^J_{s}= \begin{cases}  \upalpha^J_{w s}-\upalpha^J_w & \text{ if $\ell(w s)=\ell(w)-1$, $\ell(w)\geq 2$,}\cr
-\upalpha^J_w & \text{ if $\ell(w s)=\ell(w)-1$, $\ell(w)= 1$,}\cr
  0& \text{ if $\ell(w s)=\ell(w)+1$,}\end{cases}
\quad \quad \tau^J_{s}\cdot \upalpha^J_w= \begin{cases}\upalpha^J_{s w}- \upalpha^J_w& \text{ if $\ell( sw)=\ell(w)-1$, $\ell(w)\geq 2$,}\cr - \upalpha^J_w& \text{ if $\ell( sw)=\ell(w)-1$, $\ell(w)=1$,}\cr 0& \text{ if $\ell( sw)=\ell(w)+1$.}\end{cases}
\end{equation}
\item For $w\in W$ with $\ell(w)\geq 1$, we let \begin{equation}\upbeta^J_w:= \tau^J_{s_{1-\epsilon}} \cdot \upalpha^J_w\text{ where $w\in W^\epsilon$}.\label{defiB}\end{equation}
The set $\{\upbeta^J_w\}_{w\in W, \ell(w)\geq 1}$ is also a basis for $E^2_J$.
\end{itemize}
\end{notn}

We now describe the product structure relative to the above basis.

\begin{proposition}\hfill
\begin{enumerate}
\item The product $E_J^1\otimes_k  E_J^1\longrightarrow E_J^2$:  for $v,w\in W$ with $\ell(v), \ell(w)\geq 1$, the product of two elements $\x_v^J$ and $\x_w^J$ is zero unless $v=w=s_0$ or $v=w=s_1$, in which case we have
\begin{equation}\label{f:prod11}
\x^J_{s_0}\cdot  \x^J_{s_0}= -\upalpha^J_{s_0}\quad\quad \x^J_{s_1}\cdot  \x^J_{s_1}=-\upalpha^J_{s_1}
\end{equation}
\item The products $E_J^1\otimes_k  E_J^2\longrightarrow E_J^3$ and $E_J^2\otimes_k  E_J^1\longrightarrow E_J^3$: since $E_J^1$ is generated on the left and on the right by $\x_{s_0}^J$ and $\x_{s_1}^J$, the products are entirely described by the following, where $\epsilon\in\{0,1\}$:
\begin{align}\label{f:prod12}
{\bf x}^J_{s_\epsilon}\cdot \upalpha^J_ w & =-\tau^J_{s_\epsilon}\cdot \upphi^J_w\ =
\begin{cases}
-\upphi^J_{s_\epsilon w}+ \upphi^J_w&\text{if $w\in W^{1-\epsilon}$,}\cr
0&\text{if $w\in W^{\epsilon}$,}
\end{cases} \\
\upalpha^J_ w\cdot{\bf x}^J_{s_\epsilon} & = -\upphi^J_w\cdot \tau^J_{s_\epsilon} =
\begin{cases}
-\upphi^J_{ ws_\epsilon}+ \upphi^J_w&\text{if $w^{-1}\in W^{1-\epsilon}$,}\cr
0&\text{if $w^{-1}\in W^{\epsilon}$ \ .} \notag
\end{cases}
\end{align}
\end{enumerate}
\end{proposition}

\begin{proof}
 We want to use some of  the calculations  in \cite[\S 8.1, 8.2]{OS4} and for this we need to compute the image under $R_J^*: E^*_J\longrightarrow E^*$ of the basis elements of $E^*_J$ given above.  Recall that the restricted maps
\[ E_J^*\xrightarrow{-R^*_{J}}  e_{J} \cdot \, E^* \cdot  e_{J} \qquad \text{ and }\qquad e_{J}  \cdot  E^*  \cdot  e_{J} \xrightarrow{-C^*_{J} }   E^* _J  \]
are homomorphisms of unital $k$-algebras which are inverse to each other.
\begin{itemize}
\item Let $w\in W$ with lift $\tilde w\in \widetilde W$. Recall from \eqref{f:formulaCP} that  $C_J(\tau_{\tilde w})=-\tau_w^J $ so $C_J(e_J\tau_{\tilde w})=-\tau_w^J $  and $R_J(\tau_w^J)= -e_J \tau_{\tilde w}$.
\item Let $w\in W$ with lift $\tilde w\in \widetilde W$.  By  the commutativity of the two right hand squares in \eqref{f:dualitycompa},  we have $C_J^3(\phi_{\tilde w})=\upphi_w^J$  where  $\phi_{\tilde w}$ is introduced in \cite[\S 8]{Ext} as the unique element in $H^3(I, \X(\tilde w))$ such that $\textnormal{tr}_I\circ \trace_I^3(\phi_{\tilde w})=1$. Hence by Remark \ref{rema:RPCP*}-\ref{rema:RPCP*-ii} and \ref{rema:RPCP*-iii}, we have $R^*_J(\upphi_w^J)= e_J  \cdot \phi_{\tilde w}  \cdot e_J= e_J  \cdot \phi_{\tilde w}$.
\item  We consider the elements  $(\x_{w})_{w\in \widetilde W,\ell(w)\geq 1}$ given by $\x_{w}:= f_{(\x_{0}, \x_{1})}(\tau_w)$ with the notation of \cite[Prop. 3.18]{OS4}.  For $w\in W$ with $\ell(w)\geq 1$, we see that $R^*_J(\x_w^J)= -e_J \cdot  \x_{\tilde w}  \cdot  e_J= -e_J \cdot  \x_{\tilde w}$ for $\tilde w\in \widetilde W$ a lift of $w$ (using \eqref{f:Rc0}).
\item Consider, for $\tilde w\in \widetilde W$ with $\ell(w)\geq 1$, the element $\upalpha_{\tilde w}^\star\in H^2(I, \X(\tilde{w}))$ as introduced in \cite[\S8]{OS4}. By Remark 8.3 of \textit{op. cit.}, it satisfies $\upalpha^\star_{\tilde w}\cup \x_{\tilde w}=\phi_{\tilde w}$.

Let $w\in W$, $\ell(w)\geq 1$ with lift $\tilde w\in\widetilde W$. We have $C_J^*( e_J  \cdot \upalpha_{\tilde w}^\star  \cdot e_J)=\upalpha_w^J$ because
\begin{flushleft}
$\displaystyle{C_J^*( e_J  \cdot  \upalpha_{\tilde w}^\star  \cdot  e_J)\cup \x_w^J \quad\stackrel{ \eqref{f:projform}}{=}\quad C_J^*( e_J  \cdot \upalpha_{\tilde w}^\star  \cdot  e_J\cup R_J^*( \x_w^J)) \quad = \quad -C_J^*( e_J  \cdot \upalpha^\star_{\tilde w}  \cdot e_J\cup e_J  \cdot \x_{\tilde w}  \cdot e_J)
}$
\end{flushleft}
\begin{flushright}
$\displaystyle{\stackrel{\text{\cite[Prop. 7.18, (85)]{Ext}}}{=}-C^*_J( \upalpha^\star_{\tilde w} \cup e_J  \cdot \x_{\tilde w}  \cdot e_J) \quad \stackrel{\text{ \eqref{f:orth}}}{=} \quad C^*_J( \upalpha^\star_{\tilde w} \cup \x_{\tilde w}) \quad = \quad C^*_J(\phi_{\tilde w}) \quad = \quad \upphi_w^J\ .}$
\end{flushright}
\end{itemize}

Now we verify the product formulas \eqref{f:prod11} and \eqref{f:prod12}. 
\begin{itemize}
\item  Let $v,w\in W$ with $\ell(w), \ell(w)\geq 1$ and lifts $\tilde v, \tilde w\in \widetilde W$.  We have 
\[ \x_v^J\cdot \x_w^J= C_J^* R_J^*(\x_v^J\cdot \x_w^J)=-C_J^*(R_J^*(\x_v^J)\cdot R_J^*( \x_w^J))=
-C_J^*(e_J\cdot \x_{\tilde v} \cdot \x_{\tilde w} \cdot e_J)\ . \]
By \cite[Rmk. 8.4]{OS4}, this product is zero unless $v=w=s_\epsilon$ for $\epsilon\in \{0,1\}$ and 
$e_J\cdot \x_{s_\epsilon} \cdot \x_{s_\epsilon}\cdot e_J= 
e_J\cdot \upalpha^\star_{s_\epsilon}\cdot e_J$.
So $\x_{s_\epsilon}^J\cdot \x_{s_\epsilon}^J=-C_J^*(e_J\cdot \upalpha^\star_{s_\epsilon}\cdot e_J)=-\upalpha_{s_\epsilon}^J.$
\item We have 
\[ \x^J_{s_\epsilon}\cdot \upalpha_w^J= C_J^* R_J^*(\x^J_{s_\epsilon}\cdot \upalpha_w^J)=- C_J^* (R_J^*(\x^J_{s_\epsilon})\cdot R_J^*( \upalpha_w^J))=C_J^* (e_J\cdot \x_{s_\epsilon}\cdot  \upalpha^\star_{\tilde w}\cdot e_J)= C_J^*(-e_J\cdot \tau_{s_\epsilon}\cdot  \phi_{\tilde w}) \]
using \cite[Prop. 8.6]{OS4}. So $\x^J_{s_\epsilon}\cdot \upalpha_w^J=C_J^*(e_J\cdot \tau_{s_\epsilon}) C_J^*(e_J\cdot  \phi_{\tilde w})=-\tau_{s_\epsilon}^J\cdot \upphi_w^J$. The computation of $ \upalpha_w^J\cdot \x^J_{s_\epsilon}$ is similar.
\end{itemize}
\end{proof}

\begin{definition}\label{defZ}
Recall  Notation \ref{not:basis}. We define the following linear subspace $\mathfrak Z_J^*$ of $E^*_J$: 
\begin{itemize}
\item $\mathfrak Z_J^0=k[\zeta_J]$
\item $\mathfrak Z_J^1= f_{(\x^J_{s_0}, \x^J_{s_1})}((\zeta_J-1)k[\zeta_J])$
\item $\mathfrak Z_J^2= \langle \upbeta^J_{s_0(s_1s_0)^n}+\upbeta^J_{s_1(s_0s_1)^n}, \: n\geq 0\rangle_k$
\item $\mathfrak Z_J^3= \langle \upphi_1^J, \uppsi^J_{s_0(s_1s_0)^n}+\uppsi^J_{s_1(s_0s_1)^n}, \: n\geq 0\rangle_k$
\end{itemize}
\end{definition}

We now consider the center of $E^*_J$ (as an abstract $k$-algebra), namely
\begin{equation}\label{dcenter}
\mathcal Z(E^*_J):=\{z\in E^*_J, \:  z\cdot x=x\cdot z\text{ for any $x\in E^*_J$} \} \ .
\end{equation}

\begin{lemma}
\label{lemma:Z}
The vector space $\mathfrak  Z_J^*$ is a  subalgebra of $\mathcal Z(E^*_J)$.  It is a commutative, but not graded-commutative, algebra. 
\end{lemma}

\begin{proof} By Proposition \ref{prop:fullbimo}, we know that $\zeta_J\in E^0_J$ centralizes $E^*_J$.
Using Subsection \ref{f:relJ}, \eqref{f:leftrightHd},  \eqref{f:leftrightd-1} and the description of the product in $E^*_J$ (see in particular \eqref{f:prod11} and \eqref{f:prod12}), one checks that $\mathfrak  Z_J^*$ is central, that it is stable by the Yoneda product and that it is not graded commutative. We give some formulas for this verification:
\begin{itemize}
\item  $\mathfrak Z_J^*$ is stable by the action of $\zeta_J$ on the left.  For example, one checks that
$\zeta_J\cdot \upbeta^J_{s_0(s_1s_0)^n}=\upbeta^J_{s_0(s_1s_0)^{n-1}}$ if $n\geq 1$ and it is zero if $n=0$, while $\zeta_J\cdot \upphi^J_{s_0(s_1s_0)^n}=\upphi^J_{s_0(s_1s_0)^{n-1}}$ if $n\geq 1$ and it is equal to $\upphi^J_1$ if $n=0$. 
\item  We have
\[ f_{(\x_{s_0}, \x_{s_1})}((\zeta_J-1))\cdot f_{(\x_{s_0}, \x_{s_1})}((\zeta_J-1))=\x_{s_0}^J\cdot \x_{s_0}^J+\x_{s_1}^J\cdot \x_{s_1}^J =-\upalpha^J_{s_0}-\upalpha^J_{s_1} = \upbeta^J_{s_0}+\upbeta^J_{s_1} \ . \]
This is a nonzero square in $\mathfrak Z_J^2$, and therefore $\mathfrak Z_J^*$ is not graded-commutative.
\item  If $n\geq 1$,  then $\x^J_{s_0}\cdot (\tau^J_{s_1}+1)\cdot \upbeta_{s_0(s_1s_0)^n}^J=-\uppsi^J_{s_0(s_1s_0)^{n-1}}+\uppsi^J_{s_0(s_1s_0)^n}$ and   $\x^J_{s_1}\cdot (\tau^J_{s_0}+1)\cdot  \upbeta_{s_0(s_1s_0)^n}^J=0$. We have a symmetric result for $ \upbeta_{s_1(s_0s_1)^n}^J$. Thus, we get
\[f_{(\x_{s_0}, \x_{s_1})}(\zeta_J-1)\cdot (\upbeta_{s_0(s_1s_0)^n}^J+\upbeta_{s_1(s_0s_1)^n}^J)= -(\uppsi^J_{s_0(s_1s_0)^{n-1}}+\uppsi^J_{s_1(s_0s_1)^{n-1}})+(\uppsi^J_{s_0(s_1s_0)^n}+\uppsi^J_{s_1(s_0s_1)^{n}}) \ . \]
When $n = 0$, we have $f_{(\x_{s_0}, \x_{s_1})}(\zeta_J-1)\cdot (\upbeta_{s_0}^J+\upbeta_{s_1}^J)= \uppsi^J_{s_0}+\uppsi^J_{s_1}$ since in that case $\x^J_{s_0}\cdot (\tau^J_{s_1}+1)\cdot \upbeta_{s_0}^J=\uppsi^J_{s_0}$ and $\x^J_{s_1}\cdot (\tau^J_{s_0}+1)\cdot \upbeta_{s_1}^J=\uppsi^J_{s_1}$.
\end{itemize}
\end{proof}

\begin{proposition}  \label{proo:fullZ}
As a vector space we have
 $\mathcal Z(E^*_J)= \mathfrak Z_J^*\oplus k \uppsi_{s_0}^J\oplus k \upbeta_{s_0}^J$.
\end{proposition}

\begin{proof} Using Lemma  \ref{lemma:Z},  formulas \eqref{f:leftrightHd},  \eqref{f:leftrightd-1}, and the description of the product in $E^*_J$ (\eqref{f:prod11} and \eqref{f:prod12}), we see that the right hand side is contained in the center.  For the other inclusion, the fact that the degree $0$ graded piece of $\mathcal Z(E^*_J)$ is contained in the center $k[\zeta_J]$ of $H_J$ is clear. The degree $1$ graded piece  of $\mathcal Z(E^*_J)$ is contained in $f_{(\x_{s_0}, \x_{s_1})}((\zeta_J-1)k[\zeta_J])$ because $F^1H_J\cap k[\zeta_J]=(\zeta_J-1)k[\zeta_J]$.
Finally, using \eqref{f:leftrightHd} and \eqref{f:leftrightd-1}, it is easy to check that an element $\sum_{w\in W, \ell(w)\geq 2} \lambda_w \upbeta^J_w$ (resp., $\sum_{w\in W, \ell(w)\geq 2} \lambda_w \uppsi^J_w$) commutes with the action of $H_J$ if and only if it lies in  $\mathfrak Z_J^2$ (resp.  $\mathfrak Z_J^3$). 
\end{proof}

We can now use the above results to deduce finite generation properties.

\begin{proposition}
	 The algebra $E_J^*$ is finitely generated as a module over $\mathcal{Z}(E_J^*)$. 
\end{proposition}

\begin{proof} 
We prove that a set of generators of $E_J^*$ as a module over $\mathcal{Z}(E_J^*)$ is given by 
	\[  S =  S^0\cup S^1 \text{ where } S^0=\{1, \tau_{s_0}^J, \tau_{s_1}^J, \tau_{s_0s_1}^J\} \text{ and } S^1=\{\x_{s_0}^J, \x_{s_1}^J, \x_{s_0s_1}^J, \x_{s_1s_0}^J \} \]

It is known that $H_J$ is free  as a module over its center $\mathfrak Z_J^0$ with basis $S^0$ (compare with \cite[Cor. 3.4]{OS2}). This uses the following:
for $n \geq 0$ and $\epsilon \in \{0, 1\}$,  we have $\tau_{s_\epsilon}^J \cdot \zeta_J^n = \tau_{s_\epsilon(s_{1-\epsilon} s_\epsilon)^n}^J$
Using this identity along with the $H_J$-bimodule isomorphism \eqref{f:f}, we  also get that $E^1_J$ is generated, over $\mathfrak{Z}_J^0$ hence over $\mathcal{Z}(E_J^*)$, by $S^1$.

Next, for $n \geq 1$, we compute
\begin{align*}
	\tau_{s_\epsilon}^J \cdot \left( \upbeta_{s_0(s_1s_0)^n}^J + \upbeta_{s_1(s_0s_1)^n}^J \right)  &= -\upbeta_{s_\epsilon(s_{1-\epsilon}s_\epsilon)^n}^J + \upbeta_{(s_\epsilon s_{1-\epsilon})^n}^J, \\
	\tau_{s_{\epsilon} s_{1-\epsilon}}^J \cdot \left( \upbeta_{s_0(s_1s_0)^n}^J + \upbeta_{s_1(s_0s_1)^n}^J \right)  &= -\upbeta_{(s_\epsilon s_{1-\epsilon})^n}^J + \upbeta_{s_\epsilon (s_{1-\epsilon} s_\epsilon)^{n-1}}^J.
\end{align*}
Since $\upbeta_{s_0}^J$ and $\upbeta_{s_1}^J$ already lie in the center, it follows by induction on $\ell(w)$ that $\upbeta_w^J $ lies in the $\mathcal{Z}(E_J^*)$-module generated by $S^0$ for any $w\in W$, $\ell(w)\geq 1$. So  $E^2_J$ is contained  in the $\mathcal{Z}(E_J^*)$-module generated by $S^0$.

Finally, we show that $E_J^3 $ is contained  in the $\mathcal{Z}(E_J^*)$-module generated by $S^0\cup S^1$
 by noting that
\begin{align*}
	\x_{s_\epsilon}^J \cdot \left( \upbeta_{s_0(s_1s_0)^n}^J + \upbeta_{s_1(s_0s_1)^n}^J \right)  &= \uppsi_{s_\epsilon(s_{1-\epsilon}s_\epsilon)^n}^J - \uppsi_{(s_\epsilon s_{1-\epsilon})^n}^J \ , \\
	\x_{s_{\epsilon} s_{1-\epsilon}}^J \cdot \left( \upbeta_{s_0(s_1s_0)^n}^J + \upbeta_{s_1(s_0s_1)^n}^J \right)  &= \uppsi_{(s_\epsilon s_{1-\epsilon})^n}^J - \uppsi_{s_\epsilon (s_{1-\epsilon} s_\epsilon)^{n-1}}^J \ ,
\end{align*} 
along with an induction argument as above (recalling that $\upphi_1^J$, $\uppsi_{s_0}^J$, and  $\uppsi_{s_1}^J$ already lie in $\mathcal{Z}(E_J^*)$).
\end{proof}

\subsection{The spherical Hecke $\Ext$-algebra of  ${\rm SL}_2(\mathbb Q_p)$, $p\geq 5$}

We continue to assume $p \geq 5$, so that $K$ is $p$-torsion-free.  Our next task will be to obtain the structure of $E_K^*$.  Recall that $e_{J,K}= 1+\tau_{s_0}^J\in H_J$ and we define $T$ to be the image of $\zeta_J$ by the map \eqref{f:compacenter}.

\begin{proposition}\label{prop:fullbimoK}
The  spherical Hecke $\Ext$-algebra $E_K^*$ is supported in degrees $0$ to $3$.  We have $E_K^0= H_K=k[T]$ and, as $H_K$-bimodules; 
\begin{itemize}
\item $E_K^1\cong k[T]$,
\item $E_K^2\cong \bigcup _{n\geq 1} (k[T]/ (T^n))^\vee\cong  k[T^{\pm 1}]/k[T]$,
\item $E_K^3\cong \chi^K_{\triv}\oplus \bigcup _{n\geq 1} (k[T]/ (T^n))^\vee\cong  \chi^K_{\triv}\oplus k[T^{\pm 1}]/k[T]$.
\end{itemize}

\end{proposition}

\begin{proof}
Recall that the map $e_{J,K} \cdot E^*_J \cdot e_{J,K} \xrightarrow{C^*_{J,K} } E^*_K$ is an isomorphism of unital graded $k$-algebras. It therefore suffices to compute the graded components of $e_{J,K} \cdot E^*_J \cdot e_{J,K}$.  The graded components of $E_J^*$ are described in Proposition \ref{prop:fullbimo}, from which we obtain the following:
\begin{itemize}
\item We have $E^1_J\cong F^1H_J$ as  $H_J$-bimodules.  
Note that $e_{J,K} F^1H_J e_{J,K}$ contains $e_{J,K}(\zeta_J-1)k[\zeta_J]$, and the first (resp., second) space has codimension $1$ in $e_{J,K} H_J e_{J,K}$ (resp., $e_{J,K} k[\zeta_J]$). But by  \eqref{f:compacenter}, we have $e_{J,K} k[\zeta_J]=e_{J,K} H_J e_{J,K}$. Therefore, the containment above is actually an equality $e_{J,K} F^1H_J e_{J,K}=e_{J,K}(\zeta_J-1)k[\zeta_J]$, and $E^1_K$ is isomorphic to $(T-1)k[T]\cong k[T]$ as an $H_K$-bimodule.

\item Next, we have $E^2_J\cong \bigcup_{n\geq 1}(F^1H_J/\zeta_J^n F^1H_J)^\vee$ as  $H_J$-bimodules.   So
\[ E^2_K\cong \bigcup_{n\geq 1}((T-1)k[T]/T^n (T-1)k[T])^\vee\cong \bigcup_{n\geq 1}(k[T]/T^n k[T])^\vee \] 
which identifies with the quotient of the localization $k[T^{\pm 1}]$ by $k[T]$.

\item Lastly, we have $E_J^3\cong ({H_J})^{\vee, f}\cong  \chi^J_{\triv}\oplus \bigcup_{n\geq 1} (H_J/\zeta_J^n H_J) ^\vee $ as $H_J$-bimodules.  However, $\chi^J_{\triv}(e_{J,K})=1$ and $\chi^J_{\triv}(\zeta_J^n-\zeta_J^{n-1})=1-1=0$ for any $n\geq 1$. This proves that  
\[ E_K^3\cong  \chi^K_{\triv}\oplus \bigcup_{n\geq 1} (k[T]/T^n k[T]) ^\vee \cong  \chi^K_{\triv}\oplus k[T^{\pm 1}]/k[T] \ . \]
\end{itemize}
\end{proof}

Recall the commutative subalgebra $\mathfrak Z_J^*$  of $E_J^*$ given in Definition \ref{defZ}.
\begin{proposition} \label{prop:squeeze}
The homomorphism of algebras 
\begin{align}\label{f:squeeze}
\mathcal Z(E_J^*)&\longrightarrow  E_K^*\\
 z&\longmapsto C_{J,K}^*(e_{J,K}\cdot z\cdot e_{J,K})
\end{align} is surjective and has kernel equal to the ideal $\,k \uppsi_{s_0}^J\oplus k \upbeta_{s_0}^J$ of $\mathcal Z(E_J^*)$. It restricts to an isomorphism of $k$-algebras
\begin{equation} \label{f:squeezeZ} 
\mathfrak Z_J^* \underset{\sim}{\xrightarrow{C^*_{J,K}(e_{J,K}\cdot _-\cdot e_{J,K})}}  E_K^* \ .\end{equation}  
In particular, $E_K^*$ is commutative, but not graded-commutative. 
\end{proposition}

\begin{proof}  Note that \eqref{f:squeeze} is an  homomorphism of algebras by Remark \ref{rema:RPCP*}-\ref{rema:RPCP*-ii} and since $[K:J]=1\bmod p$. The fact that $\,k \uppsi_{s_0}^J\oplus k \upbeta_{s_0}^J$ lies in its kernel is clear from the definition of these elements and since $e_{J,K}=\tau_{s_0}^J+1$.

We consider  the map 
\begin{align}\label{f:squeeze1}
\mathcal Z(E_J^*)&\xrightarrow{e_{J,K}\cdot _-\cdot e_{J,K}} e_{J,K}\cdot E^*_J \cdot e_{J,K}\ .
\end{align} 
Its restriction to $\mathfrak Z^0_J = k[\zeta_J]$ is bijective, which comes from the bijection $k[\zeta_J]\underset{\sim}{\xrightarrow{e_{J,K}\cdot _-}} e_{J,K}H_J e_{J,K}$ (see \eqref{f:compacenter}). 
As discussed in the proof of Proposition \ref{prop:fullbimoK}, the latter restricts to a bijection
$(\zeta_J-1)k[\zeta_J]\cong e_{J,K}F^1H_J e_{J,K}$ which implies that 
\eqref{f:squeeze1} induces a bijection between $\mathfrak Z_J^1$ and $e_{J,K}\cdot E^1_J \cdot e_{J,K}$. 
Thus \eqref{f:squeeze} induces a bijection between $\mathfrak Z_J^0$ and $E_K^0$ as well as between 
$\mathfrak Z_J^1$ and $E_K^1$.

We now study the image of our chosen basis elements of $\mathfrak Z^2_J$ and  $\mathfrak Z^3_J$.  Let $n\geq 0$, and recall that $e_{K,J}\cdot \upbeta^J_{s_0(s_1s_0)^n}=0$ and $e_{K,J}\cdot \uppsi^J_{s_0(s_1s_0)^n}=0$.  Using \eqref{f:leftrightd-1}, we have:
\[ e_{K,J}\cdot (\upbeta^J_{s_0(s_1s_0)^n}+\upbeta^J_{s_1(s_0s_1)^n}) \cdot e_{J,K}=
e_{K,J}\cdot \upbeta^J_{s_1(s_0s_1)^n} \cdot e_{J,K} =\begin{cases}\upalpha^J_{s_1(s_0s_1)^{n-1}}-\upalpha^J_{s_1(s_0s_1)^n}&\text{ if $n\geq 1$,}\cr
-\upalpha^J_{s_1}&\text{ if $n=0$}. 
\end{cases} \]
Likewise, using \eqref{f:leftrightHd}, we have:
\[e_{K,J}\cdot (\uppsi^J_{s_0(s_1s_0)^n}+\uppsi^J_{s_1(s_0s_1)^n}) \cdot e_{J,K}=
e_{K,J}\cdot \uppsi^J_{s_1(s_0s_1)^n} \cdot e_{J,K} =\begin{cases}\upphi^J_{s_1(s_0s_1)^{n-1}}-\upphi^J_{s_1(s_0s_1)^n}&\text{ if $n\geq 1$,}\cr
\upphi^J_1-\upphi^J_{s_1}&\text{ if $n=0$,}
\end{cases} \]
and
$e_{K,J}\cdot \upphi^J_1 \cdot e_{J,K}=\upphi^J_1$.  This shows that \eqref{f:squeeze1} maps a basis  of $\mathfrak Z^2_J\oplus \mathfrak Z^3_J$ onto a basis of $\bigoplus_{d\in {}_K\D_K} H^2(J, \X_J(d))\oplus H^3(J, \X_J(d))$ (recalling that $H^2(J, \X_J(1))=0$).
By Lemma \ref{lemma:JKd} and Proposition \ref{Uduality},  we know that $H^i(J, \X_J(d))$ and $H^i(K, \X_K(d))$ have the same dimension for $i=2,3$ and any $d\in {}_K\D_K\smallsetminus \{1\}$  (and $H^2(K, \X_K(1))=0$) . 
We deduce, using Remark \ref{rema:res-DPP}, that  $C_{J,K}^i$ yields a bijection  between $H^i(J, \X_J(d))$ and $H^i(K, \X_K(d))$ for $d\in {}_K\D_K$.
This shows that the restriction to $\mathfrak Z^2_J\oplus \mathfrak Z^3_J$ of 
 \eqref{f:squeeze} is bijective.

To conclude, we use Lemma \ref{lemma:Z}.
\end{proof}

\begin{remark}\label{rema:square}
The vector space $H^1(K,\X_K(s_0\theta))$ is $1$-dimensional with basis $C^1_{J,K}(\x_{s_1}^J)$. 
Since $e_{J,K}(\zeta_J-1)=e_{J,K}\tau^J_{s_1}$,
it is easy to check that $C^1_{J,K}(\x_{s_1}^J)$ is the image under \eqref{f:squeeze} of $f_{(\x_{s_0}^J,\x_{s_1}^J)}(\zeta_J-1)$. By the proof of Lemma \ref{lemma:Z}, the square of the latter element is  the nonzero element
$\upbeta^J_{s_0}+\upbeta^J_{s_1}\in \mathfrak Z_J^2$ whose image under \eqref{f:squeeze} is $-C^1_{J,K}(\upalpha_{s_1}^J)$.  Thus, $H^1(K,\X_K(s_0\theta))$ and $H^2(K,\X_K(s_0\theta))$ are both $1$-dimensional vector spaces, and the Yoneda product in $E_K^*$ gives a map
\[ H^1(K,\X_K(s_0\theta))\otimes _k H^1(K,\X_K(s_0\theta))\longrightarrow H^2(K,\X_K(s_0\theta)) \]
which is not zero.
The existence of this nonzero square is the reason why we know that $E_K^*$ is not graded-commutative. In Subsection \ref{subsec:nonzerosq} we generalize this observation to the case where $\mathfrak{F}$ is unramified over $\bbQ_p$.
\end{remark}

\begin{remark}
The map \eqref{f:squeeze} is equivariant for the action of $\zeta_J(\zeta_J-1)$ on the source $\mathcal Z(E_J^*)$ and of $T(T - 1)$ on the target $E_K^*$.  It therefore induces a map on localizations.  This localization annihilates the degree $2$ and $3$ graded pieces of both algebras, and consequently we obtain an isomorphism of algebras $\mathcal Z(E_J^*)_{\zeta_J(\zeta_J-1)}\cong (E^*_K)_{T(T-1)}$.
\end{remark}

\subsection{A nonzero square in the spherical Hecke $\Ext$-algebra of $\textnormal{SL}_2(\mathfrak{F})$\label{subsec:nonzerosq}}

In this subsection, we continue to work with the spherical $\Ext$-algebra $E_K^*$ for the group $\textnormal{SL}_2(\mathfrak{F})$.  We will show below that when $\mathfrak{F}$ is unramified over $\bbQ_p$ with $p$ odd and $q \neq 3$, the algebra $E_K^*$ contains a nonzero square of a degree 1 element, which shows it is not graded-commutative.  (Compare with Remark \ref{rema:square}.)

We assume $p$ is odd and $q \neq 3$.  Recall that $\alpha$ denotes the unique positive root (relative to the upper-triangular Borel subgroup), and define $\theta := \check{\alpha}(\pi) = \sm{\pi}{0}{0}{\pi^{-1}}$.  The Cartan decomposition then gives
$ G = \bigsqcup_{n \geq 0} K\theta^n K$ (compare with \eqref{KDK}).
Thus, according to equations \eqref{f:frob} and \eqref{f:H*dec}, we have
\[ E_K^* \cong \bigoplus_{n \geq 0} H^*(K_{\theta^n},k). \]
We note that for $n \geq 1$, the group $K_{\theta^n}$ takes the form
$K_{\theta^n} = \begin{pmatrix} \mathfrak{O}^\times & \mathfrak{M}^{2n} \\ \mathfrak{O} & \mathfrak{O}^\times\end{pmatrix}.$

\subsubsection{Some cohomology groups}  

The purpose of this section is to calculate some cohomology spaces of a particular group appearing below.  To this end, let us define
\[ L := \begin{pmatrix}1 + \mathfrak{M} & \mathfrak{M}^2 \\ \mathfrak{O} & 1 + \mathfrak{M}\end{pmatrix} \cdot \langle -\textnormal{id}\rangle \subseteq K. \]

\begin{lemma}
\label{lem:Lcoh}
    We have 
    \[ [L,L] = \begin{pmatrix} 1 + \mathfrak{M}^2 & \mathfrak{M}^3 \\ \mathfrak{M} & 1 + \mathfrak{M}^2\end{pmatrix},\qquad \Phi(L) = [L,L]L^p = \begin{pmatrix} 1 + \mathfrak{M}^2 & \mathfrak{M}^3 \\ \mathfrak{M} & 1 + \mathfrak{M}^2\end{pmatrix}\cdot \langle -\textnormal{id}\rangle. \]
    Thus, we get 
    \[ L_\Phi \cong \mathfrak{sl}_2(\mathbb{F}_q) \]
    via the map 
    \[ \begin{pmatrix} 1 + \pi a & \pi^2 b \\ c & 1 + \pi d\end{pmatrix}\cdot (-\textnormal{id})^j \longmapsto \begin{pmatrix} a & b \\ c & d\end{pmatrix}~(\textnormal{mod}~\mathfrak{M}). \]
\end{lemma}

\begin{proof}
	Let us denote the group on the right-hand side of the claimed formula for $[L,L]$ by $L_{+}$.  Proceeding as in Lemmas \ref{lem:comm-PY} and \ref{lem:UYintT} (and using the Iwahori decomposition for elements of $L_{+}$) shows that $L_{+} \subseteq [L,L]$.  (To imitate the argument of Lemma \ref{lem:comm-PY}, we wish to find an element $x\in 1 + \mathfrak{M}$ such that $x^2 - 1 \in \mathfrak{M} \smallsetminus \mathfrak{M}^2$; this is guaranteed by our assumption $p > 2$.)  On the other hand, one easily checks that we have an isomorphism of groups
	\[L/L_{+} \cong \mathfrak{sl}_2(\mathbb{F}_q) \oplus \bbZ/2\bbZ\]
	(the $\bbZ/2\bbZ$ factor corresponding to $\langle -\textnormal{id}\rangle$).   In particular, $L/L_{+}$ is abelian, from which we deduce $[L,L] \subseteq L_{+}$.  The case of $\Phi(L)$ follows analogously, using that the subgroup $\langle -\textnormal{id}\rangle$ has prime-to-$p$ order.
\end{proof}

\begin{corollary}
\label{cor:H1L}
Suppose $\varphi:\mathbb{F}_q \longrightarrow k$ is an $\mathbb{F}_p$-linear homomorphism.  Define the homomorphisms $\varphi_{\textnormal{U}}, \varphi_{\textnormal{L}}, \varphi_{\textnormal{D}}: L \longrightarrow k$ via
\begin{eqnarray*}
\varphi_{\textnormal{U}} \left(\begin{pmatrix} 1 + \pi a & \pi^2 b \\ c & 1 + \pi d\end{pmatrix}\cdot (-\textnormal{id})^j\right) & = & \varphi(\overline{b}) \\
\varphi_{\textnormal{L}} \left(\begin{pmatrix} 1 + \pi a & \pi^2 b \\ c & 1 + \pi d\end{pmatrix}\cdot (-\textnormal{id})^j\right) & = & \varphi(\overline{c}) \\
\varphi_{\textnormal{D}} \left(\begin{pmatrix} 1 + \pi a & \pi^2 b \\ c & 1 + \pi d\end{pmatrix}\cdot (-\textnormal{id})^j\right) & = & \varphi(\overline{a}) ,
\end{eqnarray*}
where $a,b,c,d \in \mathfrak{O}, j \in \{0,1\}$.  We then have
\[ H^1(L,k) =  \textnormal{span}\{\varphi_{\textnormal{U}}, \varphi_{\textnormal{L}}, \varphi_{\textnormal{D}}\}_{\varphi \in \Hom_{\mathbb{F}_p}(\mathbb{F}_q,k)}.\]
In particular, $\dim_k(H^1(L,k)) = 3f$.  
\end{corollary}

Next, we use the above corollary to calculate some corestriction maps.

\begin{lemma}
    \label{lem:corestrictions}
    \begin{enumerate}
    \item Let $\varphi:\bbF_q \longrightarrow k$ be an $\bbF_p$-linear homomorphism.  We have
    \[ \textnormal{cores}^L_{K_\theta}(\varphi_{\textnormal{D}} \cup \varphi_{\textnormal{U}}) = \textnormal{cores}^L_{K_\theta}(\varphi_{\textnormal{D}} \cup \varphi_{\textnormal{L}}) = 0.\]
    \item Suppose furthermore that $\mathfrak{F}$ is unramified over $\bbQ_p$.  Then the trace map $\varphi = \textnormal{Tr}_{\bbF_q/\bbF_p}$ satisfies
    \[ \textnormal{cores}^L_{K_\theta}(\varphi_{\textnormal{L}} \cup \varphi_{\textnormal{U}}) \neq 0. \]
    \end{enumerate}
\end{lemma}

\begin{proof}

We first note the following general fact.  Suppose $\chi:\bbF_q^\times \longrightarrow k^\times$ is a non-trivial character and $\beta \in H^i(L,k)$ is a cohomology class which is an eigenvector for the conjugation action of $T^0$ with eigencharacter $\chi$ (that is, we have
\[\begin{pmatrix} x & 0 \\ 0 & x^{-1}\end{pmatrix}_* \beta = \chi(\overline{x})\beta\]
for all $x \in \mathfrak{O}^\times$.)  Then we have $\textnormal{cores}_{K_\theta}^L(\beta) = 0$.  Indeed, since $T^0$ is contained in $K_\theta$, we have
    \begin{eqnarray*}
        \textnormal{cores}^L_{K_\theta}(\beta) & = & \begin{pmatrix}x & 0 \\ 0 & x^{-1}\end{pmatrix}_*\textnormal{cores}^L_{K_\theta}(\beta)  =  \textnormal{cores}^L_{K_\theta}\left(\begin{pmatrix} x & 0 \\ 0 & x^{-1}\end{pmatrix}_*\beta\right) 
         =  \chi(\overline{x}) \textnormal{cores}^L_{K_\theta}(\beta).
    \end{eqnarray*}
    Since $\chi(\overline{x}) \neq 1$ for some $x$, we conclude that $\textnormal{cores}^L_{K_\theta}(\beta) = 0$.
    
    We also note that by base-changing from $k$ to $k\bbF_q$, we may assume that the coefficient field contains $\bbF_q$.  

\begin{enumerate}
	\item We prove the claim for $\textnormal{cores}^L_{K_\theta}(\varphi_{\textnormal{D}} \cup \varphi_{\textnormal{U}})$, the other case being completely analogous.

Suppose first that $\varphi:\mathbb{F}_q \longrightarrow k$ is of the form $a \longmapsto a^{p^i}$.   In this case, the action of $T^0$ on $\varphi_{\textnormal{D}}\cup \varphi_{\textnormal{U}}$ is given by the character $\sm{x}{0}{0}{x^{-1}} \longmapsto \overline{x}^{-2p^i}$.  Since $q > 3$, this character is non-trivial, and by the first paragraph we get $\textnormal{cores}^L_{K_\theta}(\varphi_{\textnormal{D}} \cup \varphi_{\textnormal{U}}) = 0$.

    Next, suppose that $\varphi: \mathbb{F}_q \longrightarrow k$ is an arbitrary $\mathbb{F}_p$-linear homomorphism.  By writing $\varphi$ as a linear combination of Frobenius maps and analogously decomposing $\varphi_{\textnormal{D}} \cup \varphi_{\textnormal{U}}$, and by using the previous paragraph, we obtain $\textnormal{cores}^L_{K_\theta}(\varphi_{\textnormal{D}} \cup \varphi_{\textnormal{U}}) = 0$.

    \item Suppose $\mathfrak{F}$ is unramified over $\mathbb{Q}_p$, and let $\varphi:\mathbb{F}_q \longrightarrow k$ denote a nonzero $\bbF_p$-linear homomorphism.  In this case, it is well known that the group 
    \[K_\theta' := \begin{pmatrix} 1 + \mathfrak{M} & \mathfrak{M}^2 \\ \mathfrak{O} & 1 + \mathfrak{M}\end{pmatrix}\]
    is a uniform pro-$p$-group (since it is conjugate to the principal congruence subgroup of $\textnormal{SL}_2(\mathfrak{O})$, this follows directly from the definition of uniformity and \cite[Prop. 3.62]{OS2}).  In particular, since $\textnormal{res}^{L}_{K_\theta'}(\varphi_{\textnormal{L}})$ and $\textnormal{res}^{L}_{K_\theta'}(\varphi_{\textnormal{U}})$ are linearly independent in $H^1(K_\theta',k)$, we obtain
    \[\textnormal{res}^{L}_{K_\theta'}(\varphi_{\textnormal{L}} \cup \varphi_{\textnormal{U}}) = \textnormal{res}^{L}_{K_\theta'}(\varphi_{\textnormal{L}}) \cup \textnormal{res}^{L}_{K_\theta'}(\varphi_{\textnormal{U}}) \neq 0\]
    in $H^2(K_\theta',k)$.

    Suppose now that $\varphi(a) = a^{p^i}$.  We claim that the element $\textnormal{cores}_{K_\theta}^L(\varphi_{\textnormal{L}} \cup \varphi_{\textnormal{U}})$ is nonzero.  In order to do this, it suffices to prove this cohomology class is nonzero after applying $\textnormal{res}^{K_\theta}_{K_\theta'}$.  Using that $L$ is normal in $K_{\theta}$, \cite[Cor. 1.5.7]{NSW} gives
    \begin{eqnarray*}
        \textnormal{res}^{K_\theta}_{K_\theta'} \circ \textnormal{cores}^L_{K_\theta}(\varphi_{\textnormal{L}} \cup \varphi_{\textnormal{U}}) & = &   \textnormal{res}^{L}_{K_\theta'}\circ \textnormal{res}^{K_\theta}_{L} \circ \textnormal{cores}^L_{K_\theta}(\varphi_{\textnormal{L}} \cup \varphi_{\textnormal{U}})  =   \textnormal{res}^{L}_{K_\theta'}\left(\sum_{g\in K_\theta/L} g_*(\varphi_{\textnormal{L}} \cup \varphi_{\textnormal{U}})\right) \\
        & = &  \textnormal{res}^{L}_{K_\theta'}\left(\sum_{a \in \mathbb{F}_q^\times/\pm 1} \begin{pmatrix}[a] & 0 \\ 0 & [a]^{-1}\end{pmatrix}_*(\varphi_{\textnormal{L}} \cup \varphi_{\textnormal{U}})\right) 
          =   \textnormal{res}^{L}_{K_\theta'}\left(\sum_{a \in \mathbb{F}_q^\times/\pm 1} \varphi_{\textnormal{L}} \cup \varphi_{\textnormal{U}}\right) \\
         & = & \frac{q - 1}{2}\cdot  \textnormal{res}^{L}_{K_\theta'}\left(\varphi_{\textnormal{L}}\cup \varphi_{\textnormal{U}}\right) \quad
          \neq  0,
    \end{eqnarray*}
    where we have used that $T^0$ acts on $\varphi_{\textnormal{L}}$ and $\varphi_{\textnormal{U}}$ by inverse characters.  

    Finally, suppose $\varphi = \textnormal{Tr}_{\bbF_q/\bbF_p}$, and write $\varphi = \sum_{i = 0}^{f - 1}\sigma_i$, where $\sigma_i(a) = a^{p^i}$.  We then have
    \begin{eqnarray*}
    \textnormal{cores}_{K_\theta}^L(\textnormal{Tr}_{\bbF_q/\bbF_p, \textnormal{L}} \cup \textnormal{Tr}_{\bbF_q/\bbF_p, \textnormal{U}}) & = & \sum_{i,j = 0}^{f - 1}\textnormal{cores}_{K_\theta}^L(\sigma_{i,\textnormal{L}} \cup \sigma_{j, \textnormal{U}})       =  \sum_{i= 0}^{f - 1}\textnormal{cores}_{K_\theta}^L(\sigma_{i,\textnormal{L}} \cup \sigma_{i, \textnormal{U}}),
    \end{eqnarray*}
    where we have used the argument of the first paragraph to conclude that $\textnormal{cores}_{K_\theta}^L(\sigma_{i,\textnormal{L}} \cup \sigma_{j, \textnormal{U}}) = 0$ if $i \neq j$.  The summands above are nonzero (by the previous paragraph) and linearly independent (this can be seen by applying $\textnormal{res}^{K_\theta}_{K_\theta'}$ and again using that $K_\theta'$ is uniform).  Hence $\textnormal{cores}_{K_\theta}^L(\textnormal{Tr}_{\bbF_q/\bbF_p, \textnormal{L}} \cup \textnormal{Tr}_{\bbF_q/\bbF_p, \textnormal{U}}) \neq 0$.
        \end{enumerate}
\end{proof}

\subsubsection{Yoneda products}

We now compute some Yoneda products.  Fix an element $\alpha\in E_K^1$ which corresponds to a nontrivial homomorphism $\varphi \in H^1(K_\theta,k)$ (by Remark \ref{rem:H1parahoric} and Corollary \ref{cor:frattini-longest}, such a homomorphism exists and is supported on $T^0$). We furthermore assume that $\mathfrak{F}$ is unramified over $\bbQ_p$, and choose $\varphi = \textnormal{Tr}_{\bbF_q/\bbF_p}$ as in Lemma \ref{lem:corestrictions}.

We will compute $\alpha^2 = \alpha\cdot \alpha \in E_K^2$ using Proposition \ref{yoneda-product-U}.  We have
\begin{equation}
\label{nongradcomm-square}
\alpha^2 = \sum_{\substack{n \geq 0 \\ K\theta^nK \subseteq K \theta K \theta K}} \gamma_{\theta^n},
\end{equation}
where $\gamma_{\theta^n} \in H^2(K,\textnormal{c-ind}_K^{K\theta^n K}(k))$, and 
\begin{equation}
\label{nongradcomm-Sh}
\textnormal{Sh}_{\theta^n}(\gamma_{\theta^n}) = \sum_{h \in K_{\theta^{-1}}\backslash (\theta^{-1}K\theta^n \cap K\theta K)/K_{\theta^{-n}}}\textnormal{cores}^{K_{\theta^n} \cap \theta^nh^{-1}Kh\theta^{-n}}_{K_{\theta^n}}\left(\widetilde{\Gamma}_{\theta^n,h}\right)
\end{equation}
where
\begin{equation}
    \label{cup-product-Gamma}
    \widetilde{\Gamma}_{\theta^n,h} := \textnormal{res}^{K \cap \theta^nh^{-1}Kh\theta^{-n}}_{K_{\theta^n} \cap \theta^nh^{-1}Kh\theta^{-n}}(a_*\varphi) \cup \textnormal{res}^{\theta^nK\theta^{-n} \cap \theta^nh^{-1}Kh\theta^{-n}}_{K_{\theta^n} \cap \theta^nh^{-1}Kh\theta^{-n}}((a\theta c)_*\varphi),  
\end{equation}
where $h = c\theta d = \theta^{-1}a^{-1}\theta^n$ for $a,c,d \in K$.

We unpack the above formulas piece by piece.  

\begin{lemma}
    We have $K\theta K\theta K = K \sqcup K \theta K \sqcup K \theta^2 K$.  We therefore only have three summands in \eqref{nongradcomm-square}.
\end{lemma}

\begin{proof}
    This follows from examining the action of $G$ on the Bruhat--Tits tree (and noting that $K$ fixes the standard vertex).
\end{proof}

Next, we examine the double coset space $K_{\theta^{-1}}\backslash (\theta^{-1}K\theta^n \cap K\theta K)/K_{\theta^{-n}}$.

\begin{lemma}
    The coset space $K_{\theta^{-1}}\backslash (\theta^{-1}K\theta^n \cap K\theta K)$ admits the following representatives:
    \begin{itemize}
        \item $n = 0$: 
        \[ \left\{\begin{pmatrix}-y\pi^{-1} & -\pi^{-1} \\ \pi & 0 \end{pmatrix}\right\}_{y \in \mathfrak{O}/\mathfrak{M}^2}\sqcup \left\{\begin{pmatrix}\pi^{-1} & -y \\ 0 & \pi\end{pmatrix}\right\}_{y \in \mathfrak{O}/\mathfrak{M}}. \]
        \item $n = 1$:
        \[ \left\{\begin{pmatrix}1 & -\pi^{-1}y \\ 0 & 1\end{pmatrix}\right\}_{y \in \mathfrak{O}/\mathfrak{M}, y \neq 0}. \]
        \item $n = 2$:
        \[ \left\{\begin{pmatrix}\pi & 0 \\ 0 & \pi^{-1}\end{pmatrix}\right\}. \]
    \end{itemize}
\end{lemma}

\begin{proof}
    Note first that we may choose coset representatives for $K/K_\theta$ as follows:
    \[ \left\{\begin{pmatrix} 0 & 1 \\ -1 & 0\end{pmatrix}\begin{pmatrix}1 & y \\ 0 & 1\end{pmatrix}\right\}_{y \in \mathfrak{O}/\mathfrak{M}^2}\sqcup \left\{\begin{pmatrix}1 & \pi y \\ 0 & 1\end{pmatrix}\right\}_{y \in \mathfrak{O}/\mathfrak{M}}. \]
    According to \cite[Rmk. 5.1]{Ext}, in order to obtain the result we must verify which of the above representatives $a$ satisfy $\theta^{-1}a^{-1}\theta^n \in K\theta K$.  (Note that the inclusion ``$v^{-1}au \in IwI$'' in \cite[Rmk. 5.1]{Ext} should read ``$v^{-1}a^{-1}u \in IwI$.'')
    \begin{itemize}
        \item $n = 0$: We note that for any $a$ as above, we have $a \in K$, and therefore $\theta^{-1}a^{-1} \in K\theta^{-1}K = K \theta K$.
        \item $n = 1:$ for $a = \sm{0}{1}{-1}{0}\sm{1}{y}{0}{1}$, we have
        \begin{eqnarray*}
            \theta^{-1}a^{-1}\theta & = & \begin{pmatrix} \pi^{-1} & 0 \\ 0 & \pi\end{pmatrix}\begin{pmatrix}1 & -y \\ 0 & 1\end{pmatrix}\begin{pmatrix}0 & -1 \\ 1 & 0\end{pmatrix}\begin{pmatrix} \pi & 0 \\ 0 & \pi^{-1}\end{pmatrix}  
             =  \begin{pmatrix}-y & -\pi^{-2} \\ \pi^2 & 0\end{pmatrix} \\
            & = & \begin{pmatrix} \pi^{-2} & 0 \\ 0 & \pi^2\end{pmatrix}\begin{pmatrix}-\pi^2y & -1 \\ 1 & 0\end{pmatrix} 
             \in  K\theta^{-2}K = K \theta^2 K.
        \end{eqnarray*}
        For $a = \sm{1}{\pi y}{0}{1}$ with $y \neq 0$, we have
        \begin{eqnarray*}
            \theta^{-1}a^{-1}\theta & = & \begin{pmatrix} \pi^{-1} & 0 \\ 0 & \pi \end{pmatrix}\begin{pmatrix}1 & -\pi y \\ 0 & 1\end{pmatrix}\begin{pmatrix}\pi & 0 \\ 0 & \pi^{-1}\end{pmatrix} 
             =  \begin{pmatrix} 1 & -\pi^{-1}y \\ 0 & 1\end{pmatrix} \\
            & = & \begin{pmatrix}0 & 1 \\ -1 & -\pi y^{-1} \end{pmatrix}\begin{pmatrix} \pi & 0 \\ 0 & \pi^{-1}\end{pmatrix}\begin{pmatrix}-y^{-1} & 0 \\ \pi & -y\end{pmatrix}
             \in  K \theta K.
        \end{eqnarray*}
        Finally, if $a = 1$, then $\theta^{-1}a^{-1}\theta = 1 \in K$.
        \item $n = 2$: for $a = \sm{0}{1}{-1}{0}\sm{1}{y}{0}{1}$, we have
        \begin{eqnarray*}
            \theta^{-1}a^{-1}\theta^2 & = & \begin{pmatrix} \pi^{-1} & 0 \\ 0 & \pi\end{pmatrix}\begin{pmatrix}1 & -y \\ 0 & 1\end{pmatrix}\begin{pmatrix}0 & -1 \\ 1 & 0\end{pmatrix}\begin{pmatrix}\pi^2 & 0 \\ 0 & \pi^{-2}\end{pmatrix}  
             =  \begin{pmatrix}-\pi y & -\pi^{-3} \\ \pi^3 & 0\end{pmatrix} \\
            & = & \begin{pmatrix}\pi^{-3} & 0 \\ 0 & \pi^3\end{pmatrix}\begin{pmatrix}-\pi^4y & -1 \\ 1 & 0\end{pmatrix} 
             \in  K\theta^{-3}K = K \theta^3 K.
        \end{eqnarray*}
        For $a = \sm{1}{\pi y}{0}{1}$ with $y \neq 0$, we have
        \begin{eqnarray*}
            \theta^{-1}a^{-1}\theta^2 & = & \begin{pmatrix} \pi^{-1} & 0 \\ 0 & \pi\end{pmatrix}\begin{pmatrix}1 & -\pi y \\ 0 & 1\end{pmatrix}\begin{pmatrix}\pi^2 & 0 \\ 0 & \pi^{-2}\end{pmatrix} \\
            & = & \begin{pmatrix}0 & 1 \\ -1 & -\pi y^{-1} \end{pmatrix}\begin{pmatrix} \pi^2 & 0 \\ 0 & \pi^{-2}\end{pmatrix}\begin{pmatrix}-y^{-1} & 0 \\ \pi^3 & -y\end{pmatrix} 
             \in  K \theta^2 K
        \end{eqnarray*}
        Finally, if $a = 1$, then $\theta^{-1}a^{-1}\theta^2 = \theta \in K\theta K$.
    \end{itemize}
\end{proof}

\begin{corollary}
    \label{double-coset-reps}
    The double coset space $K_{\theta^{-1}}\backslash (\theta^{-1}K\theta^n \cap K\theta K)/K_{\theta^{-n}}$ admits the following representatives:
    \begin{itemize}
        \item $n = 0$: 
        \[ \left\{\begin{pmatrix}\pi^{-1} & 0 \\ 0 & \pi \end{pmatrix}\right\}. \]
        \item $n = 1$:
        \[ \left\{\begin{pmatrix}1 & \pi^{-1} \\ 0 & 1\end{pmatrix}\right\}. \]
        \item $n = 2$:
        \[ \left\{\begin{pmatrix}\pi & 0 \\ 0 & \pi^{-1}\end{pmatrix}\right\} \]
    \end{itemize}
    Thus, the sum in \eqref{nongradcomm-Sh} (for $n = 0,1,2$) has one summand.
\end{corollary}

\begin{proof}
    This follows from the previous proposition: the $n = 0$ case follows from the fact that the representatives are of the form $\theta^{-1}a^{-1}$ for various $a \in K = K_{\theta^0}$; the $n = 1$ case follows from the identity
    \begin{eqnarray*}
    \begin{pmatrix} 1 & -\pi^{-1}z \\ 0 & 1\end{pmatrix} & = & \begin{pmatrix} 1 & 0 \\ \pi(y^{-1} - z^{-1}) & 1\end{pmatrix}\begin{pmatrix} 1 & -\pi^{-1}y \\ 0 & 1\end{pmatrix}\begin{pmatrix} yz^{-1} & 0 \\ 0 & y^{-1}z\end{pmatrix}\begin{pmatrix} 1 & 0 \\ \pi(yz^{-2} - z^{-1}) & 1\end{pmatrix} \\
    & \in & K_{\theta^{-1}}\begin{pmatrix} 1 & -\pi^{-1}y \\ 0 & 1\end{pmatrix} K_{\theta^{-1}} 
    \end{eqnarray*}
    for nonzero $y, z \in \mathfrak{O}/\mathfrak{M}$ satisfying $y \neq z$; the $n = 2$ case is due to the fact that $K_{\theta^{-1}}\backslash (\theta^{-1}K\theta^n \cap K\theta K)$ is a singleton.
\end{proof}

Next, we slightly rewrite the formula for $\widetilde{\Gamma}_{\theta^n,h}$. The first term in the cup product is equal to
\[ \textnormal{res}^{K \cap \theta^n h^{-1} K h\theta^{-n}}_{K_{\theta^n} \cap \theta^n h^{-1} K h\theta^{-n}}(a_*\varphi). \]
Note that $a \in K$ and $h = \theta^{-1}a^{-1}\theta^n$.  Therefore, conjugating the two subgroups appearing above by $a^{-1}$ gives
\begin{eqnarray*}
a^{-1}(K \cap \theta^nh^{-1}Kh\theta^{-n})a & = & a^{-1}Ka \cap a^{-1}\theta^nh^{-1}Kh\theta^{-n}a =  K \cap \theta K\theta^{-1}  =  K_\theta\\
a^{-1}(K_{\theta^n} \cap \theta^n h^{-1} K h\theta^{-n})a & = & a^{-1}K_{\theta^n}a \cap a^{-1}\theta^n h^{-1} K h\theta^{-n}a  =  K \cap a^{-1}\theta^nK\theta^{-n}a \cap \theta K \theta^{-1}  =  K_\theta \cap \theta h K h^{-1}\theta^{-1}.
\end{eqnarray*}
Consequently, we may rewrite the first term as
\begin{equation}
    \label{first-cup-term}
\textnormal{res}^{K \cap \theta^n h^{-1} K h\theta^{-n}}_{K_{\theta^n} \cap \theta^n h^{-1} K h\theta^{-n}}(a_*\varphi) = a_*\textnormal{res}^{K_\theta}_{K_{\theta} \cap \theta h K h^{-1}\theta^{-1}}(\varphi).  
\end{equation}
Similarly, the second term in the cup product defining $\widetilde{\Gamma}_{\theta^n,h}$ is equal to
\[ \textnormal{res}^{\theta^n K\theta^{-n} \cap \theta^n h^{-1} K h\theta^{-n}}_{K_{\theta^n} \cap \theta^n h^{-1} K h\theta^{-n}}((a\theta c)_*\varphi). \]
Recall that $h = \theta^{-1}a^{-1}\theta^n = c\theta d$, which implies $\theta^{-n}a\theta c = d^{-1}\theta^{-1}$ and $h\theta^{-n}a\theta c = c$.  Therefore, conjugating the two subgroups appearing above by $(a\theta c)^{-1}$ gives
\begin{eqnarray*}
(a\theta c)^{-1}(\theta^n K\theta^{-n} \cap \theta^n h^{-1} K h\theta^{-n})(a\theta c) & = & c^{-1}\theta^{-1}a^{-1}\theta^{n}K\theta^{-n}a\theta c \cap c^{-1}\theta^{-1}a^{-1}\theta^n h^{-1} K h\theta^{-n}a\theta c \\
& = & \theta dKd^{-1}\theta^{-1} \cap c^{-1} K c  =  \theta K \theta^{-1} \cap K  =  K_\theta \\
(a\theta c)^{-1}(K_{\theta^n} \cap \theta^n h^{-1} K h\theta^{-n})(a\theta c) & = & c^{-1}\theta^{-1}a^{-1}K_{\theta^n}a\theta c \cap c^{-1}\theta^{-1}a^{-1}\theta^n h^{-1} K h\theta^{-n}a\theta c \\
& = & c^{-1}\theta^{-1}a^{-1}Ka\theta c \cap c^{-1}\theta^{-1}a^{-1}\theta^n K\theta^{-n}a\theta c \\
& &  \cap c^{-1}\theta^{-1}a^{-1}\theta^n h^{-1} K h\theta^{-n}a\theta c \\
& = & c^{-1}\theta^{-1}K\theta c \cap \theta K \theta^{-1} \cap K  =  c^{-1}\theta^{-1}K\theta c \cap K_\theta.
\end{eqnarray*}
We may therefore rewrite the second term as 
\begin{equation}
\label{second-cup-term}
 \textnormal{res}^{\theta^n K\theta^{-n} \cap \theta^n h^{-1} K h\theta^{-n}}_{K_{\theta^n} \cap \theta^n h^{-1} K h\theta^{-n}}((a\theta c)_*\varphi) = (a\theta c)_*\textnormal{res}^{K_\theta}_{c^{-1}\theta^{-1}K\theta c \cap K_\theta}(\varphi).  
\end{equation}

To continue, we calculate the subgroups appearing in the restriction.

\begin{lemma}
Let $h$ run over the representatives of $K_{\theta^{-1}}\backslash (\theta^{-1}K\theta^n \cap K \theta K)/K_{\theta^{-n}}$ as in Corollary \ref{double-coset-reps}.  The groups $K_\theta \cap \theta h K h^{-1}\theta^{-1}$ and $K_\theta \cap c^{-1}\theta^{-1} K \theta c$ from equations \eqref{first-cup-term} and \eqref{second-cup-term} are given as follows:
\begin{itemize}
    \item $n = 0$: for $h = \sm{\pi^{-1}}{0}{0}{\pi}$, we have $a = 1, c = \sm{0}{1}{-1}{0}, d = \sm{0}{-1}{1}{0}$, and 
    \[ K_\theta \cap \theta h K h^{-1}\theta^{-1} = K_\theta,\qquad K_\theta \cap c^{-1}\theta^{-1} K \theta c = K_\theta. \]
    \item $n = 1$: for $h = \sm{1}{\pi^{-1}}{0}{1}$, we have $a = \sm{1}{-\pi}{0}{1}, c = \sm{0}{1}{-1}{\pi}, d = \sm{1}{0}{\pi}{1}$, and
    \[ K_\theta \cap \theta h K h^{-1}\theta^{-1} = L,\qquad K_\theta \cap c^{-1}\theta^{-1} K \theta c = L. \]
    \item $n = 2$: for $h = \sm{\pi}{0}{0}{\pi^{-1}}$, we have $a = 1, c = 1, d = 1$, and 
    \[ K_\theta \cap \theta h K h^{-1}\theta^{-1} = K_{\theta^2},\qquad K_\theta \cap c^{-1}\theta^{-1} K \theta c = \theta^{-1}K_{\theta^2}\theta. \]
\end{itemize}
\end{lemma}

\begin{proof}
    The $n = 0$ and $n = 2$ cases are straighforward, so we focus on $n = 1$.  Suppose $g = \sm{a}{\pi^2b}{c}{d}$ is an element of $K_\theta$ (so that $a,d \in \mathfrak{O}^\times, b,c \in \mathfrak{O}$) which also lies in $\theta hKh^{-1}\theta^{-1}$.  This implies $h^{-1}\theta^{-1}g\theta h \in K$.  Writing this out gives
    \begin{eqnarray*}
        h^{-1}\theta^{-1}g\theta h & = & \begin{pmatrix}1 & -\pi^{-1} \\ 0 & 1\end{pmatrix} \begin{pmatrix}\pi^{-1} & 0 \\ 0 & \pi\end{pmatrix} \begin{pmatrix}a & \pi^2b \\ c & d\end{pmatrix} \begin{pmatrix}\pi & 0 \\ 0 & \pi^{-1}\end{pmatrix} \begin{pmatrix}1 & \pi^{-1} \\ 0 & 1\end{pmatrix} 
         =  \begin{pmatrix}a - \pi c & (b - c) + \pi^{-1}(a - d) \\ \pi^2 c & d + \pi c\end{pmatrix}
    \end{eqnarray*}
    Since this latter matrix must lie in $K$, we obtain $a \equiv d~\textnormal{mod}~\mathfrak{M}$; combining this with the equation $ad - \pi^2bc = 1$ shows that $a \equiv d \equiv \pm 1 ~\textnormal{mod}~\mathfrak{M})$, which implies $K_\theta \cap \theta h K h^{-1}\theta^{-1} = L$.  The calculation of $K_\theta \cap c^{-1}\theta^{-1} K \theta c$ is completely analogous.
\end{proof}

We may now determine the terms appearing in the cup product \eqref{cup-product-Gamma}.

\begin{lemma}
    Let $h$ run over the representatives of $K_{\theta^{-1}}\backslash (\theta^{-1}K\theta^n \cap K \theta K)/K_{\theta^{-n}}$ as in Corollary \ref{double-coset-reps}.  Then the terms in the cup product \eqref{cup-product-Gamma} take the following form (using the determination of the subgroups of the previous lemma):
    \begin{itemize}
        \item $n = 0$: for $h = \sm{\pi^{-1}}{0}{0}{\pi}$, we have 
        \[ a_*\textnormal{res}^{K_\theta}_{K_\theta}(\varphi) = \varphi, \qquad (a\theta c)_*\textnormal{res}^{K_\theta}_{K_\theta}(\varphi) = -\varphi. \]
        Consequently, we have
        \[ \widetilde{\Gamma}_{\theta^0,h} = 0. \]
        \item $n = 1$: for $h = \sm{1}{\pi^{-1}}{0}{1}$, we have
        \[ a_*\textnormal{res}^{K_\theta}_{L}(\varphi) = \varphi_{\textnormal{D}} + \varphi_{\textnormal{L}},\qquad (a\theta c)_*\textnormal{res}^{K_\theta}_{L}(\varphi) = \varphi_{\textnormal{D}} - \varphi_{\textnormal{U}}. \]
        Consequently, we have
        \[ \widetilde{\Gamma}_{\theta,h} = -(\varphi_{\textnormal{D}} \cup \varphi_{\textnormal{U}}) + (\varphi_{\textnormal{L}}\cup \varphi_{\textnormal{D}}) - (\varphi_{\textnormal{L}}\cup \varphi_{\textnormal{U}}). \]
        (Here, we use the notation of Corollary \ref{cor:H1L} and identify $\varphi \in H^1(K_\theta, k)$ with the associated $\mathbb{F}_p$-linear homomorphism $\mathbb{F}_q \longrightarrow k$.)
        \item $n = 2$: for $h = \sm{\pi}{0}{0}{\pi^{-1}}$, we have
        \[ a_*\textnormal{res}^{K_\theta}_{K_{\theta^2}}(\varphi) = \textnormal{res}^{K_\theta}_{K_{\theta^2}}(\varphi),\qquad (a\theta c)_*\textnormal{res}^{K_\theta}_{\theta^{-1}K_{\theta^2}\theta}(\varphi) = \textnormal{res}^{K_\theta}_{K_{\theta^2}}(\varphi). \]
        Consequently, we have
        \[ \widetilde{\Gamma}_{\theta^2,h} = 0. \]
    \end{itemize}
\end{lemma}

\begin{proof}
    The $n = 0$ and $n = 2$ cases follow in a straightforward manner, using the determination of the elements $a,c,d \in K$ from the previous lemma, and the fact that $\varphi$ is supported on $T^0$ (see Corollary \ref{cor:frattini-Klambda}).  We focus on the $n = 1$ case.  Note first that we have $\textnormal{res}^{K_{\theta}}_L(\varphi) = \varphi_{\textnormal{D}}$.  Furthermore, the elements $a = \sm{1}{-\pi}{0}{1}$ and $a\theta c = \sm{1}{0}{-\pi^{-1}}{1}$ both normalize $L$, and we have
    \begin{eqnarray*}
        (a_*\varphi_{\textnormal{D}})\left(\begin{pmatrix} 1 + \pi a & \pi^2b \\ c & 1 + \pi d\end{pmatrix}\cdot (-\textnormal{id})^j\right) & = & \varphi_{\textnormal{D}}\left(\begin{pmatrix}1 & \pi \\ 0 & 1\end{pmatrix}\begin{pmatrix} 1 + \pi a & \pi^2b \\ c & 1 + \pi d\end{pmatrix}\cdot (-\textnormal{id})^j\begin{pmatrix}1 & -\pi \\ 0 & 1\end{pmatrix}\right) \\
        & = & \varphi_{\textnormal{D}}\left(\begin{pmatrix} 1 + \pi(a + c) & \pi^2(- a + b - c + d) \\ c & 1 + \pi(-c + d)\end{pmatrix}\cdot (-\textnormal{id})^j\right) \\
        & = & \varphi(\overline{a}) + \varphi(\overline{c})\\
        ((a\theta c)_*\varphi_{\textnormal{D}})\left(\begin{pmatrix} 1 + \pi a & \pi^2b \\ c & 1 + \pi d\end{pmatrix}\cdot (-\textnormal{id})^j\right) & = & \varphi_{\textnormal{D}}\left(\begin{pmatrix}1 & 0 \\ \pi^{-1} & 1\end{pmatrix}\begin{pmatrix} 1 + \pi a & \pi^2b \\ c & 1 + \pi d\end{pmatrix}\cdot (-\textnormal{id})^j\begin{pmatrix}1 & 0 \\-\pi^{-1} & 1\end{pmatrix}\right) \\
        & = & \varphi_{\textnormal{D}}\left(\begin{pmatrix} 1 + \pi(a - b) & \pi^2b \\ a - b + c - d & 1 + \pi(b + d)\end{pmatrix}\cdot (-\textnormal{id})^j\right) \\
        & = & \varphi(\overline{a}) - \varphi(\overline{b}).
    \end{eqnarray*}
    Thus, we obtain
    \[ a_*\varphi_{\textnormal{D}} = \varphi_{\textnormal{D}} + \varphi_{\textnormal{L}},\qquad (a\theta c)_*\varphi_{\textnormal{D}} = \varphi_{\textnormal{D}} - \varphi_{\textnormal{U}}, \]
    and therefore
    \[ \widetilde{\Gamma}_{\theta,h} = a_*\varphi_{\textnormal{D}} \cup (a\theta c)_*\varphi_{\textnormal{D}} = - (\varphi_{\textnormal{D}} \cup \varphi_{\textnormal{U}}) + (\varphi_{\textnormal{L}}\cup \varphi_{\textnormal{D}}) - (\varphi_{\textnormal{L}}\cup \varphi_{\textnormal{U}}). \]
\end{proof}

We now conclude.  Combining equation \eqref{nongradcomm-square} with the above lemmas shows that $\alpha^2 = \gamma_\theta$, with 
\begin{eqnarray*}
    \textnormal{Sh}_{\theta}(\gamma_\theta) & = &  \textnormal{cores}^L_{K_\theta}\left(\widetilde{\Gamma}_{\theta,\sm{1}{\pi^{-1}}{0}{1}}\right) 
     =  -\textnormal{cores}^L_{K_\theta}(\varphi_{\textnormal{D}} \cup \varphi_{\textnormal{U}}) +\textnormal{cores}^L_{K_\theta}(\varphi_{\textnormal{L}}\cup \varphi_{\textnormal{D}}) - \textnormal{cores}^L_{K_\theta}(\varphi_{\textnormal{L}}\cup \varphi_{\textnormal{U}}) \\
    & = & -\textnormal{cores}^L_{K_\theta}(\varphi_{\textnormal{L}}\cup \varphi_{\textnormal{U}}) \quad
     \neq  0,
\end{eqnarray*}
where the last inequality follows from Lemma \ref{lem:corestrictions}.  Hence $\alpha^2 \neq 0$, and therefore $E_K^*$ is not graded-commutative.

\appendix

\section{\label{app:RC}Proof of Propositions    \ref{prop:RP} and \ref{prop:CP}}

 Recall that $\U$ and $\V$ are two open compact subgroups such that $\U\subseteq \V$.  Proposition \ref{prop:CP} follows from the commutativity of \eqref{f:C1} and \eqref{f:C2} below. For Proposition \ref{prop:RP}, see the end of the section. We first record the following obvious result.
 
 \begin{lemma}
Given an open compact subgroup $\mathcal C$ of $G$, we   have the commutative diagrams
\begin{equation}
\label{H*piP}
\begin{tikzcd}
H^*(\mathcal C, \X_\U) \ar[d, equals] \ar[rrrr, "H^*(\mathcal C{,} \pi_{\U,\V})"] &&& & H^*(\mathcal C, \X_\V)\ar[d,equals] \\
\Ext_{\Mod(G)}^*(\X_{\mathcal C}, \X_\U) \ar[rrrr, "f \mapsto \pi_{\U,\V}\cdot f"] &&&& \Ext_{\Mod(G)}^*(\X_{\mathcal C}, \X_\V)  \ .
\end{tikzcd} 
\end{equation}  

\begin{equation} 
\label{H*iotaP}
\begin{tikzcd}
H^*({\mathcal C}, \X_\V) \ar[d, equals] \ar[rrrr, "H^*({\mathcal C}{,} \iota_{\V,\U})"] &&& & H^*({\mathcal C}, \X_\U)\ar[d, equals] \\
\Ext_{\Mod(G)}^*(\X_{\mathcal C}, \X_\V)
\ar[rrrr, "f \mapsto \iota_{\V,\U}\cdot f"] &&&& \Ext_{\Mod(G)}^*(\X_{\mathcal C}, \X_\U)  \ .
\end{tikzcd}
\end{equation}   
where  the  lower horizontal maps are Yoneda compositions.

\end{lemma}

\begin{lemma} Let $Y$ be a representation of $G$.  We have the following commutative diagrams
\begin{equation}
\label{resPI}
\begin{tikzcd}
H^*(\V, Y) \ar[d, equals] \ar[rrrr, "\res^\V_\U"]&&& & H^*(\U, Y)\ar[d, equals]\\
\Ext_{\Mod(G)}^*(\X_\V, Y) \ar[rrrr, "f \mapsto f\cdot  \pi_{\U,\V}"] &&&& \Ext_{\Mod(G)}^*(\X_\U, Y)  
\end{tikzcd}
\end{equation}  
and 
\begin{equation}
\label{coresPI}
\begin{tikzcd}
H^*(\U, Y) \ar[d, equals] \ar[rrrr, "\cores^\U_\V"] &&& & H^*(\V, Y)\ar[d, equals]\\
\Ext_{\Mod(G)}^*(\X_\U, Y) \ar[rrrr, "f \mapsto f\cdot  \iota_{\V,\U}"] &&&& \Ext_{\Mod(G)}^*(\X_\V, Y)  \ .
\end{tikzcd} 
\end{equation} 
\end{lemma}

\begin{proof}
Let $Y \longrightarrow \mathcal Y^\bullet$ be an injective resolution in $\Mod(G)$ (hence in $\Mod(\U)$ and $\Mod(\V)$ by restriction).  We have the following commutative diagrams:

\begin{equation*}
\begin{tikzcd}
\Hom_\V(k, \mathcal Y^\bullet\vert_\V) \ar[d, equals] \ar[rrrr, "\res^\V_\U"] &&& & \Hom_\U(k, \mathcal Y^\bullet\vert_ \U)\ar[d, equals]\\
\Hom_G(\X_\V, \mathcal Y^\bullet) \ar[rrrr, "f\mapsto f\cdot \pi_{\U,\V}"] &&&& \Hom_G(\X_\U, \mathcal Y^\bullet)  \ .
\end{tikzcd}
\end{equation*}  
and
\begin{equation*}
\begin{tikzcd}
( \mathcal Y^\bullet)^\U\ar[d, equals] \ar[rrrr, "N_{\V/\U}: x\mapsto \sum_{v\in \V/\U} vx"] &&& & (\mathcal{Y}^\bullet)^ \V\ar[d, equals]\\
\Hom_G(\X_\U, \mathcal Y^\bullet) \ar[rrrr, "f\mapsto f\cdot \iota_{\V, \U}"] &&&& \Hom_G(\X_\V, \mathcal Y^\bullet)  \ .
\end{tikzcd} 
\end{equation*}
where the vertical equality signs are the isomorphisms given by Frobenius reciprocity.  Passing to cohomology we obtain the lemma.
\end{proof}

\begin{lemma}
\label{lemma:Shcomp}
Let $\U, \V$ and $\W$ be three open compact subgroups of $G$ satisfying $\W \supseteq \V\supseteq \U$. Let $g\in G$.  Then the following diagram is commutative:
\begin{equation}
\begin{tikzcd}
 H^*(\U\cap g\V g^{-1},k) \ar[d, "H^*(\U \cap g\V g^{-1}{,} \mathrm{i}'_g)"'] \ar[rrrr, "\cores^{\U\cap g\V g^{-1}}_{\V\cap g\W  g^{-1}}"] &&& & H^*(\V\cap g\W  g^{-1},k) \\
H^*(\U\cap g\V g^{-1},\ind_\W ^{\V g \W }(1)) \ar[rrrr, "\cores^{\U\cap g\V g^{-1}}_\V"] &&&&H^*(\V,\ind_\W ^{\V g \W }(1)) \ar[u, "\Sh_g^{\V, \W}"']
\end{tikzcd}
\end{equation} 
where the map $\mathrm{i}_g':k \longrightarrow \ind_{\W}^{\V g\W}(1)$ is the $(\V \cap g\W g^{-1})$-equivariant map defined by $a \longmapsto a\chara_{g\W}$.
\end{lemma}

\begin{proof} 
Suppose that $* = r$. Let $[A]\in H^r(\U\cap g\V g^{-1},k)$, where $A$ is a homogeneous $r$-cocycle, and $[\widetilde{A}] := H^r(\U \cap g\V g^{-1}, \mathrm{i}'_g)([A]) = [A(_-) \chara_{g\W}]\in H^r(\U\cap g\V g^{-1},\ind_\W ^{\V g \W }(1))$.  We fix a system of representatives ${\mathcal X}$ of $(\U\cap g\V g^{-1} )\backslash\V$ and for $p\in \V$, we denote by $\bar p$ the unique element in ${\mathcal X}$ such that $(\U\cap g\V g^{-1}) p=(\U\cap g\V g^{-1})\bar p$.
Let  $(p_0,\ldots,p_r) \in \V^{r + 1}$.  By \cite[\S I.5.4]{NSW}, we have
\begin{eqnarray*}
 \left(\cores^{\U\cap g\V g^{-1} }_\V(\widetilde A)\right)(p_0,\ldots,p_r) & = &  
 \sum_{x\in \mathcal X} x ^{-1}\cdot \widetilde{A}(x p_0 \overline{xp_0}^{-1}, x p_1 \overline{xp_1}^{-1} \ldots, x p_r \overline{xp_r}^{-1}) \\
 & = & \sum_{x\in \mathcal X}  A(x p_0 \overline{xp_0}^{-1}, x p_1 \overline{xp_1}^{-1} \ldots, x p_r \overline{xp_r}^{-1}) \chara _{x^{-1}g\W } \ .
\end{eqnarray*}
Now applying the right-hand side Shapiro map means that we first restrict to  $\V\cap g \W  g^{-1}$ and then evaluate at $g$.  But $g\in x^{-1}g\W $ if and only if $x\in \V \cap g \W  g^{-1}$.  Hence, if we let $\mathcal{Y} := \mathcal{X} \cap \V\cap g\W  g^{-1}$, then $\mathcal{Y}$ is a system of representatives of $ (\U \cap g\V g^{-1} )\backslash (\V\cap g\W  g^{-1})$. Therefore, for $(p_0,\ldots,p_r)\in (\V\cap g \W  g^{-1})^{r + 1}$, we have:
\begin{eqnarray*}
  \Sh_g^{\V, \W}\left(\cores^{\U\cap g\V g^{-1} }_\V(\widetilde{A})\right)(p_0,\ldots,p_r) & = & 
 \sum_{y\in \mathcal Y}   A(y p_0 \overline{yp_0}^{-1}, y p_1 \overline{yp_1}^{-1} \ldots, y p_r \overline{yp_r}^{-1}) \ ,
\end{eqnarray*} 
which shows $ \Sh_g^{\V, \W}\big(\cores^{\U\cap g\V g^{-1} }_\V(\widetilde{A})\big) = \cores^{\U\cap g\V g^{-1} }_{\V\cap g\W  g^{-1}}(A)$.
\end{proof}

\begin{lemma}
Let $g\in G$.
We have the commutative diagram
\begin{equation}
\label{f:C1}
\begin{tikzcd}
\Ext^*_{\Mod (G)} (\X_\U, \X_\V ) \ar[d, equals] \ar[rrr, "f \mapsto f\cdot \iota_{\V,\U}"] &&& \Ext^*_{\Mod (G)} (\X_ \V, \X_ \V )\ar[d, equals]\\
H^*(\U , \X_ \V)\ar[rrr, "\cores^\U _{\V}"] & && H^*(\V, \X_ \V) \\
H^*(\U , \ind_{\V}^{ \U g\,\V}(1)) \ar[d, "\Sh^{\U,\V}_g"] \ar[u, hook] \ar[rrr, "\cores^\U _{\V}\circ H^*(\U {,} \upsilon_g)"] & && \ar[d, "\Sh_g^{\V}"] \ar[u, hook]\ H^*(\V, \X_ \V(g)) \\
H^*(\U \cap \V_g,k) \ar[rrr, "\cores^{\U \cap \V_g}_{\V_g}"] & && H^*(\V_g, k) 
\end{tikzcd}
\end{equation} 
where $\upsilon_g$ is the natural $\U $-equivariant inclusion $\ind_{\V}^{ \U g\,\V}(1) \longhookrightarrow \X_ \V(g)$.  Further, we have the commutative diagram
\begin{equation}
\label{f:C2}
\begin{tikzcd}
\Ext^*_{\Mod (G)} (\X_\U, \X_\U ) \ar[rrr, "f \mapsto \pi_{\U,\V}\cdot f "] \ar[d, "{\rotatebox{90}{$\sim$}}"] &&& \Ext_{\Mod(G)}^*(\X_\U, \X_\V) \ar[d, "{\rotatebox{90}{$\sim$}}"]\\
H^*(\U, \X_{\U})\ar[rrr, "{H^*(\U{,} \pi_{\U,\V})}"] & && H^*(\U, \X_\V)  \\
H^* (\U, \X_\U(g) ) \ar[u, hook] \ar[d, "\Sh_g^{\U}"] \ar[rrr, "H^*(\U{,} \pi_{\U,\V})"] &&&  H^*(\U, \ind_{\V }^{\U g \V }(1))\ar[u, hook] \ar[d, "\Sh_g^{\U,\V}"] \\
H^*(\U_g, k) \ar[rrr, "\cores^{\U_g}_{\U\cap \V_g}"] & && H^*(\U\cap \V_g, k)  \ .
\end{tikzcd}
\end{equation} 
\end{lemma}

\begin{proof}
We first prove the commutativity of \eqref{f:C1}.  The commutativity of the top and middle rectangles follows from \eqref{coresPI}.  To prove the commutativity of the remaining rectangle, we consider the following diagram:
\begin{center}
\begin{tikzcd}
H^*(\U \cap \V_g, k) \ar[rrrr, "\cores_{\V_g}^{\U \cap \V_g}"] \ar[dd, "H^*(\U \cap \V_g{,} \textrm{i}_g)"'] \ar[ddrr, "H^*(\U\cap \V_g{,} \textrm{i}_g')"]  & & & & H^*(\V_g, k) \\
 &   & & & \\
 H^*(\U \cap \V_g, \ind_{\V}^{\U g\V}(1)) \ar[rr, "H^*(\U \cap \V_g{,} \upsilon_g)"] \ar[d, "\cores_{\U}^{\U \cap \V_g}"'] & & H^*(\U \cap \V_g, \ind_{\V}^{\V g\V}(1)) \ar[ddll, "\cores_{\U}^{\U \cap \V_g}"] \ar[ddrr, "\cores_{\V}^{\U \cap \V_g}"'] & & \\
 H^*(\U , \ind_{\V}^{\U g\V}(1)) \ar[d, "H^*(\U{,} \upsilon_g)"'] & &  & & \\
 H^*(\U, \ind_{\V}^{\V g\V}(1)) \ar[rrrr, "\cores_{\V}^{\U}"]  & &  & &  H^*(\V, \ind_{\V}^{\V g\V}(1)) \ar[uuuu, "\textnormal{Sh}^{\V}_g"']
\end{tikzcd}
\end{center}
Here, $\textnormal{i}_g: k \longrightarrow \ind_{\V}^{\U g\V}(1)$ is the $(\U \cap g\V g^{-1})$-equivariant map defined by $a \longmapsto a\chara_{g\V}$, and $\textnormal{i}_g': k \longrightarrow \ind_{\V}^{\V g\V}(1)$ is the $(\V \cap g\V g^{-1})$-equivariant map defined by $a \longmapsto a\chara_{g\V}$.  We have:
\begin{itemize}
\item The upper left triangle is commutative by the relation $\upsilon_g\circ \textnormal{i}_g = \textnormal{i}_g'$.
\item The lower left triangle is commutative by functoriality of the corestriction map.
\item The lower triangle is commutative by transitivity of the corestriction map.
\item The upper right triangle is commutative by applying Lemma \ref{lemma:Shcomp} to $(\W,\V, \U) = (\V, \V, \U)$.
\end{itemize}
Consequently, we obtain the commutativity of the following diagram:
\begin{center}
\begin{tikzcd}
H^*(\U \cap \V_g, k) \ar[rrr, "\cores_{\V_g}^{\U \cap \V_g}"] \ar[d, "\cores^{\U \cap \V_g}_{\U} \circ H^*(\U \cap \V_g{,} \textrm{i}_g)"']  & & & H^*(\V_g, k) \\
 H^*(\U, \ind_{\V}^{\U g\V}(1)) \ar[rrr, "\cores_{\V}^{\U}\circ H^*(\U{,}\upsilon_g)"]  & & &  H^*(\V, \ind_{\V}^{\V g\V}(1)) \ar[u, "\textnormal{Sh}^{\V}_g"']
\end{tikzcd}
\end{center}
By \eqref{f:Shapiro-inversehybrid}, the left vertical map is exactly $(\textnormal{Sh}_g^{\U,\V})^{-1}$, which concludes the proof for  \eqref{f:C1}.

Now we turn to  the commutativity of \eqref{f:C2}.  The commutativity of the top and middle rectangles follows from the definitions and diagram \eqref{H*piP}.  
Note first that we have a commutative diagram
\begin{center}
\begin{tikzcd}
& H^*(\U_g, \ind_{\U}^{\U g\U}(1)) \ar[rrr, "\cores_{\U}^{\U_g}"] \ar[dd, "H^*(\U_g{,} \pi_{\U,\V})"]& & & H^*(\U, \ind_{\U}^{\U g\U}(1)) \ar[dd, "H^*(\U{,} \pi_{\U,\V})"] \\
H^*(\U_g,k) \ar[ur, "H^*(\U_g{,} \textrm{i}_g')"] \ar[dr, "H^*(\U_g{,} \textrm{i}_g)"']& & & & \\
& H^*(\U_g, \ind_{\V}^{\U g\V}(1)) \ar[rrr, "\cores_{\U}^{\U_g}"] & & & H^*(\U, \ind_{\V}^{\U g\V}(1)) 
\end{tikzcd}
\end{center}
Here, $\textnormal{i}_g: k \longrightarrow \ind_{\V}^{\U g\V}(1)$ is the $(\U \cap g\V g^{-1})$-equivariant map defined by $a \longmapsto a\chara_{g\V}$, and $\textnormal{i}_g': k \longrightarrow \ind_{\U}^{\U g\U}(1)$ is the $(\U \cap g\U g^{-1})$-equivariant map defined by $a \longmapsto a\chara_{g\U}$.  The commutativity of the triangle follows from the relation $\pi_{\U,\V} \circ \textnormal{i}_g' = \textnormal{i}_g$, and the commutativity of the square follows from functoriality of the corestriction map.  Thus, by \eqref{f:Shapiro-inversehybrid}, we obtain a commutative diagram
\begin{center}
\begin{tikzcd}
H^*(\U_g, k) \ar[rrr, "(\textnormal{Sh}_g^{\U})^{-1}"] \ar[d, "H^*(\U_g{,}\textrm{i}_g)"'] & & & H^*(\U, \ind_{\U}^{\U g\U}(1)) \ar[d, "H^*(\U{,} \pi_{\U,\V})"] \\
H^*(\U_g, \ind_{\V}^{\U g\V}(1)) \ar[rrr, "\cores^{\U_g}_{\U}"] & & & H^*(\U, \ind_{\V}^{\U g \V}(1))
\end{tikzcd}
\end{center}
On the other hand, taking $(\W,\V, \U) = (\V,\U,\U)$ in Lemma \ref{lemma:Shcomp} gives the following commutative diagram:
\begin{center}
\begin{tikzcd}
 H^*(\U_g,k) \ar[d, "H^*(\U_g{,} \mathrm{i}_g)"'] \ar[rrrr, "\cores^{\U_g}_{\U\cap g\V  g^{-1}}"] &&& & H^*(\U\cap g\V  g^{-1},k) \\
H^*(\U_g,\ind_\V ^{\U g \V }(1)) \ar[rrrr, "\cores^{\U_g}_\U"] &&&&H^*(\U,\ind_\V ^{\U g \V }(1)) \ar[u, "\Sh_g^{\U, \V}"']
\end{tikzcd}
\end{center}
Combining the two previous commutative diagrams gives the commutativity of the bottom rectangle of \eqref{f:C2}, and concludes the proof of the lemma.
\end{proof}

\begin{proof}[Proof of Proposition \ref{prop:RP}]

We recall that when $g\in G$ and $y=v g v'\in \V g \V$ we have $v \V_g v^{-1}= \V_y$ and the map
\begin{eqnarray*}
H^*(\V_{g}, k) & \longrightarrow & H^*(\U_y,k), \\
 a & \longmapsto & \res^{ \V_{y}}_{\U_y} (v_*a)
\end{eqnarray*}
is well-defined and independent of the choice of $v$.

The commutativity of the top two rectangles in \eqref{f:defiRP*} comes from \eqref{H*iotaP} and \eqref{resPI}, and functoriality of the restriction map. To prove commutativity of the remaining portion, let $A\in H^*(\V, \X_\V(g))$ and $y\in \V g\V$ as above. We have \begin{eqnarray*}
\res^{\V_y}_{\U_y} \left(v_*\Sh^\V_{g}(A)\right) & \stackrel{\textnormal{Lem. \ref{lemma:shapindep}}}{=} & \res^{\V_y}_{\U_y} \left(\Sh^\V_{y}(A)\right) \\
 & \stackrel{\eqref{f:Shapiro1}}{=} & \res^{\V_y}_{\U_y} \left( H^*(\V_y, \ev^\V_y) \circ\res^{\V}_{\V_y}(A) \right)\\
 & = & H^*(\U_y, \ev^\V_y) \circ  \res^{\V_y}_{\U_y} \circ \res^{\V}_{\V_y}(A) \\
 & = & H^*(\U_y, \ev^\V_y) \circ \res^{\V}_{\U_y}(A) \ ;
\end{eqnarray*}
in particular, the first equality shows that the lower left triangle commutes.  On the other hand, note that the restriction of $H^*(\U,\iota_{\V,\U})\circ \res_{\U}^{\V} : H^*(\V, \X_\V) \longrightarrow H^*(\U,\X_\U)$ to $H^*(\V, \X_\V(g))$ factors through $\bigoplus_{y\in  {\U\backslash \V g\V/\U}} H^*(\U ,\X_\U(y))$.  Consequently, we may apply the map $\Sh_y^{\U}$ to each component to obtain
\begin{eqnarray*}
\Sh_y^{\U}\left(  H^*(\U, \iota_{\V,\U}) \circ \res^\V_\U (A) \right) & \stackrel{\eqref{f:Shapiro1}}{=} & H^*(\U_y, \ev^\U_y) \circ \res^\U_{\U_y} \circ H^*(\U, \iota_{\V,\U}) \circ \res^\V_\U (A) \\
& = & H^*(\U_y, \ev^\U_y)  \circ H^*(\U_y, \iota_{\V,\U})  \circ \res^\U_{\U_y} \circ \res^\V_\U (A) \\
& = & H^*(\U_y, \ev^\U_y \circ \iota_{\V,\U}) \circ \res^\V_{\U_y} (A) \\
& = & H^*(\U_y, \ev^\V_y) \circ \res^\V_{\U_y} (A) \ .
\end{eqnarray*}
(For the final equality, we note that $\ev^\U_y \circ \iota_{\V,\U} = \ev^\V_y$.)  Combining the two sets of equalities verifies the desired commutativity and concludes the proof.
\end{proof}

\section{Frattini quotients for parahoric subgroups \label{appendix:frattini}}
Our goal will be to calculate some cohomology groups appearing in the body of the article.  We maintain the notation introduced above: $\mathfrak{F}$ is a finite extension of $\mathbb{Q}_p$ with ring of integers $\mathfrak{O}$, maximal ideal $\mathfrak{M}$, uniformizer $\pi$, and residue field $\mathbb{F}_q$ of size $q = p^f$.  We suppose that $\mathbf{G}$ is a split connected reductive group over $\mathfrak{F}$, and let $\mathbf{T}$ denote a split maximal torus of $\mathbf{G}$.  We then let $\mathbf{G}_{x_0}$ denote the $\mathfrak{O}$-group scheme associated to a hyperspecial point in the apartment $\mathscr{A}$ of the Bruhat--Tits building of $\mathbf{G}$ corresponding to $\mathbf{T}$.  Set $G := \mathbf{G}(\mathfrak{F}), K:= \mathbf{G}_{x_0}(\mathfrak{O})$.  Further, we let $J \subseteq K$ denote a choice of Iwahori subgroup contained in $K$, corresponding to a chamber $C$ containing $x_0$ in its closure.

We assume henceforth that $q \neq 2,3$.

Given a profinite group $\mathcal{G}$ and elements $g,h \in \mathcal{G}$, we let 
\[ [g,h] := ghg^{-1}h^{-1},\]
and let $[\mathcal{G},\mathcal{G}]$ denote the closure of the subgroup generated by all commutators $[g,h]$.  Further, we let $\Phi(\mathcal{G}) := [\mathcal{G},\mathcal{G}]\mathcal{G}^p$ denote the Frattini subgroup of $\mathcal{G}$, and let $\mathcal{G}_\Phi := \mathcal{G}/\Phi(\mathcal{G})$ denote the Frattini quotient.

We will make use of the various matrix identities in $\textnormal{SL}_2$.  We have:
\begin{eqnarray}
        \left[\begin{pmatrix} x & 0 \\ 0 & x^{-1} \end{pmatrix},\begin{pmatrix} 1 & y \\ 0 & 1\end{pmatrix}\right] & = & \begin{pmatrix} 1 & (x^2 - 1)y \\ 0 & 1 \end{pmatrix} \label{comm-1}\\
        \left[\begin{pmatrix} x & 0 \\ 0 & x^{-1} \end{pmatrix},\begin{pmatrix} 1 & 0 \\ z & 1\end{pmatrix}\right] & = & \begin{pmatrix} 1 & 0 \\(x^{-2} - 1)z & 1 \end{pmatrix} \label{comm-2} \\
     \begin{pmatrix} 1 - xy & 0 \\ 0 & (1 - xy)^{-1}\end{pmatrix} & = & \begin{pmatrix} 1 & 0 \\ x(1 - xy)^{-1} & 1\end{pmatrix}\begin{pmatrix} 1 & y \\ 0 & 1\end{pmatrix} \begin{pmatrix}1 & 0 \\ -x & 1\end{pmatrix} \begin{pmatrix}1 & - y(1 - xy)^{-1} \\ 0 & 1\end{pmatrix} \label{coroot}
\end{eqnarray}

We also recall the following construction from \cite{Sch-St} (see also \cite[\S 6]{BT1}).   Let $Y$ denote a non-empty subset of $\mathscr{A}$, and let $f_Y: \Phi \longrightarrow \mathbb{R} \cup \{\infty\}$ denote the concave function given by
\[ f_Y(\alpha) = - \inf_{y \in Y}\{\langle\alpha, y\rangle\}. \]
Note that by \cite[Prop. 6.4.5]{BT1}, we have $f_Y(\alpha) + f_Y(-\alpha) \geq 0$.  We then set 
\begin{eqnarray*}
\U_{\alpha, f_Y(\alpha)} & := & x_\alpha\left(\mathfrak{M}^{\lceil f_{Y}(\alpha)\rceil}\right), \\
\U_{Y} & := & \big\langle \U_{\alpha, f_Y(\alpha)}\big\rangle_{\alpha \in \Phi}, \\
\P_{Y} & := & \big\langle \U_{\alpha, f_Y(\alpha)},~ T^0\big\rangle_{\alpha \in \Phi},
\end{eqnarray*}
Note that $\P_Y$ is open in $G$, though in general $\U_Y$ is not.

We begin with the subgroup $\P_Y$.

\begin{lemma}
\label{lem:comm-PY}
Suppose $Y$ is a non-empty subset of $\mathscr{A}$.  We then have
\[ [\P_Y,\P_Y] = \U_Y. \]
Thus, we get 
\[ (\P_Y)^{\textnormal{ab}} = \P_Y/[\P_Y,\P_Y] \cong T^0/(T^0 \cap \U_Y). \]
\end{lemma}

\begin{proof}
Fix $\alpha \in \Phi$.  Applying the morphism $\varphi_\alpha$ to equation \eqref{comm-1} gives
\[ \left[\check{\alpha}(x), x_\alpha\left(\pi^{\lceil f_Y(\alpha)\rceil}y\right)\right] = x_\alpha\left((x^2 - 1)\pi^{\lceil f_Y(\alpha)\rceil}y\right). \]
If we choose $x \in \mathfrak{O}^\times$ such that $x^2 - 1 \in \mathfrak{O}^\times$ (which is possible thanks to the assumption $q > 3$), then we see that $x_\alpha(\mathfrak{M}^{\lceil f_Y(\alpha)\rceil}) \in [\P_Y,\P_Y]$ for every $\alpha \in \Phi$.  Consequently $\U_Y \subseteq [\P_Y, \P_Y]$.  On the other hand, since $T^0$ normalizes each $\U_{\alpha,f_{Y}(\alpha)}$, the subgroup $\U_Y$ is normal in $\P_Y$, and we have $\P_Y = \U_Y T^0$.  Hence,
\[ \P_Y/\U_Y \cong \U_YT^0/\U_Y \cong T^0/(T^0 \cap \U_Y). \]
Since this quotient is abelian, we conclude $[\P_Y,  \P_Y] \subseteq \U_Y$.
\end{proof}

\begin{corollary}
\label{cor:frattini-PY}
Suppose $Y$ is a non-empty subset of $\mathscr{A}$.  We then have
\[ \Phi(\P_Y) = \U_Y\mathbf{T}(\mathbb{F}_q)(T^1)^p\ . \]
Thus, we get
\[ (\P_Y)_\Phi = \P_Y/[\P_Y,\P_Y]\P_Y^p \cong T^1/\big\langle T^1 \cap \U_Y,~ (T^1)^p \big\rangle \ . \]
\end{corollary}

\begin{proof}
Using the fact that $T^0$ normalizes $\U_Y$, and the fact that $\mathbf{T}(\mathbb{F}_q)$ has order prime to $p$, we get
\[ [\P_Y, \P_Y]\P_Y^p =  \U_Y(\U_Y\mathbf{T}(\mathbb{F}_q)T^1)^p = \U_Y\mathbf{T}(\mathbb{F}_q)(T^1)^p. \]
\end{proof}

Next, we more explicitly describe the subgroup $T^1 \cap \U_Y$.  

\begin{lemma}
\label{lem:UYintT}
Suppose $Y$ is a non-empty subset of $\mathscr{A}$.  We then have 
\[ T^0 \cap \U_Y = \left\langle \check{\alpha}(1 + \mathfrak{M}^{\lceil f_Y(\alpha)\rceil + \lceil f_Y(-\alpha)\rceil})\right\rangle_{\alpha\in \Phi^+} \]
 (where we define $1 + \mathfrak{M}^0 := \mathfrak{O}^\times$).  Thus, we get
 \[ T^1 \cap \U_Y = \left\langle \check{\alpha}(1 + \mathfrak{M}^{\max\{1,~\lceil f_Y(\alpha)\rceil + \lceil f_Y(-\alpha)\rceil\}})\right\rangle_{\alpha\in \Phi^+}~. \]
\end{lemma}

\begin{proof}
According to \cite[\S 6.4.2, Prop. 6.4.17]{BT1}, the group $T^0 \cap \U_Y$ is generated by $T^0 \cap \U_Y^{(\alpha)}$ for $\alpha \in \Phi^+$, where $\U_Y^{(\alpha)}$ is the subgroup generated by $\U_{\alpha, f_Y(\alpha)}$ and $\U_{-\alpha,f_Y(-\alpha)}$.  We claim that $T^0 \cap \U_Y^{(\alpha)} = \check{\alpha}(1 + \mathfrak{M}^{\lceil f_Y(\alpha)\rceil + \lceil f_Y(-\alpha)\rceil})$.  By pulling back via $\varphi_\alpha$, we are reduced to proving
\begin{equation}
\label{eqn:Torusgenpullback}
 T_{\textnormal{SL}_2}^0 \cap \left\langle\begin{pmatrix}1 & \mathfrak{M}^{\lceil f_Y(\alpha) \rceil} \\ 0 & 1\end{pmatrix},~\begin{pmatrix}1 & 0 \\ \mathfrak{M}^{\lceil f_Y(-\alpha) \rceil} & 1\end{pmatrix}\right\rangle = \begin{pmatrix} 1 + \mathfrak{M}^{\lceil f_Y(\alpha) \rceil + \lceil f_Y(-\alpha) \rceil} \\ 0 & (1 + \mathfrak{M}^{\lceil f_Y(\alpha) \rceil + \lceil f_Y(-\alpha) \rceil})^{-1} \end{pmatrix} 
 \end{equation}
Equation \eqref{coroot} implies that the right-hand side is contained in the left-hand side.  To prove the reverse inclusion, we may conjugate by diagonal elements of $\textnormal{GL}_2(\mathfrak{F})$ and assume $\lceil f_Y(\alpha) \rceil = 0$. If $\lceil f_Y(-\alpha) \rceil = 0$, then we have
\[ \left\langle\begin{pmatrix}1 & \mathfrak{O} \\ 0 & 1\end{pmatrix},~\begin{pmatrix}1 & 0 \\ \mathfrak{O} & 1\end{pmatrix}\right\rangle = \textnormal{SL}_2(\mathfrak{O}),\]
and equality \eqref{eqn:Torusgenpullback} follows in this case.  Thus, let us assume $\lceil f_Y(-\alpha) \rceil \geq 1$.  We have
\[\left\langle\begin{pmatrix}1 & \mathfrak{O} \\ 0 & 1\end{pmatrix},~\begin{pmatrix}1 & 0 \\ \mathfrak{M}^{\lceil f_Y(-\alpha) \rceil} & 1\end{pmatrix}\right\rangle = \begin{pmatrix} 1 + \mathfrak{M}^{\lceil f_Y(-\alpha) \rceil} & \mathfrak{O} \\ \mathfrak{M}^{\lceil f_Y(-\alpha) \rceil} & 1 + \mathfrak{M}^{\lceil f_Y(-\alpha) \rceil}\end{pmatrix};\]
indeed, the group on the right-hand side is contained in group on the left-hand side by using the Iwahori decomposition (and the inclusion in \eqref{eqn:Torusgenpullback} already shown).  Intersecting with $T_{\textnormal{SL}_2}^0$ gives the claim and finishes the proof.
\end{proof}

We may now apply the previous results to calculate Frattini quotients of various parahoric subgroups.

\begin{corollary}
\label{cor:frattini-K}
The group $\Phi(K)$ is given by
\[ \Phi(K) = \big\langle x_\alpha(\mathfrak{O}),~ (T^1)^p,~ \mathbf{T}(\mathbb{F}_q),~ \check{\alpha}(1 + \mathfrak{M})\big\rangle_{\alpha \in \Phi^+}. \]
Thus, we get 
\[ K_\Phi \cong T^1/\big\langle (T^1)^p,~ \check{\alpha}(1 + \mathfrak{M})\big\rangle_{\alpha \in \Phi^+}. \]
In particular, if $\mathbf{G}$ is semisimple and simply connected, we have $\Phi(K) = K$ and $K_\Phi  = 1$.
\end{corollary}

\begin{proof}
This follows from Corollary \ref{cor:frattini-PY} and Lemma \ref{lem:UYintT}, noting that $K = \P_{x_0}$, so that $f_{x_0}(\alpha) = 0$ for all $\alpha \in \Phi$.  
\end{proof}

Next, we slightly generalize the above lemma.  Let $\lambda \in X_*(T)$ be a cocharacter.  Then, by \cite[Prop. 6.2.10(iii)]{BT1} we have $\lambda(\pi)K\lambda(\pi)^{-1} = \lambda(\pi)\P_{x_0}\lambda(\pi)^{-1} = \P_{\lambda(\pi)\cdot x_0}$.  Thus, we get
\[ K_{\lambda(\pi)} = K \cap \lambda(\pi)K\lambda(\pi)^{-1} = \P_{x_0} \cap \P_{\lambda(\pi)\cdot x_0} = \P_{\{x_0,~ \lambda(\pi)\cdot x_0\}}. \]

\begin{corollary}
\label{cor:frattini-Klambda}
Let $\lambda \in X_*(T)$.  The group $\Phi(K_{\lambda(\pi)})$ is given by
\[ \Phi(K_{\lambda(\pi)}) = \big\langle x_{\alpha}(\mathfrak{M}^{\max\{0, \langle \lambda, \alpha\rangle\}}),~(T^1)^p,~ \mathbf{T}(\mathbb{F}_q),~\check{\alpha}(1 + \mathfrak{M}^{|\langle\lambda, \alpha\rangle|})\big\rangle_{\alpha \in \Phi^+} \]
 (where we define $1 + \mathfrak{M}^0 := \mathfrak{O}^\times$).  Thus, we get 
\[ (K_{\lambda(\pi)})_\Phi \cong T^1/\big\langle (T^1)^p,~ \check{\alpha}(1 + \mathfrak{M}^{\max\{1, |\langle\lambda, \alpha\rangle|\}})\big\rangle_{\alpha \in \Phi^+}. \]
\end{corollary}

\begin{proof}
This follows from Corollary \ref{cor:frattini-PY} and Lemma \ref{lem:UYintT}, using the equality $K_{\lambda(\pi)} = \P_{\{x_0,~ \lambda(\pi)\cdot x_0\}}$ and the fact that 
\[ \lceil f_{\{x_0, \lambda(\pi)\cdot x_0\}}(\alpha)\rceil + \lceil f_{\{x_0, \lambda(\pi)\cdot x_0\}}(-\alpha)\rceil = |\langle \lambda, \alpha\rangle|. \]
\end{proof}

\begin{corollary}
\label{cor:frattini-longest}
Suppose $\mathbf{G}$ is almost simple and simply connected, and let $\beta \in \Phi^+$ denote the highest root. Then the group $\Phi(K_{\check{\beta}(\pi)})$ is given by 
\[ \Phi(K_{\check{\beta}(\pi)}) = \big\langle x_{\alpha}(\mathfrak{M}^{\max\{0, \langle\check{\beta},\alpha\rangle\}}),~(T^1)^p,~ \mathbf{T}(\mathbb{F}_q),~\check{\alpha}(1 + \mathfrak{M}^{|\langle\check{\beta}, \alpha\rangle|})\big\rangle_{\alpha \in \Phi^+}. \]
Thus, we get
\[ (K_{\check{\beta}(\pi)})_\Phi  =  \begin{cases} \check{\beta}(1 + \mathfrak{M})/\check{\beta}(1 + \mathfrak{M}^2) & \textnormal{if $\Phi$ is of type $A_1$,} \\ 1 & \textnormal{otherwise}.\end{cases} \]
\end{corollary}

\begin{proof}
By \cite[Ch. VI, \S 1.8, Prop. 18(iv)]{Bki-LA}, if $\alpha \in \Phi^+$ is not proportional to $\beta$, then we have $\langle \check{\beta}, \alpha\rangle \in \{0,1\}$, which implies $\max\{1, |\langle \check{\beta}, \alpha\rangle|\} = 1$. Therefore, the previous corollary implies
\begin{eqnarray*}
(K_{\check{\beta}(\pi)})_\Phi & \cong & T^1/\big\langle (T^1)^p,~ \check{\alpha}(1 + \mathfrak{M}^{\max\{1, |\langle\check{\beta},\alpha\rangle|\}})\big\rangle_{\alpha \in \Phi^+}\\
 & = & T^1/\big\langle (T^1)^p,~ \check{\beta}(1 + \mathfrak{M}^2),~\check{\alpha}(1 + \mathfrak{M})\big\rangle_{\alpha \in \Phi^+ \smallsetminus\{ \beta\}}\\
\end{eqnarray*}
If $\beta \in \Pi$, then $\Phi$ is of type $A_1$, and we obtain $(K_{\check{\beta}(\pi)})_\Phi = T^1/((T^1)^p, \check{\beta}(1 + \mathfrak{M}^2))$.  On the other hand, if $\beta \not\in \Pi$, then $\Phi$ is not of type $A_1$.  By writing $\check{\beta} = \sum_{\gamma \in \Pi}m_\gamma \check{\gamma}$ as a sum of simple coroots, we have (for $x\in \mathfrak{O}$)
\[ \check{\beta}(1 + \pi x) = \prod_{\gamma \in \Pi}\check{\gamma}(1 + \pi x)^{m_\gamma} \in \langle (T^1)^p, \check{\alpha}(1 + \mathfrak{M})\rangle_{\alpha \in \Phi^+ \smallsetminus \{\beta\}}, \]
which shows $(K_{\check{\beta}(\pi)})_\Phi = T^1/\big\langle (T^1)^p,~ \check{\alpha}(1 + \mathfrak{M})\big\rangle_{\alpha \in \Phi^+}.$ Finally, since $\mathbf{G}$ is simply connected, we have $X_*(T) = \bigoplus_{\gamma \in \Pi}\bbZ\check{\gamma}$.  Therefore, $T^1$ is generated by $\check{\gamma}(1 + \mathfrak{M})$ for $\gamma \in \Pi$, which concludes the proof when $\Phi$ is not of type $A_1$.  An application of the binomial theorem shows that $(1 + \mathfrak{M})^p \subseteq 1 + \mathfrak{M}^{\min\{e + 1, p\}} \subseteq 1 + \mathfrak{M}^2$, where $e$ denotes the ramification index of $\mathfrak{F}$ over $\mathbb{Q}_p$.  This gives the result when $\Phi$ is of type $A_1$. 
\end{proof}

Next, we consider the Iwahori subgroup $J$, which is defined as $\P_{C}$.  Given $w \in W$, \cite[Prop. 6.2.10(iii)]{BT1} once again implies
\[ J_w = J \cap wJw^{-1} = \P_C \cap w\P_C w^{-1} =  \P_C \cap \P_{w\cdot C} = \P_{C \cup w\cdot C}.\]

\begin{corollary}
\label{cor:frattini-Jw}
Let $w \in W$.  Then the group $\Phi(J_w)$ is given by
\[ \Phi(J_w) = \big\langle U_{\alpha, f_{C \cup w\cdot C}(\alpha)},~ (T^1)^p,~\mathbf{T}(\mathbb{F}_q),~ \check{\alpha}(1 + \mathfrak{M}^{\lceil f_{C \cup w\cdot C}(\alpha)\rceil + \lceil f_{C \cup w\cdot C}(-\alpha)\rceil}) \big\rangle_{\alpha \in \Phi}. \]
Thus, we get 
\[ (J_w)_{\Phi} \cong T^1/\big\langle (T^1)^p,  \check{\alpha}(1 + \mathfrak{M}^{\lceil f_{C \cup w\cdot C}(\alpha)\rceil + \lceil f_{C \cup w\cdot C}(-\alpha)\rceil}) \big\rangle_{\alpha \in \Phi}.\]
In particular, when $w = 1$ and $\mathbf{G}$ is semisimple and simply connected, we get
\[J_\Phi = 1 \ .\]
\end{corollary}

\begin{proof}
This follows from Corollary \ref{cor:frattini-PY} and Lemma \ref{lem:UYintT}, using the fact that $\lceil f_{C \cup w\cdot C}(\alpha)\rceil + \lceil f_{C \cup w\cdot C}(-\alpha)\rceil \geq 1$.  The last claim follows from $\lceil f_{C}(\alpha)\rceil + \lceil f_{C}(-\alpha)\rceil = 1$ (and the fact that the coroots $\check{\alpha}$ span the cocharacter lattice, as in the previous proof).  
\end{proof}

\begin{remark}
\label{rem:H1parahoric}
Suppose $Y$ is a non-empty subset of $\mathscr{A}$.  We have
\[ H^1(\P_Y,k) = \Hom(\P_Y,k) = \Hom(\P_Y/(\P_Y)_\Phi,k),\]
and Corollary \ref{cor:frattini-PY} shows that any such homomorphism is determined by its restriction to $T^0$.  In particular, this applies to the groups $K, K_{\lambda(\pi)}$ and $J_w$ considered above. 
\end{remark}

\section{\for{toc}{On mod $p$ orientation characters (by K. Kozio{\l} and D. Schwein)}\except{toc}{On mod $p$ orientation characters}}
\label{app:orientation}

\begin{center}by \textsc{Karol Kozio{\l} and David Schwein}\footnote{We thank Loren Spice for several useful comments.}\end{center}

Suppose $\mathfrak{F}$ is a finite extension of $\mathbb{Q}_p$ and $G$ the group of $\mathfrak{F}$-points of a connected reductive $\mathfrak{F}$-group.  In this appendix, we prove that the mod $p$ orientation character of $G$ is trivial, giving a different proof of a result of Schneider--Sorensen.

\subsection{Introduction}

We maintain the same notation as in the body of the article.  Thus, we let $\mathfrak{F}$ denote a finite extension of $\bbQ_p$ and $\mathbf{G}$ a connected reductive group over $\mathfrak{F}$, with $G := \mathbf{G}(\mathfrak{F})$ its group of $\mathfrak{F}$-rational points.  We will often conflate algebraic groups over $\mathfrak{F}$ with their groups of $\overline{\mathfrak{F}}$-points.

Given a point $x$ in the semisimple Bruhat--Tits building of $G$ and a real number $r > 0$, we denote by $G_{x,r}$ the depth-$r$ Moy--Prasad subgroup.  (We note that the assumption $r > 0$ guarantees that $G_{x,r}$ is pro-$p$.)  We will at various points make the assumption that $G_{x,r}$ is torsion-free\footnote{This condition is satisfied (for all $x$ and all $r > 0$) if $p$ is sufficiently large relative to the rank of $\mathbf{G}$ and the absolute ramification index of $\mathfrak{F}$; for an explicit bound (at least when $\mathbf{G}$ is semisimple), see \cite[Prop. 12.1]{totaro:eulerchar}.}.  This assumption guarantees that the cohomological dimension of $G_{x,r}$ is finite, equal to $d := \dim_{\bbQ_p}(G)$ (see \cite[Cor. (1)]{serre:cohdim}).

Our goal will be to prove the following:

\begin{theorem}\label{mainthm:appendix}
Let $x$ be a point in the semisimple Bruhat--Tits building of $G$ and $r > 0$.  Assume $G_{x,r}$ is torsion-free, and suppose $g \in G$ satisfies $g\cdot x = x$.  Then the adjoint action of $g$ on $H^d(G_{x,r},\bbF_p)$ is trivial.
\end{theorem}

\subsection{Preparation}

In order to prove the theorem, we make a few reductions and preparations.  Firstly, replacing $\mathbf{G}$ by the Weil restriction $\textnormal{Res}_{\mathfrak{F}/\bbQ_p}(\mathbf{G})$, we may assume $\mathfrak{F} = \bbQ_p$.  The compatibility of Moy--Prasad subgroups is given by \cite[Eqn. (A.1)]{aubertplymen}, whose proof works for an arbitrary (not necessarily Galois) finite separable extension $E/F$\footnote{More precisely, the proofs in \cite[App. A]{aubertplymen} as written work when $E/F$ is totally ramified and not necessarily Galois.  Factoring an arbitrary finite separable extension into a tower of an unramified extension and a totally ramified extension gives the result in general.}.

Next, we recall some properties of Moy--Prasad filtrations \cite{moyprasad2} which we will need.  For each point $x$ in the semisimple Bruhat--Tits building and each real number $r > 0$, we let $G_{x,r}$ denote the depth-$r$ Moy--Prasad subgroup.  Further, we let $\boldsymbol{\mathfrak{g}}$ denote the Lie algebra of $\mathbf{G}$, and let $\mathfrak{g}_{x,r}$ denote the depth-$r$ Moy--Prasad Lie sublattice.  It is a free $\bbZ_p$-module whose $\bbQ_p$-span is $\boldsymbol{\mathfrak{g}}(\bbQ_p)$.  These filtrations satisfy the following properties (see \cite[\S\S~3.2, 3.3]{moyprasad2} or \cite[\S\S~ 13.2, 13.3]{kalethaprasad}):
\begin{itemize}
\item If $r > 0$, then $G_{x,r}$ is an open pro-$p$ subgroup of $G$.
\item If $r \leq s$, then $G_{x,s}$ is a normal subgroup of $G_{x,r}$.
\item The set $\{G_{x,r}\}_{r > 0}$ gives a neighborhood basis of the identity.
\item For $g \in G$, we have $\Ad(g)(G_{x,r}) = G_{g\cdot x,r}$ and $\Ad(g)(\mathfrak{g}_{x,r}) = \mathfrak{g}_{g \cdot x,r}$.  In particular, if $g\cdot x = x$, then $\Ad(g)$ preserves $G_{x,r}$ and $\mathfrak{g}_{x,r}$.  
\item We have $p\mathfrak{g}_{x,r} = \mathfrak{g}_{x,r + 1}$.
\end{itemize}

Suppose $r > 0$ is such that $G_{x,r}$ is torsion-free.  By \cite[Cor. (1)]{serre:cohdim} and \cite[Thm. V.2.5.8]{lazard}, $G_{x,r}$ is a Poincar\'e group of dimension $d$.  In particular, we have $H^d(G_{x,r},\bbF_p) \cong \bbF_p$ (see \emph{op. cit.,} \S V.2.5.7), and Theorem \ref{mainthm:appendix} is trivially true for $p = 2$.  Therefore, we may assume that $p$ is odd in what follows.

\begin{lemma}\label{moy-prasad-isom}
Suppose $r \geq 1$.  Then there exists an isomorphism
\[ G_{x,r}/G_{x,r + 1} \stackrel{\sim}{\longrightarrow} \mathfrak{g}_{x,r}/\mathfrak{g}_{x,r + 1} = \mathfrak{g}_{x,r}\otimes_{\bbZ_p} \bbF_p. \]
Moreover, if $g\in G$ satisfies $g\cdot x = x$, then this isomorphism is equivariant for the adjoint action of $g$.
\end{lemma}

\begin{proof}
The claimed isomorphism is known as the \emph{Moy--Prasad isomorphism}
and is well-known when $\mathbf{G}$ splits over a tamely ramified extension
\cite[Cor.~2.4]{yu:sc}.
When $\mathbf{G}$ is wild, the Moy--Prasad filtration
has many deficiencies \cite[\S\S~ 0.4, 4.6]{yu:models}.
Nonetheless, something can still be said in general.
The theory of dilatation,
a close relative of the Moy--Prasad filtration,
constructs a congruence filtration~$(\mathbf{H}_n)_{n\geq0}$
on any flat affine group scheme~$\mathbf{H}$ of finite type over~$\mathbb{Z}_p$.
The dilatation $\mathbf{H}_n$ is again
a flat, affine, finite-type $\mathbb{Z}_p$-group scheme
and $\mathbf{H}_n(\mathbb{Z}_p) = \ker(\mathbf{H}(\mathbb{Z}_p) \longrightarrow \mathbf{H}(\bbZ/p^n\bbZ))$.
In this setting one can prove that when $\mathbf{H}$ is in addition
smooth and with connected generic fiber,
there is a functorial isomorphism of $\bbZ/p^n\bbZ$-group schemes
\[
\mathbf{H}_{n,\bbZ/p^n\bbZ}
\stackrel{\sim}{\longrightarrow}
\boldsymbol{\mathfrak{h}}_{n,\bbZ/p^n\bbZ},
\]
where $\boldsymbol{\mathfrak{h}}_n$ is the Lie algebra of $\mathbf{H}_n$
and the subscript $\bbZ/p^n\bbZ$ denotes base change.
Kaletha and Prasad explain this general result
on dilatation in \cite[Prop.~A.5.19]{kalethaprasad}
and use it to prove an
even stronger version of the claimed isomorphism
(\cite[Thm.~13.5.1]{kalethaprasad}).
\end{proof}

\begin{corollary}\label{h1-cor}
Suppose $p$ is odd, $r \geq 1$, and $G_{x,r}$ is torsion-free.  We then have
\[H^1(G_{x,r},\bbF_p) \cong \Hom_{\bbF_p}(\mathfrak{g}_{x,r}/\mathfrak{g}_{x,r + 1},\bbF_p) = (\mathfrak{g}_{x,r}/\mathfrak{g}_{x,r + 1})^\vee.\]
\end{corollary}

\begin{proof}
Since $G_{x,r}$ acts trivially on $\bbF_p$, we have 
\[H^1(G_{x,r},\bbF_p) = \Hom^{\textnormal{cts}}(G_{x,r},\bbF_p) = \Hom\big((G_{x,r})_{\Phi},\bbF_p\big),\]
where $(G_{x,r})_{\Phi} := G_{x,r}/\overline{G_{x,r}^p[G_{x,r}, G_{x,r}]}$ denotes the Frattini quotient.  The group $\mathfrak{g}_{x,r}/\mathfrak{g}_{x,r + 1}$ is abelian and $p$-torsion, and the above lemma implies that the same is true for $G_{x,r}/G_{x,r + 1}$.  The universal property of the Frattini quotient then gives a surjective map
\[(G_{x,r})_{\Phi} \longtwoheadrightarrow G_{x,r}/G_{x,r + 1} \stackrel{\sim}{\longrightarrow} \mathfrak{g}_{x,r}/\mathfrak{g}_{x,r + 1}.\]
By dualizing, we obtain
\begin{equation}\label{frattini-inj}
\Hom_{\bbF_p}(\mathfrak{g}_{x,r}/\mathfrak{g}_{x,r + 1},\bbF_p) \longhookrightarrow \Hom\big((G_{x,r})_{\Phi},\bbF_p\big) = H^1(G_{x,r},\bbF_p),
\end{equation}
which implies $\dim_{\bbF_p}(H^1(G_{x,r},\bbF_p)) \geq \dim_{\bbF_p}(\mathfrak{g}_{x,r}/\mathfrak{g}_{x,r + 1}) = \dim_{\bbQ_p}(G) = d$.  On the other hand, \cite[Prop. 1.2]{klopsch} implies that $\dim_{\bbF_p}(H^1(G_{x,r},\bbF_p)) \leq d$.  Thus, the injection \eqref{frattini-inj} is an isomorphism.
\end{proof}

\begin{corollary}\label{uniform-cor}
\hfill
\begin{enumerate}
\item Suppose $p \geq 5$, $r \geq 1$, and $G_{x,r}$ is torsion-free.  Then $G_{x,r}$ is a uniform pro-$p$ group.  
\item Suppose $p = 3$, $r \gg 0$ is sufficiently large, and $G_{x,r}$ is torsion-free.  Then $G_{x,r}$ is a uniform pro-$3$ group.
\end{enumerate}
\end{corollary}

\begin{proof}
For $p \geq 5$, this follows from the previous corollary and \cite[Prop. 1.10, Rmk. 1.11]{klopschsnopce}.  Assume $p = 3$.  By \cite[Cor. 4.3]{DDMS}, the group $G_{x,0+} = \bigcup_{s > 0}G_{x,s}$ possesses an open normal subgroup $H$ which is uniform, and in particular saturable.  Since the Moy--Prasad subgroups give a neighborhood basis of the identity, there exists $r_0 > 0$ such that $G_{x,r_0} \subseteq H$.  Then, for any $r \geq \max\{r_0,1\}$ we have $G_{x,r} \subseteq H$, and the result follows from the previous corollary and \cite[Cor. 1.7]{klopschsnopce}.
\end{proof}

\subsection{Proof}

We now prove Theorem \ref{mainthm:appendix}, given the reductions in the previous section.  Thus, we assume that $p$ is odd, and $\mathbf{G}$ is defined over $\bbQ_p$.  We fix a point $x$ in the semisimple Bruhat--Tits building and a real number $r > 0$, and assume $G_{x,r}$ is torsion-free.  Finally, we suppose $g\in G$ satisfies $g\cdot x = x$.

By Corollary \ref{uniform-cor}, there exists a real number $r'\geq \max\{r, 1\}$ for which the group $G_{x,r'}$ is uniform.  By \cite[pf. of Prop. 30]{serre:galoiscoh}, the corestriction map
\[\textnormal{cores}^{G_{x,r'}}_{G_{x,r}}:H^d(G_{x,r'},\bbF_p) \longrightarrow H^d(G_{x,r},\bbF_p)\]
is an isomorphism, which is furthermore equivariant for the adjoint action of $g$.  Therefore, replacing $G_{x,r}$ by $G_{x,r'}$, we may further assume that $G_{x,r}$ is uniform.

Since $G_{x,r}$ is uniform, \cite[Thm. 5.1.5]{symondsweigel} and Corollary \ref{h1-cor} imply that we have $\Ad(g)$-equivariant isomorphisms of $\bbF_p$-vector spaces
\[H^d(G_{x,r},\bbF_p) \cong \sideset{_{}^{}}{_{\bbF_p}^d}{\bigwedge} H^1(G_{x,r},\bbF_p) \cong \sideset{_{}^{}}{_{\bbF_p}^d}{\bigwedge} (\mathfrak{g}_{x,r}/\mathfrak{g}_{x,r + 1})^\vee.\]
Therefore, by dualizing, it suffices to prove that $\Ad(g)$ acts trivially on
\[\sideset{_{}^{}}{_{\bbF_p}^d}{\bigwedge} \mathfrak{g}_{x,r}/\mathfrak{g}_{x,r + 1} \cong \sideset{_{}^{}}{_{\bbF_p}^d}{\bigwedge} \big(\mathfrak{g}_{x,r} \otimes_{\bbZ_p} \bbF_p\big) \cong \left(\sideset{_{}^{}}{_{\bbZ_p}^d}{\bigwedge} \mathfrak{g}_{x,r}\right)\otimes_{\bbZ_p} \bbF_p.\]
Using the above isomorphisms, it suffices to show that $\bigwedge^d\textnormal{Ad}(g)$ acts trivially on $\bigwedge_{\bbZ_p}^d \mathfrak{g}_{x,r}$.  The latter is a free $\bbZ_p$-module of rank $1$ since $\mathfrak{g}_{x,r}$ is a free $\bbZ_p$-module of rank $d$, and we are reduced to showing that $\det_{\bbZ_p}(\Ad(g)) = 1$ for the adjoint action of $g$ on $\mathfrak{g}_{x,r}$ (note that the determinant is defined as $\mathfrak{g}_{x,r}$ is free of finite rank).

In order to prove that the determinant of the map $\Ad(g): \mathfrak{g}_{x,r} \longrightarrow \mathfrak{g}_{x,r}$ is equal to 1, it suffices to do so after tensoring by $\overline{\mathbb{Q}}_p$:
\[\Ad(g): \boldsymbol{\mathfrak{g}} = \mathfrak{g}_{x,r}\otimes_{\bbZ_p}\overline{\mathbb{Q}}_p \longrightarrow \boldsymbol{\mathfrak{g}} = \mathfrak{g}_{x,r}\otimes_{\bbZ_p}\overline{\mathbb{Q}}_p.\]
Recall that the group $\mathbf{G}$ has an almost-direct product decomposition $\mathbf{G} = \mathbf{Z}\mathbf{G}^{\textnormal{der}}$, where $\mathbf{Z}$ denotes the connected center of $\mathbf{G}$ and $\mathbf{G}^{\textnormal{der}}$ denotes the derived subgroup (see \cite[Cor. 8.1.6(i)]{springer}).  Let us write $g = zg'$ with $z\in \mathbf{Z}$ and $g' \in \mathbf{G}^{\textnormal{der}}$.  As linear maps on $\boldsymbol{\mathfrak{g}}$, we have 
\[\Ad(g) = \Ad(z) \circ \Ad(g') = \Ad(g'),\]
since $z$ is central.  Therefore, we get
\[\sideset{_{}^{}}{_{\overline{\mathbb{Q}}_p}^{}}\det(\Ad(g)) = \sideset{_{}^{}}{_{\overline{\mathbb{Q}}_p}^{}}\det(\Ad(g')) = 1,\]
since $\mathbf{G}^{\textnormal{der}}$ is generated by commutators.  This concludes the proof of Theorem \ref{mainthm:appendix}.

\subsection{Orientation character}

We once again make no assumption on $p$, and suppose that the pro-$p$ Iwahori subgroup $I$ of $G$ is torsion-free.  We define the mod $p$ orientation character $\xi:G \longrightarrow \bbF_p^\times$ as follows.  Let $\bbF_p$ denote the trivial $G$-representation.  For $g \in G$, we set $I_g := I \cap gIg^{-1}$, and let $\xi(g)$ denote the scalar defined by the sequence of isomorphisms
\[\bbF_p \xrightarrow{\textnormal{tr}^{-1}_{I_{g^{-1}}}} H^d(I_{g^{-1}},\bbF_p) \stackrel{g_*}{\longrightarrow} H^d(I_{g},\bbF_p) \xrightarrow{\textnormal{tr}_{I_g}} \bbF_p.\]
Here $\textnormal{tr}_{K}:H^d(K,\bbF_p) \stackrel{\sim}{\longrightarrow} \bbF_p$ is the trace map, defined in \cite[\S I.4.5, Prop. 30(b)]{serre:galoiscoh} for any torsion-free open pro-$p$ subgroup of $G$.  It is known that the association $g \longmapsto \xi(g)$ gives a homomorphism, which is trivial on $I$ (\cite[Lem. 2.6]{schneidersorensen} and \cite[Lem. 4.11, Def. 4.12]{koziol:functorial}).

The following result was first proved by Schneider--Sorensen (\cite[Lem. 2.10]{schneidersorensen}, see also \cite[Cor. 5.2]{kohlhaase}).

\begin{theorem}
The character $\xi: G \longrightarrow \bbF_p^\times$ is trivial.
\end{theorem}

\begin{proof}
We fix an apartment of the semisimple Bruhat--Tits building which contains $C$, the chamber corresponding to $I$, and let $\mathbf{S}$ denote the associated maximal $\mathfrak{F}$-split torus.  Recall that the group $G$ has a Bruhat factorization
\[G = I N I,\]
where $N$ denotes the group of $\mathfrak{F}$-points of the normalizer of $\mathbf{S}$ (see \cite[Prop. 3.35]{vigneras:hecke1}).  Since $\xi$ is trivial on $I$, it suffices to prove that it is also trivial on $N$.

The group $N$ acts by simplicial automorphisms on the apartment associated to $\mathbf{S}$, and we define $N_{C}$ to be the subgroup of $N$ which stabilizes $C$.  By equations (53) of \cite{vigneras:hecke1}, the group $N$ is generated by $N_{C}$ and representatives of affine reflections in the walls of $C$.  It therefore suffices to prove that $\xi$ is trivial on $N_C$ and on each such affine reflection.

Suppose first that $g \in N_C$, so that $g$ normalizes $I$ and $I_g = I_{g^{-1}} = I$.  Let $x$ denote the barycenter of $C$, so that $g \cdot x = x$ and $I = G_{x,r}$ for some sufficiently small $r > 0$.  By Theorem \ref{mainthm:appendix}, we have that 
\[g_* = \Ad(g): \textnormal{H}^d(I,\bbF_p) \longrightarrow \textnormal{H}^d(I,\bbF_p)\]
is the identity map.  By definition of $\xi$, we get $\xi(g) = 1$.

Suppose now that $g$ is a representative of an affine reflection in a wall of $C$.  Let $x \in \overline{C}$ denote a point on this wall, so that $g\cdot x = x$.  Since $I_g \cap I_{g^{-1}}$ is open in $G$, we may choose $r > 0$ such that $G_{x,r} \subseteq I_g \cap I_{g^{-1}}$.  By canonicity of the trace map (cf. \cite[Lem. 4.9]{koziol:functorial}) and functoriality of the corestriction map, the following diagram is commutative:
\begin{center}
\begin{tikzcd}[column sep = huge, row sep = huge]
\bbF_p   \ar[r,"\textnormal{tr}^{-1}_{I_{g^{-1}}}"]  \ar[dr, "\textnormal{tr}^{-1}_{G_{x,r}}"']&  H^d(I_{g^{-1}},\bbF_p) \ar[r, "g_*"] & H^d(I_{g},\bbF_p) \ar[r,"\textnormal{tr}_{I_{g}}"] & \bbF_p \\
 &  H^d(G_{x,r},\bbF_p) \ar[u, "\textnormal{cores}^{G_{x,r}}_{I_{g^{-1}}}"'] \ar[r, "\Ad(g)"] & H^d(G_{x,r},\bbF_p) \ar[u, "\textnormal{cores}^{G_{x,r}}_{I_{g}}"] \ar[ur, "\textnormal{tr}_{G_{x,r}}"'] & 
\end{tikzcd}
\end{center}
(Recall that by \cite[pf. of Prop. 30]{serre:galoiscoh}, the displayed corestriction maps are isomorphisms.)  By Theorem \ref{mainthm:appendix}, the bottom horizontal arrow is the identity map, which implies $\xi(g) = 1$.  
\end{proof}

\bibliographystyle{amsalpha}
\bibliography{refs}

\end{document}